\documentclass[a4paper,10pt]{article}

\usepackage{geometry}
\usepackage{amsmath,amsfonts,amsthm,amssymb}
\usepackage{xcolor,mathrsfs,stmaryrd}
\usepackage{url}
\usepackage{bbold}

\usepackage{bera}

\usepackage[normalem]{ulem}

\usepackage{graphicx}
\usepackage{wrapfig}
\usepackage{enumerate}

\usepackage{hyperref}

\makeatletter
\patchcmd{\SK@@ref}{\footnotesize}{\tiny}{}{}
\patchcmd{\SK@@ref}{\underbar}{}{}{}
\patchcmd{\SK@@ref}{\vrule}{}{}{}
\makeatother

\numberwithin{equation}{section}

\newtheorem{theo}{Theorem}\numberwithin{theo}{section}
\newtheorem{coro}[theo]{Corollary}

\newtheorem{prop}[theo]{Proposition}
\newtheorem{lemm}[theo]{Lemma}

\newtheorem{defn}[theo]{Definition}

\newtheorem{rema}[theo]{Remark}

\def\tP{{\hat{\P}}}
\def\A{{\cal A}}
\def\tA{{\hat{\A}}}
\def\tz{{\hat{z}}}

\def\by{{\mathbf{y}}}
\def\bx{{\mathbf{x}}}

\def\sl{<}
\def\sg{>}

\newcommand{\PP}{\mathbb{P}}
\newcommand{\XX}{\mathbb{X}}
\newcommand{\MM}{\mathbb{M}}
\newcommand{\YY}{\mathbb{Y}}
\newcommand{\WW}{\mathbb{W}}
\newcommand{\QQ}{\mathbb{QQ}}

\def\N{\mathbb{N}}

\def\X{\mathbb{X}} 
\def\S{\mathbb{S}}
\def\0{{\bf 0}}
\def\QQ{\mathbb{Q}}

\newcommand{\E}{\mathbb E \,}
\def\P{{\cal P}}

\newcommand{\M}{{\cal M}}

\newcommand{\G}{{\cal G}}
\def\R{\mathbb{R}}

\def\la{{\lambda}}

\newcommand{\ind}[1]{\one(#1)}

\newcommand{\Var}{{\rm Var}}

\newcommand{\Cyl}{{\rm Cyl}}

\newcommand{\hyp}{\mathcal{H}^d}
\newcommand{\DD}{D^{(2)}}

\def\bdm{\begin{displaymath}}
\newcommand{\edm}{\end{displaymath}}
\def\benu{\begin{enumerate}}
\def\eenu{\end{enumerate}}
\def\beqn{\begin{equation}}
\def\eeqn{\end{equation}}
\def\be{\begin{equation}}
\def\ee{\end{equation}}
\def\bea{\begin{eqnarray}}
\def\eea{\end{eqnarray}}
\newcommand{\bean}{\begin{eqnarray*}}
\newcommand{\eean}{\end{eqnarray*}}
\newcommand{\bear}{\begin{eqnarray}}
\newcommand{\eear}{\end{eqnarray}}

\def\one{{\bf{1}}}

\def\qed{\hfill\hbox{${\vcenter{\vbox{
    \hrule height 0.4pt\hbox{\vrule width 0.4pt height 6pt
    \kern5pt\vrule width 0.4pt}\hrule height 0.4pt}}}$}}

\usepackage{titlesec}

\titleformat*{\section}{\normalfont\large\bfseries}
\titleformat*{\subsection}{\normalfont\bfseries}

\date{\vspace{-0.95cm}}

\usepackage[backend=biber]{biblatex} 
\begin{document}

\title{Second-order Poincaré inequalities and localization on the Poisson space }

\author{Tara Trauthwein \footnotemark[1] \ \ and \ J. E. Yukich\footnotemark[2]}

\maketitle

\footnotetext[1]{Institute of Mathematical Stochastics, University of Münster, Germany}
\footnotetext[2]{Department of Mathematics, Lehigh University, Bethlehem, PA, 18015, USA}

\begin{abstract}
Given a mean zero functional $F$ of   a Poisson measure on a metric space, we apply the Malliavin-Stein method to  establish 
sharpened second-order Poincar\'e inequalities for 
  $F/\sqrt{\Var (F)}$ in terms of fourth  moments of difference operators.  The  rates of normal approximation are expressed in the Kolmogorov and Wasserstein distances and require
fewer error terms than corresponding  previous results.     
When $F$ is expressible as a sum of score functions which 
are distributionally close to scores having  short-range structure, then we deduce that $F/\sqrt{\Var(F)}$  satisfies Berry-Esseen bounds.
The normal approximation criteria of the scores, here called bounded Lipschitz localization, are more general than stabilization criteria and allow for unbounded interactions of scores.  This approach yields Berry-Esseen bounds for local U-statistics on metric measures spaces, localizing functionals on hyperbolic space, as well as for Poisson functionals in a space-time setting, with infinite time horizon, including statistics of  spatial birth-growth models and Laguerre tessellations.
\end{abstract}



\section{Introduction}
Consider a Poisson point process  $\mathcal{P}$ on a measure space $(\X, \nu)$, with $\nu$ a $\sigma$-finite measure.  Let $F = f(\P)$ be a measurable function of $\P$.  Under what conditions on $F$ can one obtain good proximity bounds between $F$ and a standard normal $N$?  More precisely, if ${\bf{d}}$ is metric on the space of probability laws, when does one obtain the second-order Poincar\'e inequality 
\be \label{Berry}
{\bf{d}}\left( \frac{F - \E F} {\sqrt{\Var F}}, N\right) = O\left( \frac{1} {\sqrt{\Var F}}\right) \ ?
\ee
The Berry-Esseen bound \eqref{Berry} is  unimprovable in general and represents an optimal quantitative central limit theorem for Poisson functionals. 

This question, and more generally the study of qualitative and quantitative 
normal convergence of functionals of point processes have a long history.  Bickel and Breiman \cite{BB} in 1983 established the normal convergence of the total edge length of nearest neighbor graphs on i.i.d. binomial input in $\R^d$ and noted:  `Our proof is long.  We believe this is due to the complexity of the problem.'  Kesten and Lee \cite{KL}
in 1996, when proving that the total edge length of the minimal spanning tree on Poisson input on growing  cubes in $\R^d$ converged to the normal wrote: `Another drawback of our approach is that it is not quantitative. Further ideas are needed to obtain an
error estimate in our central limit theorem'. 
In the case that $(\X, \nu)$ is  Euclidean space, $\nu$ Lebesgue measure, and $F = f(\P)$ is representable as a sum of scores depending a.s. on local data, the work \cite{PY05} developed 
proximity bounds in the Kolmogorov distance between  
$\hat{F}:= (F - \E F)/\sqrt{\Var F}$ and $N$, but these bounds are suboptimal. 

It was not until $2016$ that a systematic investigation of proximity bounds between 
$\hat{F}$ and $N$ was undertaken with the aid of the Malliavin-Stein method.  The pioneering paper of Last, Peccati and Schulte \cite{LPS}  established the remarkably versatile second-order Poincar\'e inequality 
\be \label{LPSgeneral}
{\mathbf{d}_K}\left( \frac{F - \E F} {\sqrt{\Var F}}, N\right) \leq
\sum_{i = 1}^6 \gamma_i
\ee
where the 
formulae for $\gamma_i, 1 \leq  i \leq 6,$ involve integrals of products of difference operators, e.g., 
$$
\gamma_1= \left[\int_{ \mathbb{X}^3}  [\E(D_x F)^2 (D_y F)^2]^{1/2}  [\E(D_{x, y} F)^2 (D_{y, z} F)^2]^{1/2}\mu^3(dx,dy,dz)\right]^{1/2},
$$
with similar geometry based formulae for $\gamma_i, 2 \leq  i \leq 6.$  Thus the Kolmogorov distance ${\mathbf{d}_K}$ is controlled by integrals 
of products of moments of the first order difference operators
 $D_xF = f(\P + \delta_x) - f(\P)$ as well as of the second order difference operators $\DD_{x,y}F= D_xD_yF$. To quote \cite{LPS}, the error of normal approximation has a `significantly more complex structure' than the Poincar\'e inequalities of \cite{Ch09}
 and \cite{NPR}. 

 This raises the question:  In view of \eqref{Berry} are there relatively simple verifiable criteria under which the errors $\gamma_i, 1 \leq  i \leq 6,$ in \eqref{LPSgeneral},  collapse into terms which are $O( \frac{1} {\sqrt{\Var F}})$?  Secondly, does the simplicity of the criteria reveal
 quantitative central limit theorems for Poisson functionals  previously not known to satisfy asymptotic normality? 

We answer both questions in the affirmative when $(\X, \nu)$ is a metric measure space  and show that if  the Poisson functional $
F$ is a sum of scores which are distributionally well-approximated by short-range scores, then $F$ satisfies 
the Berry-Esseen bound \eqref{Berry}. 
The rates are 
in general unimprovable and hold for 
Poisson functionals  on  general metric spaces,  including hyperbolic space, and including functionals of 
Poisson input which may have an infinite time component, as well as for scores 
which do not necessarily depend  on  local data. This extends  Penrose and Yukich \cite{PY05}, Lachi\`{e}ze-Rey et al.
\cite{LSY},  Lachi\`{e}ze-Rey \cite{LR},
Lachi\`{e}ze-Rey et al. \cite{LRPY} and
Bhattacharjee and Molchanov \cite{BM}; most of these works were either confined to Euclidean space and/or assumed that the scores satisfied strong localization conditions, without allowing for an infinite time component. 

Our approach consists of two parts, the first of which extends
\cite{LPS} and \cite{TT}.  Given a Poisson functional $F$, we apply the Malliavin-Stein method to  obtain rates of normal convergence of  $\hat{F}$ in terms of fourth  moments of difference operators, and requiring fewer error terms than the corresponding results in \cite{LPS, TT}.  This  strictly improves the corresponding rate results in \cite{LPS,TT}, as the error terms $\gamma_3,\gamma_4$ in \cite{LPS} and $\gamma_7^{(2)}$ in \cite{TT} are  shown to be superfluous for the Kolmogorov distance.  
Removing these terms is satisfying, as e.g. $\gamma_4$  
requires control of the fourth moment of $F$, often a source of technical difficulty.
The improved second order Poincar\'e bounds established here, given in terms of the Kolmogorov and Wasserstein distances, are the cleanest and shortest  second order Poincaré inequalities currently available for general Poisson functionals, when working with fourth  moment conditions on the difference operators. The bounds are stated in full generality in Theorem \ref{thm:LPSMarked} and include functionals of {\em marked} Poisson point processes.

In a second part, we apply the general Theorem \ref{thm:LPSMarked} to Poisson functionals  
expressible as a sum of scores 
$$
H := \sum_{z \in \P\cap W} \xi((z,M_z), \tP), 
$$
where $\xi$ are measurable functions,  $W \subset \X$ is a window with finite  measure, $\P$ is a  Poisson measure on $W$,  $\{M_z\}_{z \in \P}$ are i.i.d. marks, and $\tP:= \{(z,M_z)\}_{z \in \P}$ is the Poisson  process $\P$ equipped with independent marks.

Our localization criterion consists of comparing, for each $r > 0$, the distribution of $\xi$ with the distribution of a suitably chosen {\em short-range score}  $\xi^{[r]}$, where for any locally finite $\chi$ the score {\em $\xi^{[r]}((z,M_z), \chi)$ depends  only on the input in the radius $r$ ball around $z \in \XX$}, namely on  $\chi \cap (B_r(z) \times \MM)$.
We often choose  $\xi^{[r]}$  to be the restriction of $\xi$ to a radius $r$ ball, i.e.,
$\xi^{[r]}((w,M_w),\chi)=\xi((w,M_w),\chi \cap (B_r(z) \times \MM))$, but the approach here allows for more general choices of $\xi^{[r]}$.
We show that as soon as $\xi$ and the family $(\xi^{[r]})_{r > 0}$ satisfy a fifth moment condition as well as either
$$
 \sup_{z \in \XX} \sup_{ \mathcal A   } \PP( \xi((z,M_z), \P \cup \mathcal A) \neq \xi^{[r]}((z,M_z), \P \cup \mathcal A)) \leq  \psi(r), \quad r > 0
$$
or the slightly weaker condition
$$
 \sup_{z \in \XX}
\sup_{ \mathcal A   }
||\xi((z,M_z), \P \cup \mathcal A) - \xi^{[r]}((z,M_z), \P \cup \mathcal A)||_1   \leq  \psi(r),
\quad r > 0,
$$ 
where the sup runs over sets $\mathcal A \subset \XX$  of cardinality at most $6$, and
where $\psi$ satisfies a mild integrability condition on $(\X, \nu)$, {\em then the Berry-Esseen bound \eqref{LPSgeneral} holds}, i.e.
\be \label{TYBerry}
{\bf{d}}\left( \frac{H - \E H} {\sqrt{\Var H}}, N\right) = O\left( \frac{1} {\sqrt{\Var H}}\right),
\ee
where ${\bf{d}}$ is either the Kolmogorov or Wasserstein distance and where we assume
$\Var H = \Omega(\nu(W))$.
More generally 
as soon as four-tuples of scores 
$(\xi(z_1, \P),...,\xi(z_4, \P))$ 
have enough short-range structure, in the sense that they are  distributionally close to
four-tuples of short-range scores
$(\xi^{[r]}(z_1, \P),...,\xi^{[r]}(z_4, \P))$, uniformly over all four-tuples, as measured by the bounded Lipschitz distance, then the short-range interactions dominate, which is to say, heuristically speaking, that $H$ behaves like a sum of independent random variables 
with a distribution which is well-approximated by the normal as at
\eqref{TYBerry}.

This geometric localization, called {\em bounded Lipschitz localization} as introduced in \cite{BYY24}, is  weaker and more flexible than the standard stopping set stabilization criterion in \cite{LSY}, \cite{PY05}.
Though it requires $H$ to have enough distributional short-range structure over all possible ranges, as measured by $\psi$, it also allows for interactions of scores at distant points.

Our general quantitative CLT for
 $\hat{H} :=  (H - \E H)/\sqrt{\Var H}$ holds
for Poisson input on general metric measure spaces, in contrast to \cite{LR,LRPY,LSY, PY05,SY23}. The general main results given by Theorem~\ref{mainthm1} and Corollaries ~\ref{cor:spaceonly} and ~\ref{mainthmXisW}
provide rates of normal convergence in the Kolmogorov and  Wasserstein distances for local U-statistics on metric measure spaces and localizing functionals on $d$-dimensional hyperbolic space, including statistics of 
the random geometric graph and the $k$-nearest neighbor graphs in these spaces.  
 It also yields precise
quantitative  central limit theorems for Poisson functionals where previously only qualitative results were available. This includes statistics  
of $\R^d$-valued diffusions on graphs on Poisson random input on $\R^d$ \cite{BYY}.

We also consider Poisson functionals in a {\em space-time setting}, with $\X$ enlarged to $\X \times \R$, giving rise to sums 
$$
H := \sum_{(z,t_z) \in \P \cap (W \times \R)} \xi((z,t_z, M_z), \tP). 
$$
Such statistics arise in general space-time growth models with infinite time horizon, including 
statistics of 
classic birth-growth models and Laguerre tessellations.  In such models the score on 
an infinite time interval is well approximated by input on a large but finite time interval.
This phenomenon is quantified through `time-localization', a companion concept to `space-localization'.
We use Theorem \ref{thm:LPSMarked} to deduce rates of normal convergence for $\hat{H} :=  (H - \E H)/\sqrt{\Var H}$; see 
Theorem \ref{mainthm1}.  This quantitative central limit theorem yields presumably optimal rates of normal convergence for the afore-mentioned statistics.
 
 This paper focuses on rates of convergence in the univariate central limit theorem but the methods and approach can also be used to achieve rates of convergence in the multivariate central limit theorem for vectors of Poisson functionals, with each entry expressible as a sum of  scores satisfying bounded Lipschitz localization, cf. Remark~\eqref{rem:multi} in Section \ref{remarks}.

\vskip.3cm

\section{Main results}
In this section we present our main results in detail.

\smallskip

\noindent\textbf{Underlying space.}
We let $(\X, \nu)$ be a $\sigma$-finite measure space, with $\nu$ a $\sigma$-finite measure on the $\sigma$-algebra 
$\mathcal{F}$.  

\smallskip
\noindent\textbf{Marked Poisson point processes.}
Throughout, we let $\mathcal{P}$ be a Poisson process on $\X$ with intensity measure $\nu$.
In order to treat marked point processes, let $(\MM,\mathcal{F}_{\MM},\QQ)$ be a probability space. By $\MM$ we mean the space of marks whereas  $\QQ$ is the underlying probability measure of the marks. Let $\widehat{\XX}:= \XX\times \MM$. By $\widehat{\mathcal{F}}$ we denote the product $\sigma$-field of $\mathcal{F}$ and $\mathcal{F}_\MM$ and by $\widehat{\nu}$ the product measure of $\nu$ and $\QQ$. 
To obtain an unmarked point process, one may consider 
 the case where $(\MM,\mathcal{F}_{\MM},\QQ)$ is a singleton endowed with a Dirac point mass, and the `hat' superscript can be removed in all occurrences.
For any given point $x\in\XX$, we denote by $M_x$ the corresponding random mark, which is distributed according to $\QQ$ and is independent of everything else.

Let $\mathbf{N}_{\widehat{\X}}$ be the set of $\sigma$-finite counting measures on $\widehat{\XX}$, which can be interpreted as point configurations in $\widehat{\XX}$. Whenever it is clear from context, we drop the subscript and simply write $\mathbf{N}$. The set $\mathbf{N}$ is equipped with the smallest $\sigma$-field $\mathcal{N}$ such that the maps $m_A: \mathbf{N}\to \N\cup\{0,\infty\}, \mathcal{M}\mapsto \mathcal{M}(A)$ are measurable for all $A\in\widehat{\mathcal{F}}$. A point process is a random element in $\mathbf{N}$. 

\noindent\textbf{Malliavin Derivative.} Given $F$ a measurable function of $\hat{\mathcal{P}}$ and $(x,m_x), (y,m_y) \in \widehat{\XX}$, we let
$$
D_{(x,m_x)} F(\hat{\mathcal{P}}) = F(\hat{\mathcal{P}} \cup \{(x,m_x)\}) -
F(\hat{\mathcal{P}})
$$
and 
$$
\DD_{(x,m_x), (y,m_y) } F(\hat{\mathcal{P}}) = 
D_{(x,m_x)} D_{(y,m_y)} F(\hat{\mathcal{P}})
$$
be the first and second order difference operators applied to $F$.

\subsection{Sharpened Malliavin-Stein bounds on the marked Poisson space}
In this section we present significantly improved second-order Poincaré inequalities, which simplify and extend some of the work in \cite{LPS} and \cite{TT}.  The main result are the normal approximation bounds in Theorem \ref{thm:LPSMarked},  
which both sharpen and extend the $p = 2$ version of Theorems 
3.3 and 3.4 of \cite{TT},  which themselves improve upon Theorems 
1.1 and 1.2 in \cite{LPS}.  In Sections $3$ and $4$ we  apply Theorem \ref{thm:LPSMarked} to
Poisson functionals which may be represented as 
sums of BL-localizing score functions, but we
emphasize that this quantitative CLT applies to general functionals 
on the marked Poisson space, including those which may not
be conveniently expressed as a sum of BL-localizing scores.  Even for functionals on the unmarked Poisson space, this sharper result contains fewer approximating terms than the corresponding results in \cite{TT} and \cite{LPS}.

The Kolmogorov distance between the laws of random variables $X$ and $Y$ is
$$
\mathbf{d}_K(X,Y) = \sup_{t \in \R} |\PP(X \leq t) - \PP(Y \leq t)|
$$
whereas the Wasserstein distance is
$$
\mathbf{d}_W(X,Y) = \sup_{h \in {\rm{Lip}}(1)} |
\E h(X) - \E h(Y)|,
$$
where ${\rm{Lip}}(1)$ denotes the class of Lipschitz functions on $\R$ with Lipschitz constant one.
Throughout $N$  denotes a standard Gaussian random variable.

\begin{theo}\label{thm:LPSMarked} (sharpened Malliavin-Stein bounds for Poisson functionals)
    Let $(\X,\nu)$ be a $\sigma$-finite measure space and $(\MM,\QQ)$ a space of marks with probability measure $\QQ$. Let $\hat{\mathcal{P}}$ be a Poisson measure on $\X \times \MM$ of intensity $\nu \otimes \QQ$. Let $F$ be a measurable function of $\hat{\mathcal{P}}$ such that $\E F^2 < \infty$ and $\E F = 0$ and
    \[
    \int_{\X \times \MM} \E \left[\left( D_{(x,m_x)} F \right)^2\right] (\nu \otimes \QQ)(dx,dm_x) < \infty.
    \]
    For any point $x\in \X$, define a random variable $M_x$ independent of everything else, distributed according to $\QQ$.
    Then it holds that
    \[
    \mathbf{d}_W(F,N) \leq \sqrt{\tfrac{2}{\pi}}\hat{\gamma}_0+ \sqrt{\tfrac{2}{\pi}}\hat{\gamma}_1 + \tfrac{1}{\sqrt{2\pi}}\hat{\gamma}_2 + \hat{\gamma}_3
    \]
    and
    \[
    \mathbf{d}_K(F,N) \leq \hat{\gamma}_0 + \hat{\gamma}_1 + \tfrac{1}{2} \hat{\gamma}_2 + \hat{\gamma}_4 + \hat{\gamma}_5 + \hat{\gamma}_6,
    \]
    where
    \begin{align*}
        \hat{\gamma}_0 &:= \E \left| 1- \Var(F) \right|\\
        \hat{\gamma}_1 &:= 2 \bigg(\int_\X \bigg(\int_\X \E \big[ |D_{(y,M_y)}F|^4\big]^{1/4} \cdot \E \big[|\DD_{(x,M_x),(y,M_y)}F|^4\big]^{1/4} \nu(dy) \bigg)^2 \nu(dx) \bigg)^{1/2}\\
        \hat{\gamma}_2 &:= 2  \bigg(\int_\X \bigg(\int_\X \E \big[|\DD_{(x,M_x),(y,M_y)}F|^4\big]^{1/2} \nu(dy) \bigg)^2 \nu(dx) \bigg)^{1/2}\\
        \hat{\gamma}_3 &:= 2 \int_\X \E|D_{(y,M_y)}F|^3 \nu(dy)
    \end{align*}
    and
    \begin{align*}
        \hat{\gamma}_4 &:= \bigg(4\int_\X \E|D_{(y,M_y)}F|^4 \nu(dy)\bigg)^{1/2}\\
        \hat{\gamma}_5 &:= \bigg(8\int_\X\int_\X \E|\DD_{(x,M_x),(y,M_y)}F|^4 \nu(dy)\nu(dx)\bigg)^{1/2}\\
        \hat{\gamma}_6 &:= \bigg(32\int_\X\int_\X \E\big[|\DD_{(x,M_x),(y,M_y)}F|^4\big]^{1/2} \cdot \E\big[|D_{(x,M_x)}F|^4\big]^{1/2} \nu(dy)\nu(dx)\bigg)^{1/2}.
    \end{align*}
\end{theo}

\begin{rema}\label{rem:LPSasUsual}
    We do not assume $\Var(F)=1$ and instead introduce the term $\hat{\gamma}_0$. This is a standard extension and can be useful in cases where the functional $F$ is not normalized. Note that one easily derives that
    \[
    d_{W}\left(\frac{F - \E F}{\sqrt{\Var(F)}}, N \right) \leq \sqrt{\tfrac{2}{\pi}} \Var(F)^{-1} \hat{\gamma}_1 + \tfrac{1}{\sqrt{2\pi}} \Var(F)^{-1} \hat{\gamma}_2 + \Var(F)^{-3/2} \hat{\gamma_3}
    \]
    and
    \[
    d_{K}\left(\frac{F - \E F}{\sqrt{\Var(F)}}, N\right) \leq \Var(F)^{-1} \big(\hat{\gamma}_1 + \tfrac{1}{2} \hat{\gamma}_2 + \hat{\gamma}_4 + \hat{\gamma}_5 + \hat{\gamma}_6 \big)
    \]
    by linearity of the operators $D,\DD$.
\end{rema}

In the following, we compare Theorem \ref{thm:LPSMarked} with existing general results, all of which are less sharp and less general whenever the difference operators satisfy fourth  moment assumptions.

\begin{enumerate}
    \item \textbf{Removal of complicated terms in \cite{LPS}.} In \cite[Theorem 1.2]{LPS}, the bound on the Kolmogorov distance is given by the sum of terms $\gamma_1,...,\gamma_6$. We remove two terms from the bound on the Kolmogorov distance: $\gamma_3$ and $\gamma_4$. The term $\gamma_3$ (corresponding to our $\hat{\gamma}_3$) only contains the first Malliavin derivative to a third power - this differs from the other terms and can in some cases lead to worse bounds. It is a (presumably) necessary term for the Wasserstein distance, but can be removed in the Kolmogorov case. The term $\gamma_4$ in \cite{LPS} has a structure that differs strongly from the other terms: it contains the term $\E F^4$, which is often difficult to bound and complicates matters significantly. Methods for removing some of the terms from the Kolmogorov bound first appeared in \cite{LRPY}, based on ideas from \cite{SZ19}. The result in \cite{LRPY} required however strong additional assumptions on the functional $F$. In \cite{TT}, these additional assumptions were removed and a shortened bound was shown without any additional starting assumptions. We proceed in a similar way and show that  $\gamma_3$ and $\gamma_4$ are unnecessary in the Kolmogorov second-order Poincaré inequality.

   \item \textbf{Removal of term $\gamma_7$ in \cite{TT}.}  
     The development in \cite{TT} focused on achieving a bound under a lower $(2+\epsilon)$-moment assumption on $F$, and while the terms $\gamma_3,\gamma_4$ from \cite{LPS} were removed in the Kolmogorov distance, an additional term $\gamma_7$ appeared in the bound as an artifact of the more general proof. While \cite[term $\gamma_7$]{TT} was much more in line with the remaining terms and thus easier to treat, we stress that when working with fourth moment assumptions on the difference operators, this term is superfluous and no such term is necessary in Theorem~\ref{thm:LPSMarked}.
    \item \textbf{Extension to marked spaces.} In \cite{LPS}, Theorem~1.1 gives a bound for the Wasserstein distance for functionals $F$ of generic spaces $\X$. It is possible to introduce marks by choosing $\XX:=\XX' \times \MM$ to be a marked space, but the resulting terms $\gamma_1,...,\gamma_3$ will be different, i.e. the integrals over the marked spaces will be grouped with integrals over the space $\XX'$, which is often unfavorable. In our result, by a careful rewriting of the proof, we achieve a setting where each point is equipped with an independent mark when assessing the add-one costs, i.e. integrals over the mark space $\MM$ are grouped with the expectation.  A similar adaptation of the arguments of \cite{LPS} to marked input is achieved in \cite{LSY}.
\end{enumerate}

\subsection{BL-localization in space and time}\label{Sec2.3}

Considerable attention has been given to systematically establishing rates of normal convergence for functionals of  Poisson input on general spatial domains.  The rates of normal convergence, expressed with respect to either the Kolmogorov, Wasserstein, $\ell_2$ or $\ell_3$ distances, all rely on Stein's method;  see e.g. \cite{BX, PY05, LPS, LSY, LR, LRPY, SY19, SY23, TT}. 
The rates of convergence in \cite{LPS, TT}, like those of Theorem \ref{thm:LPSMarked}, are valid for Poisson input on general metric spaces, though evaluation of the terms
$\gamma_i, i = 1,...,6,$ may be challenging.

When the Poisson functional may be expressed as a sum of score functions satisfying either a stopping set stabilization criterion as in \cite{LSY} or an $L^4$ stabilization criterion as in \cite{LR}, then evaluation of the terms
$\hat{\gamma}_i, i = 1,...,6,$ is facilitated and leads to rates of convergence which are unimprovable in general.  Up to now, these results hold in Euclidean space and not a general metric space \cite{PY05,LSY,LR,SY19,SY23}; moreover neither the stopping set  stabilization nor the 
$L^4$ stabilization criterion is satisfied by scores arising in certain statistics of  random point sets in Euclidean space, including e.g. interacting diffusions on sparse graphs in $\R^d$.  When the Poisson functional is a sum of score functions satisfying a weak spatial dependency condition here termed \textit{BL-localization} and which is weaker than existing stabilization criteria, then one may systematically simplify  the terms $\hat{\gamma}_i, i = 1,...,6,$ in Theorem \ref{thm:LPSMarked}, thereby establishing rates of convergence in the Kolmogorov and Wasserstein distances.  The results are valid in an arbitrary metric measure space and the rates are optimal in general.
This is achieved by establishing bounds on the moments of the difference operators and showing that these moment bounds  control the order of growth of the integrals in the expressions for $\hat{\gamma}_i, 1 \leq i \leq 6,$ in Theorem \ref{thm:LPSMarked}. Establishing BL-localization of scores is often  more transparent than finding good bounds for moments of difference operators by ad hoc methods.

In addition to considering  score functions defined purely in space, we widen our scope and  consider sums of scores on Poissonian space-time input on  $X \times \R$.  We shall be interested in summing scores on points in the unbounded space-time domains $W \times \R$, that is to say we consider input with possibly unbounded time component.  We treat such statistics subject to the scores satisfying \textit{localization in time} as well as \textit{localization in space}.  Roughly speaking, localization in time says that the difference between scores at finite time $t_0$ and infinite time becomes negligible as $t_0$ increases. 
This condition, apparently new, is a useful companion to space-localization and it manifests in statistics  of  
generalized birth-growth models \cite{SY08} and  generalized Delaunay tessellations obtained by consideration of the projection of a parabolic hull process in space-time $\R^d \times \R$ onto $\R^d$ as in \cite{GKT}. These statistics are related to statistics of  points retained in a parabolic thinning of a space-time Poisson point processes in a half-space, which arises in solutions to the inviscid
Burgers' equation \cite{AMS, Ba} with random initial data.  These applications are developed in Section \ref{Sectionspacetime}. 

\vskip.3cm

\noindent\textbf{Setup.} Let $(X,d,\nu)$ be a $\sigma$-finite metric measure space and $W \subseteq X$ a subset with $\nu(W)<\infty$. We let $(\MM,\QQ)$ be a probability space of marks and denote by $M_z$ the independent mark associated to $z \in X$, distributed according to $\QQ$. For any set $\A \subset X$, we denote by $\hat{\A}$ the random set containing the points $(z,M_z)$, where $z \in \A$. We take $\mu$ to be a $\sigma$-finite measure on $\R$ and equip $\R$ with the Euclidean metric.

The set $\mathbf{N}=\mathbf{N}_{X \times \R \times \MM}$ denotes the set of locally finite point measures on $X \times \R \times \MM$. By $\P$ we mean a Poisson measure on $X \times \R$ of intensity $\nu \otimes \mu$. The set $\hat{\P}$ denotes the points $(z,t_{z},M_z)$ with $(z,t_{z}) \in \mathcal{P}$; $\hat{\mathcal{P}}$ has the law of a Poisson process on $X \times \R \times  \MM$ with intensity measure $\nu \otimes \mu \otimes \mathbb{Q}$.

We consider scores 
\[
\xi: (X \times \R \times \MM) \times \mathbf{N} \rightarrow \R
\]
assumed to be measurable functions from  $(X \times \R \times \MM) \times \mathbf{N}$ to $\R$. The scores give rise to the Poisson functionals
\[
H:= H(\hat{\mathcal{P}}) = \sum_{(z,t_z) \in \mathcal{P} \cap (W \times \R)}  \xi((z,t_z,M_z),\hat{\P}).
\]
The sets $X \times \R$ control the  space-time domains carrying the (unmarked) input of $\xi$, whereas $W \times \R$ is the  window over which the scores are summed. Two natural cases arise:
\begin{itemize}
    \item $W \equiv X$, i.e. the point process is constrained to a window and we sum scores over all points in $W$;
    \item $W \subsetneq X$, i.e. the score function $\xi$ takes as input a larger, possibly infinite (in space) point process $\hat{\P}$, but we only sum scores on points inside the window $W$.
\end{itemize}
These are discussed in more detail in Subsection \ref{usersguide}. Additionally, it is natural to include scores which do not depend on time. This is treated in  Corollary \ref{cor:spaceonly}.

Denote by $B_r(x)=B(x,r)$ the open ball of radius $r$ around the point $x \in X$ with respect to the metric $d$. Given a subset $X_0 \subset X$, we define $d(x,X_0):= \inf_{z \in X_0} d(x,z)$. We shall always write $\xi((z,t_z,M_z),\chi)$ for $\xi((z,t_z,M_z),\chi \cup \{(z,t_z,m_z)\})$, where $\chi \in \mathbf{N}$ and $(z,t_z,m_z) \in X \times \R \times \MM$.

\vskip.3cm

\noindent\textbf{$BL$-localization of scores in space-time.}
The definition below is crucial to our development, as it defines the notions of \textit{$BL$-localization in space} and \textit{$BL$-localization in time}. To derive central limit theorems, it suffices that four-tuples of scores are distributionally well-approximated by four-tuples of short-range scores, as measured by the bounded Lipschitz distance. A short-range score in space is one which depends only on data in a ball with fixed radius. A short-range score in time depends only on data up to a certain time, and is zero after this time.

For $p\in\N$, denote generic points in 
$\R^p$ by $\bx := (x_1,\hdots,x_p)$. 
By $BL(\R^p)$ we mean the  class of Lipschitz($1$) functions $f : \R^p \to \R$
with supremum norm bounded by $1$, that is  
$$
\sup_{\bx, \by \in \R^p} \frac{ |f(\bx) - f(\by)|}{||\bx - \by||} \leq 1, \quad \sup_{\bx\in \R^p} |f(\bx) | \leq 1,
$$
where $\|\bx - \by\|$ is the Euclidean distance on $\R^p$.

Given $\R^p$-valued random vectors ${\bf X}:= (X_1,...,X_p)$ and ${\bf Y}:= (Y_1,...,Y_p)$, defined on possibly different probability spaces, the 
{\em bounded Lipschitz distance} between the laws of  ${\bf X}$ and ${\bf Y}$ is
$$
d_{BL}({\cal L}({\bf X}), {\cal L}({\bf Y}))= \sup_{f \in BL(\R^p)} | \E f({\bf X}) - \E f({\bf Y}))|.
$$

$BL$-localization establishes closeness in the $d_{BL}$ metric between the laws of the random variables $\xi$ and their short-range versions $\xi^{[r]}$ and $\xi^{(s)}$. This motivates the term `localization' in contrast to `stabilization', a notion comparing  specific realizations of the scores - either exactly (via stopping sets) or approximately in $L^q$. 

\begin{defn}\label{def:space-time-loc}
($BL$-localization of scores in space and in time) Given $\theta >0$, and $\theta' \in (0,1/2]$ the score $\xi$ is said to  $BL(\theta, \theta')$\textbf{-localize on the space-time domain $(X,d,\nu) \times (\R,\mu)$} if there exist families of scores $(\xi^{[r]})_{r>0}$ and $(\xi^{(s)})_{s \in \R}$ and non-increasing functions $\psi :[0,\infty) \rightarrow [0,2]$, with $\psi(0)=2$, and $\phi :\R \rightarrow [0,2]$ such that
\begin{itemize}
    \item   for any locally finite $\chi \in  \mathbf{N}$, and any $z \in X$, $t_z \in \R$ we have for any $r>0$,
    \begin{align}
        &\xi^{[r]}((z,t_{z},M_z),\chi) = \xi^{[r]}((z,t_{z},M_z),\chi \cap (B_r(x) \times \R \times \MM)) \qquad \QQ-\text{a.s.} \label{stoppingspace}\\
        \intertext{and for any $s \in \R$}
        &\xi^{(s)}((z,t_{z},M_z),\chi) = \ind{t_{z}<s}\xi^{(s)}((z,t_{z},M_z),\chi \cap \XX \times (-\infty,s) \times \MM) \qquad \QQ-\text{a.s.} \label{stoppingtime};
    \end{align}
        \item \sloppy for  any $z_1,...,z_4,x,y \in X$, and $t_{z_1},...,t_{z_4},t_{x},t_y \in \R$, and any $\mathcal{A}_i \subseteq \{(z_1,t_{z_1}),...,(z_4,t_{z_4}),(x,t_{x}),(y,t_y)\}$ for $i=1,\hdots,4$, we have for all  $r > 0$
    \begin{equation} \label{dBLspace}
    d_{BL} \bigg( \big(\xi((z_i,t_{z_i},M_{z_i}),\hat{\mathcal{P}} \cup \hat{\mathcal{A}}_i)\big)_{i=1,...,4},\
    \big(\xi^{[r]}((z_i,t_{z_i},M_{z_i}),\hat{\mathcal{P}} \cup \hat{\mathcal{A}}_i)\big)_{i=1,...,4} \bigg) \leq \psi(r),
    \end{equation}
    and for all $s \in \R$
    \begin{equation}\label{dBLtime}
    d_{BL} \bigg( \big(\xi((z_i,t_{z_i},M_{z_i}),\hat{\mathcal{P}} \cup \hat{\mathcal{A}}_i)\big)_{i=1,...,4},\
    \big(\xi^{(s)}((z_i,t_{z_i},M_{z_i}),\hat{\mathcal{P}} \cup \hat{\mathcal{A}}_i)\big)_{i=1,...,4} \bigg) \leq \phi(s),
    \end{equation}
    \item we have
    \be \label{integralpsi}
    {\cal I}_{\psi}(\theta) := {\cal I}_{\psi,X}(\theta):= \max \big(1, 
     \sup_{x \in X} \int_{X} \psi \left( \frac{d(x,z)}{2} \right)^\theta \nu(dz)\big) < \infty;
    \ee
    and 
    \be \label{integralphi}
     {\cal I}_{\phi}(\theta') := \max \big(1, \int_\R \phi(t) ^{\theta'} \mu(dt)\big) < \infty.
    \ee
\end{itemize}
The collections $(\xi^{[r]})_{0<r<\infty}$ and
$(\xi^{(s)})_{s \in \R}$,
satisfying \eqref{stoppingspace} and \eqref{stoppingtime}, respectively, are said to be families of {\em short-range scores in  space} and {\em short-range scores in  time}, respectively.
Lastly we put for all $p > 0$
\begin{align}
        M^\xi_{p,X} :=  \max\bigg\{ & {1}, \sup_{r \in (0,\infty]} \sup_{\substack{(z,t_z) \in X \times \R \\ \mathcal{A} \subset X \times \R, |\mathcal{A}|\leq p+1}} \E |\xi^{[r]}((z,t_z,M_z),\hat{\P} \cup \hat{\mathcal{A}})|^p,\notag\\
    &\sup_{s \in \R}
   \sup_{\substack{(z,t_z) \in X \times \R \\ \mathcal{A} \subset X \times \R, |\mathcal{A}|\leq p+1}} \E |\xi^{(s)}((z,t_z,M_z),\hat{\P} \cup \hat{\mathcal{A}})|^p \bigg\}, \label{eq:moment}
\end{align}
where $\xi^{[\infty]}=\xi$.
\end{defn} 

The $BL$-localization definition may appear complex, but as we shall see shortly in Subsection \ref{3conditions}, its verification reduces to checking some relatively straightforward conditions on a single score.
The bound $d_{BL}(\mathcal{L}(\mathbf{X}),\mathcal{L}(\mathbf{Y})) \leq 2$ always holds and thus the requirement that $\psi$ and $\phi$ are bounded by $2$ is not a constraint. Likewise, imposing $\psi(0)=2$ is convenient for the purpose of presentation, but of no consequence in applications. Note that $BL(\theta_1)$ localization implies $BL(\theta_2)$ localization in space or time, whenever $\theta_1 < \theta_2$.

In practice, we often choose $\xi^{[r]}$ and $\xi^{(s)}$ to be the restrictions of $\xi$ given as
\begin{equation}\label{Restricted2}
\xi((z,t_{z},m_z),\chi \cap (B_r(z) \times \R \times \MM)) \quad \text{and} \quad \ind{t_{z}<s}\xi((z,t_{z},m_z),\chi \cap (\XX \times (-\infty,s) \times \MM))
\end{equation}
respectively. However in some cases we will choose $\xi^{[r]}$ to be dependent on the geometry of the balls  $B_r(\cdot)$. This flexibility in the choice of $\xi^{[r]}$ widens the scope of applications. 

\begin{rema} \label{Hisfinite} $BL(\theta, \frac{1}{2})$-space-time localization and $M^\xi_{2,X}<\infty$ imply that $H$ is a.s. finite. Indeed, we have
\begin{align*}
    \E |H| &\leq \E \sum_{(z,t_z) \in \P \cap W \times \R} |\xi(z,t_z,M_z),\tP)| \\
    &\stackrel{\text{Mecke}}{=} \int_{W \times \R} \E |\xi((z,t_z,M_z),\tP)| (\nu \otimes \mu)(dz,dt_z)\\
    &= \int_{W \times \R} \left(\E |\xi((z,t_z,M_z),\tP)| - \E |\xi^{(t_z)}((z,t_z,M_z),\tP)|\right) (\nu \otimes \mu)(dz,dt_z)\\
    &\leq 48 (M^\xi_{2,X})^{1/2}  \nu(W) \int_\R \phi(t_z)^{1/2} \mu(dt_z) < \infty,
\end{align*}
where the inequality
\[
\left(\E |\xi((z,t_z,M_z),\tP)| - \E |\xi^{(t_z)}((z,t_z,M_z),\tP)|\right) \leq 52 (M^\xi_{2,X})^{1/2} \phi(t_z)^{1/2}
\]
follows from Lemma \ref{A5thmoment} below applied to the localizing pair $|\xi|,|\xi^{(t_x)}|$
and with $q = 1$ there.
\end{rema}

\begin{rema}
Localization in time is satisfied in two natural cases. The first is whenever the probability of the score being non-zero decays fast with respect to the magnitude of the time component.  This is the case in space-time models where points  in the distant  future are unlikely to contribute to the statistic $H$, as in
spatial birth-growth models and generalized beta-Delaunay tessellations.
The second natural case arises when the scores $\xi$ do not depend on any time coordinate. In this case, one can choose $\phi \equiv 2$ and $\mu = \delta_0$, the Dirac mass at $0$. See Corollary~\ref{cor:spaceonly} for details.
\end{rema}

\begin{rema}
    Condition~\eqref{integralphi} is in fact a condition on $\mu((-\infty,0))$ as well as on the decay of $\phi$ at $+\infty$. Indeed, since $\phi$ is non-decreasing, we have $\phi(t)\geq \phi(0)$ for every $t<0$. It  follows from \eqref{integralphi} that $\mu((-\infty,0))<\infty$. Heuristically , this means that $\mu$ is a finite measure on the negative half-line and hence not too many  points of the point process will be situated here, while there can be many points on the positive half-line but their importance and contribution to the total sum decreases with $\phi$ fast enough to ensure finiteness. Alternatively, it is possible to choose $\mu$ such that $\mu((0,+\infty))<\infty$, in which case the choice of $\phi$ is irrelevant and not too many points will be situated on the positive half-line. Their importance however may be large.
\end{rema}

\subsection{The main results}
Our main foundational result for the Poisson functionals
$H:=H(\hat{\P})$ 
goes as follows.  It is a consequence of Theorem \ref{thm:LPSMarked}. 

\begin{theo} \label{mainthm1} (rates of normal convergence for sums of $BL$-space-time localizing scores on metric spaces) 
Let $\xi:(X \times \R \times \MM) \times \mathbf{N}_{X \times \R \times \MM} \rightarrow \R$ be a score function on the marked space domain $X \times \R \times \MM$ and $\P$ a Poisson process on $X \times \R$ with intensity $\nu \otimes \mu $ and $\hat{\P}$ its marked version. Define
\[
H:=H(\hat{\P}) = \sum_{(z,t_z) \in \P \cap (W \times \R)} \xi((z,t_z, M_z),\hat{\P}).
\]
Assume that $\xi$ satisfies  $BL\big(\tfrac{1}{240}, \tfrac{1}{120}\big)$ space-time localization according to Def.~\ref{def:space-time-loc}, with $M^\xi_{5,X}<\infty$.
For $\psi$ as in Def.~\ref{def:space-time-loc} and $q>0$, put 
\[
G_{q}:=\int_{X} \psi(d(x,W))^{q} \nu(dx) = 2^q \nu(W) + \int_{X\setminus W} \psi(d(x,W))^{q} \nu(dx).
\]
If $G_{1/120}<\infty$, then there is a 
$C_K<\infty$ given by
\[
C_K:= c\cdot  {\cal I}_{\psi,X}\big(\tfrac{1}{240}\big)^{3} \cdot {\cal I}_{\phi}\big(\tfrac{1}{120}\big)^{7/2},
\]
where $c<\infty$ is a universal scalar, such that
\begin{equation} \label{eq:dkMain}
    \mathbf{d}_K \left(  \frac{H - \E H} {\sqrt{\Var H}}, N \right) \leq C_K \cdot (M_{5,X}^\xi)^{2/5} \cdot \frac{G_{1/120}^{1/2}}{\Var H}
\end{equation}
and there is a $C_W <\infty$ given by
\[
C_W := c \cdot  {\cal I}_{\psi,X}(\tfrac{1}{50})^{3} \cdot {\cal I}_{\phi}(\tfrac{3}{200})^{4}
\]
such that
\begin{equation} \label{eq:dwMain}
    \mathbf{d}_W \left(  \frac{H - \E H} {\sqrt{\Var H}}, N \right) \leq C_K \cdot (M_{5,X}^\xi)^{2/5} \cdot \frac{G_{1/120}^{1/2}}{\Var H} + C_W  \cdot (M_{5,X}^\xi)^{3/5} \cdot\frac{G_{3/200}}{(\Var H)^{3/2}}.
\end{equation}
\end{theo}

In many applications  scores do not depend on input in a time variable and hence we give a version of Theorem~\ref{mainthm1} which looks at scores defined on the space $X$ only.

\begin{coro}\label{cor:spaceonly}(rates of normal approximation for sums of $BL$-space-localizing scores) Let $\xi:(X \times \MM) \times \mathbf{N}_{X \times \MM} \rightarrow \R$ be a score function on the marked space domain $X \times \MM$ and $\P$ a Poisson process on $X$ with intensity $\nu$ and $\hat{\P}$ its marked version. Define
\[
H:=H(\hat{\P}) = \sum_{z \in \P \cap W} \xi((z,M_z),\hat{\P}).
\]
Assume that $\xi$ satisfies $BL(\frac{1}{240})$\textbf{-space localization}, i.e., assume $\xi$ satisfies conditions \eqref{dBLspace} and \eqref{integralpsi} with $\theta= \frac{1}{240}$ (where any input $t_.$ in time is ignored), and with the short-range scores $\xi^{[r]}$ satisfying \eqref{stoppingspace}.  Define $M^\xi_{5,X}$ as 
\[
M^\xi_{5,X}:=  \max\big(1, \sup_{r \in (0,\infty]}
   \sup_{\substack{z \in X \\ \mathcal{A} \subset X, |\mathcal{A}|\leq 6}} \E |\xi^{[r]}((z,M_z),\hat{\P} \cup \hat{\mathcal{A}})|^5\big).
\]
Assume $M^\xi_{5,X}<\infty$ and $G_{1/120}<\infty.$ 
Then the bounds \eqref{eq:dkMain} and \eqref{eq:dwMain} hold with
\begin{equation} \label{eq:CKCWspace}
C_K := c \cdot \mathcal{I}_{\psi,X}(\tfrac{1}{240})^3 \quad \text{and} \quad C_W := c \cdot \mathcal{I}_{\psi,X}(\tfrac{1}{50})^3.
\end{equation}
\end{coro}
\begin{proof}
    Extend the definition of $\xi$ and $\xi^{[r]}$ to $X \times \R \times \MM$ in the natural way by simply ignoring the input in time. The moment bound $M^\xi_{5,X}$ is not affected by this extension.
    Choose $\mu:=\delta_0$, the Dirac mass at $0$ and $\xi^{(s)}\equiv 0$. Then \eqref{stoppingtime} holds and \eqref{dBLtime} holds with $\phi \equiv 2$. Display \eqref{integralpsi} becomes
    \[
    \mathcal{I}_{\phi}(\theta') = \phi(0)^{\theta'} = 2^{\theta'},
    \]
    for any $\theta'$. The bounds on $C_K$ and $C_W$ follow.
\end{proof}

In many cases the values of $C_K, C_W$ as well as the moments $M^{\xi}_{5,X}$ are  scalar pre-factors and, unlike the quantities
$G_q$ and $\Var H$, do not grow with $\nu(W)$, and hence do contribute to the order of growth in the approximation error, whence  Berry-Esseen bounds follow immediately.   However the generality of our approach allows for the case when 
$C_K, C_W, M^{\xi}_{5,X}$ are dependent on an external parameter which may for example control the score $\xi$, the intensity of the Poisson measure, or even the measure of the underlying space.   In such instances  the Berry-Esseen bounds may be modified by additional factors which are  possibly dependent on $\nu(W)$.

The next result will be put to good use in the sequel. We state it for scores defined on $X \times \R \times \MM$ (space-time) and as well as those defined on $X \times \R$ (space) only.  The corresponding quantities $C_W,C_K$ and $M^\xi_{5,X}$ are to be taken from Theorem~\ref{mainthm1} and Corollary~\ref{cor:spaceonly} respectively.

\begin{coro}  \label{mainthmXisW} (Berry-Esseen bounds) 
If $X=W$  and $\xi$ satisfies  $BL\big(\frac{1}{240},\frac{1}{120}\big)$ space-time localization (resp. $BL(\frac{1}{240})$ space-localization if $\xi$ is independent of time), then 
\[
\mathbf{d}_K \left(  \frac{H - \E H} {\sqrt{\Var H}}, N \right) \leq C_K \cdot (M^{\xi}_{5,W})^{2/5} \frac{\nu(W)^{1/2}}{\Var H}
\]
and 
\[
\mathbf{d}_W \left(  \frac{H - \E H} {\sqrt{\Var H}}, N \right) \leq C_K \cdot (M^{\xi}_{5,W})^{2/5} \frac{\nu(W)^{1/2}}{\Var H} + C_W \cdot (M^{\xi}_{5,W})^{3/5} \frac{\nu(W)}{(\Var H)^{3/2}},
\]
where $C_K,C_W$ and $M^\xi_{5,X}$ are as in Theorem~\ref{mainthm1} (resp. as in Corollary~\ref{cor:spaceonly}).

If $\Var H \geq c_0 \nu(W)$ for some positive constant $c_0$ then
there is a $C>0$ depending on $c_0,C_K,C_W,M^\xi_{5,X}$
such that for 
${\bf{d}} \in \{\mathbf{d}_W, \mathbf{d}_K\}$ 
\begin{equation} \label{BEB}
{\bf{d}}\left(  \frac{H - \E H} {\sqrt{\Var H}}, N \right) \leq \frac{ C}{\sqrt {\Var H}}. 
\end{equation} 
\end{coro}
\begin{proof}
    When $X=W$, then $G_q = 2^q\nu(W)$, 
      which immediately yields the bound.
\end{proof}

\vskip.3cm
The values of $\theta$ and $\theta'$ in the integrals appearing in  $\mathcal{I}_{\psi,X}(\theta)$ and $\mathcal{I}_\phi(\theta')$ are artifacts of the proof. In most applications, the exact values of $\theta$ and $\theta'$ are not relevant, as the functions $\psi,\phi$ decay faster than 
any power and $C_K$ and $C_W$ are simply scalars.

For scores on Euclidean space, it is often the case that  $\psi(r)$ decays exponentially fast in $r$ and hence  integrability  \eqref{integralpsi} is easily satisfied.
More concretely, if $X=\R^d$ and $\nu$ is Lebesgue measure, then we may let $\psi(r) = \exp(-a_1r^{a_2})$ for any $a_1, a_2 >0$ or we may take $\psi$ to decay as a large power, namely $\psi(r)=\min(1,r^{-\alpha})$, with $\alpha>d/\theta$.

If $X$ is hyperbolic space $\mathbb{H}^d$ with curvature $-1$, with $\nu$ the uniform measure, and $d$ the intrinsic metric, then we may take $\psi(r)=\exp(-\tau(r))$, provided that  $\tau$ has  faster than linear growth (see \eqref{eq:hypPsiCond} for details).

\begin{rema}(A word on the proof)
    At first sight it is not clear that  BL-localization should   suffice for controlling the fourth moments under the integrals in Theorem~\ref{thm:LPSMarked}.  However when applying the difference operators $D_x$ and $D_{x,y}$ to $H$ we are led to controlling expectations of products of  difference operators applied to scores $\xi$ at the respective entries of four-tuples $(z_1,...,z_4) \in X^4$.  We  portray  products of second order difference operators  on four such scores as an ostensible difference with an analogous product of difference operators on corresponding short-range scores $\xi^{[r]}$, provided $r$ is chosen to be the maximum  distance between the pair $(x,y)$ and the four-tuple $(z_1,...,z_4)$. 
    The said difference is only in appearance, as the difference operator on the analogous product of short-range scores vanishes. By writing products of second-order difference operators in this manner, we may  use 
    BL-localization \eqref{dBLspace} to control their expectations with an error proportional to $\psi(r)$,  and thus by $\psi(d(x,y)/2)$, as necessarily $r \geq d(x,y)/2$. Hence  integrated  moments of products of second-order difference operators  may be controlled by the integrals of $\psi$ as at \eqref{integralpsi}. This, together with similar bounds on first-order difference operators,  leads to efficient  control of all integrals  in the error terms of Theorem~\ref{thm:LPSMarked}, whence Theorem~\ref{mainthm1}. 
\end{rema}

\subsection{How to use the main theorem } \label{usersguide}

Theorem \ref{mainthm1} and its two corollaries include several set-ups which apply to broad classes of Poisson functionals.

\begin{enumerate}[(i)]
    \item (dilating windows) 
    One is given a sequence of windows $(W_\la)_{\la \geq 1}$ and  spaces $(X_\la)_{\la \geq 1}$, $W_\la \subseteq X_\la$, which are embedded into 
    a common metric measure space $(\XX, d, \nu)$, with $W_\la \uparrow \XX$. With $\P$ a Poisson point process on $\X$ with intensity measure $\nu$ and with $\P_\la$ the restriction of  $\P$ to $X_\la$ we put 
    \be \label{HlageneralA}
    H_\la:= H_\la(\hat{\mathcal{P}}) = \sum_{z \in \mathcal{P} \cap W_\la}  \xi((z,t_z,M_z),\hat{\P}_\la),
    \ee
    which is a sum of scores on dilating windows. 
    Then Theorem
    \ref{mainthm1}
    goes through verbatim, with $H$ replaced by $H_\la$,
    with $X$ replaced by $\XX$, with $M^\xi_{p,X}$ replaced by $M^\xi_{p,\XX}$, and
    with $I_{\psi,X}$ replaced by $I_{\psi,\XX}$. The constants $C_K, C_W$ are thus independent of $\la$.     
    When $\XX = \R^d$ and  
    $W_\la := [- \frac{1}{2}\la^{1/d}, \frac{1}{2}\la^{1/d}]^d$ we obtain a set-up 
    treated extensively in the literature, though  under the stronger assumption that $\xi$ satisfies stopping set stabilization. 
    Here we only assume localization and do not restrict to Euclidean space. 
    Several  applications shall make use of this set-up.

    \item \label{UG:manifold}
    (point processes of increasing intensity)
    We consider the normal approximation for  sums of scores on a fixed measure space
    but with point processes of increasing intensity. 
    
    Let $(X, d_\la, \nu_\la)_{\la \geq 1}$ be a sequence of metric measure spaces and   $\P_\la$ a Poisson point process on $X$ with intensity measure $\nu_\la$ where possibly  $\nu_\la(X) \uparrow \infty$ as $\la \to \infty$.   Let $W \subseteq X$.  Put 
    \be \label{Hlageneral}
    H_\la:= H_\la(\hat{\mathcal{P}_\la}) = \sum_{z \in \mathcal{P}_\la \cap W}  \xi_\la((z,t_z,M_z),\hat{\P}_\la).
    \ee
where $\xi_\la:(X \times \R \times \MM) \times \mathbf{N}_{X \times \R \times \MM} \rightarrow \R$. Consider the analog of \eqref{integralpsi}:
    \be \label{integralpsiAGEN}
         {\cal I'}_{\psi}(\theta):= 
         \max \big(1, \sup_{\la \geq 1} \sup_{x \in X} \int_{X} \psi \left( \frac{d_\la(x,z)}{2} \right)^\theta \nu_\la(dz) \big)  < \infty.
        \ee 
    Given $\theta >0$, and $\theta' \in (0,1/2]$ the scores $(\xi_\la)_{\la \geq 1}$ are  said to  $BL(\theta, \theta')$\em{-localize  on their respective space-time domains $((X,d_\la,\nu_\la) \times (\R,\mu))_{\la \geq 1}$} if for each $\la$ 
    there exist short-range functions $\xi_\la
    ^{[r]}, \xi_\la^{(s)}$ such that
    \eqref{stoppingspace}- \eqref{dBLtime} hold uniformly in $\la$ and with ${\cal I'}_{\psi}(\theta)$ replacing ${\cal I}_{\psi}(\theta)$.  
    There are two natural cases of interest.

    \begin{enumerate}[(a)]
 
    \item \label{UG:compact}(increasing intensity on a compact subset of $\R^d$)
    Consider the heavily studied case $W \equiv X = [- \frac{1}{2}, \frac{1}{2}]^d$ with $X$ equipped with the Euclidean metric  and $\P_\la$  a Poisson measure with intensity $\la\rho(x)dx$ on $X$, with $dx$ Lebesgue  measure and $\rho: X \to \R^+.$
      Berry-Esseen bounds for the sums
     $\sum_{x \in \P_\la} \xi(x, \P_\la)$,
    were established in 
     \cite{LSY} under the assumption that the scores satisfy stopping set stabilization.  Theorem
    \ref{mainthm1} yields comparable bounds for the sums $H_\la$ at \eqref{HlageneralA} under the weaker assumption that the scores satisfy localization \eqref{dBLspace}, where the integrability condition \eqref{integralpsiAGEN}  given by
     \be 
    {\cal I'}_{\psi}(\theta):= 
     \max \big(1,
    \sup_{\la \in [1, \infty)} \sup_{x \in [- \frac{1}{2}, \frac{1}{2}]^d}  \int_{[- \frac{1}{2}, \frac{1}{2}]^d} \psi \left( \frac{\la^{1/d}|x-z|)}{2} \right)^\theta \la \rho(z) dz \big) < \infty
     \ee 
     is satisfied  as soon as 
     \be 
      \sup_{x \in \R^d}  \int_{\R^d} \psi \left( \frac{|x-z|)}{2} \right)^\theta  
      ||\rho||_{\sup} dz < \infty.
     \ee 

    \item (increasing intensity on a manifold) Let  $W \equiv X \equiv {\cal M}$ with ${\cal M}$ a $d$-dimensional Riemannian manifold with Riemannian metric $d_g$ and where $\P_\la$ is a Poisson measure with intensity $\la\rho(x)dx$ on ${\cal M}$, with $dx$ the Riemannian volume measure. Theorem
    \ref{mainthm1}, with $X = {\cal M}$  and ${\cal M}$ equipped with the metric $\la^{1/d}d_g$, yields
      Berry-Esseen bounds for the sums
     $\sum_{x \in \P_\la} \xi(x, \P_\la)$ 
       whenever  the  scores satisfy localization \eqref{dBLspace}, subject to  the integrability condition \eqref{integralpsiAGEN}, which takes the form 
     \be \label{integralpsimanifold}
    \max \big(1, \sup_{\la \in [1, \infty)} \sup_{x \in {\cal M}}  \int_{{\cal M}} \psi \left( \frac{\la^{1/d}d_g(x,z)}{2} \right)^\theta \la \rho(z) dz \big)  < \infty.
     \ee 
     When ${\cal M}$ is a $C^1$ smooth manifold then 
     rates of normal convergence with spurious logarithmic factors were obtained in \cite{PY13} albeit under the relatively strong assumption that $\xi$ belongs to the class $\Sigma(k,r_0)$, i.e., the collections of scores having a stopping set whose radius is
     the maximum of a fixed constant 
     $r_0$ and the distance to the $k$th nearest neighbor, and also under the assumption that the metric is the extrinsic Euclidean metric. 
      Here we obtain Berry-Esseen rates 
      for the sums $H_\la$ at \eqref{HlageneralA}
      without logarithmic factors, assuming  only that $\xi$ satisfies localisation \eqref{integralpsimanifold}, and where distances are  with respect to  
      the metric intrinsic to 
      ${\cal M}$, which is arguably more natural.

      \end{enumerate}

    \item \label{XisRd}\sloppy(scores defined on the full space $X = \R^d$)
    The bounds for Theorem~\ref{mainthm1} simplify when $X = \R^d$, $d$ is the Euclidean metric,  $\nu$ is Lebesgue measure,  $W:= W_\la := [- \frac{1}{2}\la^{1/d}, \frac{1}{2}\la^{1/d}]^d$, and $\P$ is a  Poisson point process on $\R^d \times \R$ with intensity $\nu \otimes \mu$.  This gives rise to a sum of scores
    $$
    H_\la^{\infty}:=  \sum_{(z,t_z) \in \mathcal{P}\cap W_\la \times \R }  \xi((z,t_z,M_z),\hat{\P}).
    $$
    If  $\xi$ satisfies $BL(\frac{1}{240},\frac{1}{120})$-space-time-localization (Def.~\ref{def:space-time-loc}) and $M^\xi_{5,\R^d}$ is finite, and if in addition
    \be \label{finitepolar}
    \int_0^{\infty}\psi(r)^{1/240} r^{d-1}dr < \infty,
    \ee
        then we assert that  
    \[
    \mathbf{d}_K \left(\frac{H_\la^{\infty} - \E H_\la^{\infty}} {\sqrt{\Var H_\la^{\infty}}}, N \right) \leq C_K (M^{\xi}_{5,\R^d})^{2/5} \frac{\sqrt{\la}}{\Var H_\la}
    \]
    and
    \[
    \mathbf{d}_W \left(\frac{H_\la^{\infty} - \E H_\la^{\infty}} {\sqrt{\Var H_\la^{\infty}}}, N \right) \leq C_K (M^{\xi}_{5,\R^d})^{2/5} \frac{\sqrt{\la}}{\Var H_\la} + c C_W \cdot (M^{\xi}_{5,\R^d})^{3/5}\frac{\la}{(\Var H_\la)^{3/2}},
    \]
    where $C_K=c \mathcal{I}_\phi(\frac{1}{120})^{7/2}$ and $C_W = c\mathcal{I}_\phi(\frac{3}{200})^4$.

    To prove this assertion, note that  $G_{q} = 
     2^q \lambda + 
    \int_{\R^d \setminus W_\la} \psi(d(x, W_\la))^{q} dx$.  It suffices to show that the integral is $O(\la^{(d-1)/d})$ when $q = 1/120$ or $q = 3/200$.    Let ${\cal H}^{d-1}$ denote $(d-1)$-dimensional Hausdorff measure. 
    For any $K \subset \R^d$ and $r > 0$ we let
    $K_r:= \{y \in \R^d: |x- y| \leq r\}$.
    The co-area formula tells us that 
    $$ \int_{\R^d \setminus W_\la}  \psi\big(d(y,  W_\la)\big)^{q}
     dy \leq
    \int_0^{\infty}  \int_{\partial ((W_\la)_r)} \psi(r)^{q}
    {\cal H}^{d-1}(dx) dr
    $$
    $$ = \int_0^{1}  \int_{\partial ((W_\la)_r)} \psi(r)^{q}
    {\cal H}^{d-1}(dx) dr + \int_1^{\infty}  \int_{\partial ((W_\la)_r)} \psi(r)^{q}
    {\cal H}^{d-1}(dx) dr.
    $$
    The first integral is $O(\la^{(d-1)/d})$.  To evaluate the second integral, recall
    (see e.g. 
    the proof of Lemma 5.12 of \cite{LSY}) that for any convex $K$ we have
    $$
    {\cal H}^{d-1}(\partial K_r) \leq c \mathcal{H}^{d-1}(\partial K)  (1 + r^{d-1}), \quad r > 0.$$
    Thus under the integrability condition 
    \eqref{finitepolar},
    we find 
    \begin{align*}
    \int_{\R^d \setminus W_\la}  \psi\big(d(y, W_\la)\big)^{1/120}
      \nu(dy) &= O(\la^{(d-1)/d}) + 
     \mathcal{H}^{d-1}(\partial W_\la) 
    \int_1^{\infty} \psi(r)^{1/120} (1 + r^{d-1})dr \\
    &= O(\la^{(d-1)/d})
    \end{align*}
and hence $G_q = O(\lambda)$ for $q=1/120$. Since $\psi$ is bounded by $2$, the result follows for $q=3/200$.

Next, we need to show that $\mathcal{I}_{\psi,X}(\theta)$ is bounded by a constant independent of $\lambda$ when $\theta=1/120$ or $\theta=1/50$. For this, note that
\[
\mathcal{I}_{\psi,X}(\theta) \leq \int_{\R^d} \psi\left(\frac{d(x,y)}{2}\right)^{1/240} dy = 2^d d \kappa_d \int_0^\infty \psi(r)^{1/240} r^{d-1} dr,
\]
where $\kappa_d$ denotes the volume of a unit ball in $\R^d$. This suffices to conclude.
\end{enumerate}

  \vskip.3cm

\subsection{Conditions which imply BL-localization} \label{3conditions}
Showing BL-localization requires checking conditions \eqref{dBLspace} and \eqref{dBLtime}, which at first glance may appear challenging. However there are several conditions which are easier to verify and which imply localization. They are as follows.

\begin{enumerate}[(i)]
    \item (Bounding the probability of the difference) Assume that for any $(z,t_z) \in X \times \R$ and $\A \subset X$ with $|\A| \leq 6$ and any $r>0$, one has    \begin{equation}\label{simple1}
    \PP\big(\xi((z,t_z,M_{z}), \tP \cup \tA) \neq \xi^{[r]}((z,t_z,M_{z}), \tP \cup \tA)\big) \leq \tilde{\psi}(r)
    \end{equation}
    for a non-increasing function $\tilde{\psi}:(0,\infty) \rightarrow [0,1]$.
    \sloppy Then we can choose $\psi(r):=\max(2, 8 \tilde{\psi}(r))$ and achieve \eqref{dBLspace}. An analogous  result holds for showing localization in time \eqref{dBLtime}.

    \item (Stabilization)\label{condstab} Given a score $\xi: (X \times \R \times \MM) \times \mathbf{N} \rightarrow \R$, the quantity $R^\xi:= R^\xi((z,t_z,m_z),\chi)$ is a  radius of stabilization  at the point  $(z,t_z,m_z) \in X \times \R \times \MM$ if for any $\chi,\chi' \in \mathbf{N}$
    \begin{equation}\label{eq:ScoreStab}
       \xi\big((z,t_z,m_z),\chi_{R^\xi}\big) = \xi\big((z,t_z,m_z), \chi_{R^\xi} \cup (\chi_{R^\xi}')^c\big) 
    \end{equation}
    where $\chi_{R^\xi} = \chi \cap (B(z,R^\xi) \times \R \times \MM)$ and $(\chi_{R^\xi}')^c = \chi' \cap (B(z,R^\xi)^c \times \R \times \MM)$. Suppose there is $\tilde{\psi}$ such that 
    \begin{equation} \label{Rstab}
     \sup_{(z,t_z) \in X \times \R} \sup_{|\tA| \leq 6} \PP(R^\xi((z,t_z,M_z), \tP \cup \tA) \geq r) \leq \tilde{\psi}(r), \quad r>0.
    \end{equation}
    In this case,  $\xi$ satisfies \eqref{dBLspace} on the metric measure space $(X,d, \nu)$ with $\psi(r) = \min(2, 8 \tilde{\psi}(r))$ and with the short-range scores $\xi^{[r]}$ defined as in \eqref{Restricted2}.
    Indeed, we have for $r > 0$
    \begin{align}
        &d_{BL}\big( \big(\xi((z_i,t_{z_i},M_{z_i}), \tP \cup \tA_i)\big)_{i=1,...,4}, \\
        &\hspace{4cm}\big(\xi((z_i,t_{z_i},M_{z_i}), (\tP \cup \tA_i) \cap B_r(z_i) \times \R \times \MM)\big)_{i=1,...,4}\big) \nonumber \\
        &\leq 2 \PP(R^\xi((z_1,t_{z_1},M_{z_1}),\tP \cup \tA_1) \geq r \text{ or } \hdots \text{ or } R^\xi((z_4,t_{z_4},M_{z_4}),\tP \cup \tA_4) \geq r) \nonumber \\
        &\leq 2\sum_{i=1}^4 \PP(R^\xi((z_i,t_{z_i},M_{z_1}),\tP \cup \tA_i) \geq r)
        \nonumber \\
        &\leq 8\sup_{z \in X} \sup_{|\tA|\leq 6} \PP(R^\xi((z,t_z,M_{z}),\tP \cup \tA) \geq r). \label{dBLbound}
    \end{align}
    Thus if $X = \R^d$ and if $R^\xi$ has an exponentially decaying tail, then $\xi$  is automatically $BL$-localizing.
    \item ($L^q$ stabilization) \label{condLq} Assume there is a $q \geq 1$ and a non-increasing function $\tilde{\psi}:(0,\infty) \rightarrow \R$ such that for any $r>0$
    \begin{equation}\label{simple2}
        \sup_{(z,t_z) \in X \times \R } \sup_{|\hat{\A}|\leq 6}\E\big[|\xi((z,t_z,M_{z}),\hat{\mathcal{P}} \cup \mathcal{A}) - \xi^{[r]}((z,t_z,M_{z}),\hat{\mathcal{P}} \cup \mathcal{A})|^q\big]^{1/q} \leq \tilde{\psi}(r).
    \end{equation}
    Then \eqref{dBLspace} holds with $\psi(r) = \min(2,4\tilde{\psi}(r))$. An analogous statement holds for time localization \eqref{dBLtime}.
\end{enumerate}

For some scores it may also be easier to show that \eqref{dBLspace} or \eqref{dBLtime} hold with $d_{BL}$ replaced by a larger distance such as the total variation distance.

\subsection{ Comparison between our results and the literature} A comparison between our results (specifically Corollary~\ref{mainthmXisW}) and those of \cite{LSY} (Theorem 2.1 and Corollary 2.2) and \cite{LR} is in order: 

\begin{enumerate}[(i)]
    \item We obtain rates of convergence in the Kolmogorov and Wasserstein distances $\mathbf{d}_K$ and  $\mathbf{d}_W$, whereas  \cite{LSY} is limited to $\mathbf{d}_K$.
    
    \item BL-localization in space is weaker than the stabilization criteria required in \cite{PY05}, \cite{P07}, \cite{LSY}, and the localization criteria in \cite{SY08}, thus encompassing a broader range of statistics, including those in Sections  \ref{applic} and \ref{Sectionspacetime}.

    \item We consider Poisson input on  general metric spaces, whereas \cite{LSY} is restricted  to Euclidean space and other metric spaces satisfying a polynomial growth condition on the measure of spheres. 

    \item Theorem~\ref{mainthm1} extends \cite{LR} in several directions, as \cite{LR} is confined to $\X = \R^d$, it does not allow for infinite time horizons, and it requires that $\xi$ satisfy  $L^4$ stabilization, albeit with requirements on the rates of decay which are weaker than those corresponding to the rates given by $BL(\theta)$.

    \item Our results and those in \cite{LSY} both allow for the possibility that scores concentrate around lower-dimensional sets, though in different ways. In \cite{LSY} the concentration  requires that the probability of non-vanishing scores decays exponentially fast with respect to the boundary of a bounded subset of $\R^d$; in our set-up the concentration is expressed in terms of a time variable with respect to the space-time sets $W \times \{\0\} \subset X \times \R$, $X$ a metric space, and  allows for models with infinite time horizon.

    \item \cite{LSY} treats  Poisson and binomial input on a fixed window in Euclidean space (and other spaces satisfying a growth condition on the measure of spheres), letting the intensity go to infinity. When $\X = \R^d$, when the input is Poisson, and when the scores $\xi_{\lambda}$ in \cite{LSY} satisfy $\xi_{\lambda}(x,{\cal X}) = \xi(\la^{1/d}x, \la^{1/d} {\cal X})$, then the approach here applies, only assuming localization and not the stronger stabilization criterion.  See Item \eqref{UG:compact} in Subsection \ref{usersguide} for further details.
    \item The work \cite{BYY24} uses the cumulant method to establish qualitative CLT's for sums of $BL$-localizing scores on general input on $\R^d$, though localization  \eqref{dBLspace} must hold when the four tuple of scores is replaced by a $p$-tuple, $p \in \N$, and moreover for each $p \in \N$, the function $\psi:= \psi_p$ must be decreasing 
    faster than any power, clearly a much stronger condition than what we require here for Poisson input. 
\end{enumerate}

\subsection{Refinements and extensions} \label{remarks}
\begin{enumerate}[(i)]
    \item \label{rem:multi} (Rates of multivariate normal convergence) In Lemmas~\ref{AD1} and \ref{D2}, we establish bounds on the fourth  moment of the first and second difference operators $D,\DD$ applied to $H$. We use these bounds to control the terms $\hat{\gamma}_1,\hdots,\hat{\gamma}_6$ in Theorem~\ref{thm:LPSMarked}. Likewise, these bounds may be used to derive  quantitative CLTs for {\em multivariate Poisson functionals} where each entry is a sum of $BL$-localizing scores, thus going beyond the case where each entry is assumed to be a sum of stabilizing scores, as in \cite{SY23}. Indeed, the error terms in Theorem~1.1 in \cite{SY19} and Theorem~1 in \cite{TT26}, especially for the multivariate $\mathbf{d}_2$ and $\mathbf{d}_3$ distances, similarly consist of fourth  moments of difference operators of $H$. This is an avenue which we do not pursue in this paper.

\item \label{rem:boundedScores}(Bounded scores) If $|\xi|$ is bounded, or if $X = W$ then by modifying the proof accordingly it can be seen that the constants $C_K$ and $C_W$ can be reduced  using $\mathcal{I}_{\psi,X}(\theta)$ and $\mathcal{I}_\phi(\theta')$ with values of $\theta,\theta'$ which exceed $1/240$ and $1/120$, respectively. For example, when $|\xi|$ is bounded then the values of  $\theta,\theta'$ in our main results can be increased by a factor of $5$. This is a  consequence of Lemmas~\ref{AmainLemma} and \ref{A5thmoment} which show that differences of products of moments of four {\em bounded} random variables  are controlled by a bounded Lipschitz distance (usually less than one)  and not the fifth root of  a bounded Lipschitz distance, as is the case when the random variables have a fifth moment. Thus, for example, the values of $\theta,\theta'$ in Theorem~\ref{mainthm1} and Corollaries~\ref{cor:spaceonly} and \ref{mainthmXisW}
   can be increased to $1/48$ and $1/24$ respectively, and the function $G_{1/120}$ may be replaced by the smaller function $G_{1/24}$. 

    \item \label{rem:ScoreDepWindow}(Extensions to score functions depending additionally on windows) 
    The score functions $\xi$ can additionally be made to depend on either the surrounding windows $W$ or on the domains $X$ by defining the function as
    \[
    \xi:(X \times \R \times \MM) \times \mathbf{N} \times \mathcal{B}(X) \rightarrow \R.
    \]
    Our approach in the sequel goes through virtually unchanged with this additional dependence. It is of particular importance when accounting for boundary effects which could arise in e.g. statistics of Voronoi cells.

    \item (Further extensions)
    \sloppy A simplifying step in our proofs  uses the trivial
    inequality $\min(a,b) \leq a^{1/2}b^{1/2}$ for $a,b > 0$.
    Using instead the inequality $\min(a,b) \leq a^{\varepsilon}b^{(1 - \varepsilon)}, \varepsilon \in (0,1)$, it follows that 
    space-time localization hypotheses of the type  $BL(\theta,\theta)$  can be replaced by a 
    $BL(2\theta \varepsilon,2\theta(1 - \varepsilon))$ condition for any $\varepsilon \in (0,1)$.  We have chosen $\varepsilon = 1/2$ for simplicity, but allowing for general $\varepsilon \in (0,1)$ is potentially useful when determining minimal conditions  under which integrability conditions \eqref{integralpsi} and \eqref{integralphi} on $\psi$ and $\phi$ are both satisfied. 

\item The bounds in Theorem~\ref{mainthm1} on $C_K$ and $C_W$ can be tightened, to allow for the case where the integrals in the terms $\mathcal{I}_{\psi,X}$ and $\mathcal{I}_\phi$ are small. Define
\[
A:= c \max(1,M^\xi_{5,X})^{2/5} \mathcal{I}_{\psi,X}(\tfrac{1}{60})^2 \mathcal{I}_{\phi}(\tfrac{1}{60})^2
\]
and set
\[
{\cal J}_{\psi,X}(\theta) :=  
     \sup_{x \in X} \int_{X} \psi \left( \frac{d(x,z)}{2} \right)^\theta \nu(dz)
\]
and
\[
{\cal J}_{\phi}(\theta') :=\int_\R \phi(t) ^{\theta'} \mu(dt).
\]
Then define
\begin{align*}
    \omega_1 &:= \mathcal{J}_{\psi,X}(\tfrac{1}{240}) \mathcal{J}_{\phi}(\tfrac{11}{1200}) \mathcal{J}_{\phi}(\tfrac{1}{120})^{1/2} G_{1/120}^{1/2}\\
    \omega_2 &:= \mathcal{J}_{\psi,X}(\tfrac{1}{120}) \mathcal{J}_{\phi}(\tfrac{1}{120}) \mathcal{J}_{\phi}(\tfrac{1}{60})^{1/2} G_{1/60}^{1/2} \\
    \omega_3 &:=  \mathcal{J}_{\phi}(\tfrac{3}{200}) G_{3/200}
    \\
    \omega_4 &:=  \mathcal{J}_{\phi}(\tfrac{1}{50})^{1/2} G_{1/50}^{1/2}\\
    \omega_5 &:= \mathcal{J}_{\psi,X}(\tfrac{1}{60})^{1/2} \mathcal{J}_{\phi}(\tfrac{1}{60}) G_{1/60}^{1/2}\\
    \omega_6 &:= \mathcal{J}_{\psi,X}(\tfrac{1}{120})^{1/2} \mathcal{J}_{\phi}(\tfrac{1}{120}) G_{11/600}^{1/2}.
\end{align*}
Note that $A$ can never be smaller than $c$,
but the terms $\mathcal{J}_{\psi,X}$ and $\mathcal{J}_\phi(\theta')$ depend on $\phi,\psi$, on the underlying measure and on the space $X$ and might be small, especially when compared to the term $G_q$ with varying $q$. The bound \eqref{eq:dkMain} in Theorem~\ref{mainthm1} can be replaced by the more accurate bound
\[
\mathbf{d}_K\bigg(\frac{H-\E H}{\smash{\sqrt{\Var H}}},N\bigg) \leq
\frac{A}{\Var H} (\omega_1 + \omega_2 + \omega_4 + \omega_5 + \omega_6)
\]
and \eqref{eq:dwMain} by 
\[
\mathbf{d}_W\bigg(\frac{H-\E H}{\smash{\sqrt{\Var H}}},N\bigg) \leq
\frac{A}{\Var H} (\omega_1+\omega_2) + \frac{A^{3/2}}{(\Var H)^{3/2}} \omega_3.
\]
\end{enumerate}

\section{Proof of Theorem~\ref{thm:LPSMarked}}

Before starting the proof of Theorem~\ref{thm:LPSMarked}, we recall some definitions and properties related to Malliavin Calculus on the Poisson space.

Let $(\WW, \mathcal{W},\mu)$ denote a $\sigma$-finite measure space and $\mathbf{N}$ the set of locally finite $\N_0 \cup \{\infty\}$-valued measures on $(\WW,\mathcal{W})$. Let $h \in L^1(\mathbf{N} \times \WW)$ and $\eta$ a Poisson measure on $\WW$ of intensity $\mu$.

Let $h \in L^2(\mathbf{N}\times \WW)$. Then for $\mu$-a.e. $w\in\WW$ we have  
\[
h(\eta,w) = \sum_{n=0}^\infty \operatorname{I}_n(h_n(w,.)),
\]
where $\operatorname{I}_n(h_n)$ is the $n^{th}$ 
\textit{Wiener-Itô integral} of $h_n$ and
\[
h_n(w,w_1,...,w_n)=\frac{1}{n!} \E D^{(n)}_{w_1,...,w_n} h(\eta,w).
\]
Define the symmetrization $\tilde{h}_n$ of $h_n$ by
\[
\tilde{h}_n(w_1,...,w_{n+1})=\frac{1}{n+1} \sum_{k=1}^{n+1} h_n(w_k,w_1,...,w_{k-1},w_{k+1},...,w_{n+1}).
\]
Then we say that $h \in \operatorname{dom}\delta$
if $h \in L^2(\mathbf{N}\times\WW)$ and
\[
\sum_{n=0}^{\infty}(n+1)! \int_{\WW^{n+1}} \tilde{h}_n^2 \, d\mu^{(n+1)} < \infty
\]
and we define the \textit{Skorohod integral} $\delta(h)$ of $h$ by
\[
\delta(h) := \sum_{n=0}^\infty \operatorname{I}_n(h_n).
\]
See \cite[display (2.9) and thereafter]{TT} and \cite[displays (25) and (42)]{Last16} for details on the above notions.

Now let $h \in L^1(\mathbf{N} \times \WW) \cap L^2(\mathbf{N} \times \WW)$ and define for each $\tau \in (0,1)$ and all $x \in \WW$:
\[
P_\tau h_n(\eta,x) := \int \E[h_n(\eta^\tau + \xi,x) | \eta] \Pi_{(1-\tau)\mu}(d\xi),
\]
where $\eta^\tau$ is a $\tau$-thinning of $\eta$ and $\Pi_{(1-\tau)\mu}$ denotes the law of a Poisson process with intensity $(1-\tau)\nu_{\la}$. Note that $P_\tau$ is known as the \textit{Ornstein-Uhlenbeck operator}. We refer to \cite[display (71)]{Last16} and thereafter for details. For the convenience of the reader, we include a result giving useful properties of $P_\tau h_n$, which is taken from \cite[equation (2.13) and Lemma~6.2]{TT}. 

\begin{lemm}\label{lem: Ptauprop}
    Let $h \in L^2(\mathbf{N} \times \WW)$ and let $\tau \in (0,1)$. Then $P_\tau h$ satisfies
    \begin{equation}\label{ip:2.13}
        \E \int_\WW \int_\WW \left( D_z P_\tau h(\eta,w) \right)^2 \mu(dz)\mu(dw) < \infty.
    \end{equation}
and $P_\tau h \rightarrow h$ in $L^2(\mathbf{N} \times \WW)$ as $\tau \rightarrow 1$. Moreover, for $w,z \in \WW$ and all $p \geq 1$,
\begin{align*}
    \E |P_\tau h(\eta,w)|^p \leq \E |h(\eta,w)|^p\\
    \intertext{and}
    \E |D_zP_\tau h(\eta,w)|^p \leq \E |D_zh(\eta,w)|^p.
\end{align*}
\end{lemm}

We are now ready to start the proof of Theorem~\ref{thm:LPSMarked}.
To simplify the bound on the Kolmogorov distance, we start by proving a sharper version of \cite[Corollary~4.7, with $p=2$]{TT}, namely a version with one less term.
The proof proceeds almost identically, but we replace the use of \cite[Theorem~4.2]{TT} by the isometry relation \cite[Theorem~5]{Last16}.

\begin{prop}\label{prop:Cor4.7}
    Let $h \in L^1(\mathbf{N} \times \WW)$ and $G$ be a measurable function of $\eta$ bounded by $c_G<\infty$. Then
    \begin{multline} \label{Cor4.7}
        \left| \E \int_\WW h(\eta,x) \cdot D_x G \ \mu(dx) \right| \\
        \leq c_G \left( \E \int_\WW h(\eta,x)^2 \ \mu(dx) + \E \int_\WW \int_\WW \big( D_x h(\eta,y) \big)^2 \ \mu(dx) \mu(dy) \right)^{1/2}.
    \end{multline}
\end{prop}

Compared with Corollary~4.7 in \cite{TT}, this bound has one less term inside the
square root and it implies Corollary~4.7 in the case $p=2$, since we can choose $\WW = \XX \times [0,1]$ and $\mu = \lambda \otimes ds$. 

\begin{proof}
    Let $(U_n)_{n \in \N} \subset \WW$ be an increasing sequence of sets such that $\bigcup_{n \in \N} U_n = \WW$ and $\mu(U_n) < \infty$ for all $n \in \N$. Define for each $n \in \N$ and all $x \in \WW$:
    \[
    h_n(\eta, x) := \ind{x \in U_n} \max\{-n, \min\{h(\eta,x),n\}\}.
    \]
    It is clear that $h \in L^1(\mathbf{N} \times \WW) \cap L^2(\mathbf{N} \times \WW)$ and thus we can define $P_\tau h_n$ for all $\tau \in (0,1)$ and $n \in \N$.
    Since $h_n \in L^2(\mathbf{N}\times \WW)$, Lemma~\ref{lem: Ptauprop} applies. By \cite[Theorem~5]{Last16}, and since clearly $P_\tau h_n \in L^1(\mathbf{N} \times \WW) \cap L^2(\mathbf{N} \times \WW)$, condition \eqref{ip:2.13} implies  that $P_\tau h_n \in \operatorname{dom}\delta$. By Lemma~6.1 in \cite{TT}, we have
    \[
    \E \int_\WW P_\tau h_n(\eta,x) \cdot D_x G \ \mu(dx) = \E [\delta(P_\tau h_n)G].
    \]
    Using Jensen's inequality and the fact that $|G| \leq c_G$, we have 
    \[
    \left| \E \int_\WW P_\tau h_n(\eta,x) \cdot D_x G \ \mu(dx) \right| \leq c_G \E [\delta(P_\tau h_n)^2]^{1/2}.
    \]
    By Theorem~5 in \cite{Last16}, we have
    \begin{align*}
    \E [\delta(P_\tau h_n)^2]
    &= \E \int_\WW (P_\tau h_n(\eta,x))^2 \mu(dx) + \E \int_\WW \int_\WW D_x P_\tau h_n(\eta,y) \cdot D_y P_\tau h_n(\eta,x) \mu(dx)\mu(dy)\\
    &\leq \E \int_\WW (P_\tau h_n(\eta,x))^2 \mu(dx) + \E \int_\WW \int_\WW \left(D_x P_\tau h_n(\eta,y)\right)^2 \mu(dx)\mu(dy),
    \end{align*}
    where the second line follows by the Cauchy-Schwarz inequality. Using Lemma~\ref{lem: Ptauprop} and the fact that $|h_n|\leq |h|$ and $|Dh_n| \leq |Dh|$ by the definition of $h_n$, we now have
    \begin{equation*}
    \E [\delta(P_\tau h_n)^2] \leq \E \int_\WW (h(\eta,x))^2 \mu(dx) + \E \int_\WW \int_\WW \left(D_x h(\eta,y)\right)^2 \mu(dx)\mu(dy).
    \end{equation*}
    We have shown that for all $\tau \in (0,1)$ and $n \in \N$
    \begin{multline*}
        \left| \E \int_\WW P_\tau h_n(\eta,x) \cdot D_x G \ \mu(dx) \right| \\
        \leq c_G \E \int_\WW (h(\eta,x))^2 \mu(dx) + \E \int_\WW \int_\WW \left(D_x h(\eta,y)\right)^2 \mu(dx)\mu(dy).
    \end{multline*}
    It remains to show that the left hand side converges to 
    \[
    \left| \E \int_\WW h(\eta,x) \cdot D_x G \ \mu(dx) \right|
    \]
    when successively taking $\tau \rightarrow 1$ and $n \rightarrow \infty$. This was shown to hold in Corollary~4.7 in \cite{TT}. 
\end{proof}

\begin{proof}[Proof of Theorem~\ref{thm:LPSMarked}]
For convenience of notation, define $\hat{\YY}:=\X \times \MM \times [0,1]$. Let $\eta$ be a Poisson point process on $\hat{\YY}$ of intensity measure $\nu \otimes \QQ \otimes ds$. For a point $\hat{y}=(y,m_y,s_y) \in \hat{\YY}$, we denote by $\eta_{\hat{y}}$ the restriction of $\eta$ to points in the set $\X \times \MM \times [0,s_y)$.

Note that we can couple $\chi$, a Poisson point process on $\XX \times \MM$ of intensity $\nu \otimes \QQ$, and $\eta$ by defining $\chi$ as the projection $\Pi_{\XX \times \MM}(\eta)$ of $\eta$ onto $\XX \times \MM$. This leads to a natural extension of the functional $F=F(\chi)$ to point processes on $\XX \times \MM \times [0,1]$ by setting 
\[
F(\eta) := F(\Pi_{\XX \times \MM}(\eta)).
\]
The functional $F$ is now a measurable function of $\eta$ and we can apply Theorem~3.2 from \cite{TT} with $q=2$ to get:
\begin{multline}\label{wb}
    \mathbf{d}_W(F,N) \leq \sqrt{\frac{2}{\pi}} \E \left| 1- \int_{\hat{\YY}} D_{\hat{y}} F \cdot \E[D_{\hat{y}}F|\eta_{\hat{y}}] \,(\nu \otimes \QQ \otimes ds) (d\hat{y}) \right| \\
    + 2 \E \int_{\hat{\YY}} |\E [D_y F|\eta_{\hat{y}}]| \cdot |D_{\hat{y}}F|^q \,(\nu \otimes \QQ \otimes ds) (d\hat{y})
\end{multline}
and
\begin{multline}\label{kb}
    \mathbf{d}_K(F,N) \leq \E \left| 1- \int_{\hat{\YY}} D_{\hat{y}} F \cdot \E[D_{\hat{y}}F|\eta_{\hat{y}}] \,(\nu \otimes \QQ \otimes ds) (d\hat{y}) \right| \\
    + \sup_{z \in \R} \E \int_{\hat{\YY}} |\E [D_{\hat{y}} F|\eta_{\hat{y}}]| \cdot D_{\hat{y}} F \cdot D_{\hat{y}}(Ff_z(F) + \ind{F>z}) \,(\nu \otimes \QQ \otimes ds) (d\hat{y}),
\end{multline}
where $f_z$ is the canonical solution to the Stein equation (see (2.26) in \cite{TT}).

The term $\hat{\gamma}_0$ can now be easily obtained by bounding the first term in \eqref{kb} by
\[
    \E \left| 1- \Var(F) \right| +  \E \left| \Var(F)- \int_{\hat{\YY}} D_{\hat{y}} F \cdot \E[D_{\hat{y}}F|\eta_{\hat{y}}] \,(\nu \otimes \QQ \otimes ds) (d\hat{y}) \right|,
\]
which also implies the bound for \eqref{wb}.
\vskip.3cm
\noindent \textbf{Step 1.}
We start by proving that
\[
\E \left| \Var(F)- \int_{\hat{\YY}} D_{\hat{y}} F \cdot \E[D_{\hat{y}}F|\eta_{\hat{y}}] \,(\nu \otimes \QQ \otimes ds) (d\hat{y}) \right|
\]
is bounded by $\hat{\gamma}_1+\hat{\gamma}_2$. Our proof follows the lines of the proof of Theorem~3.3 in \cite{TT}, though we will need to adapt some of the steps in order to bring the marks inside the inner integrals. Define
\[
G:= \int_{\hat{\YY}} D_{\hat{y}} F \cdot \E[D_{\hat{y}}F|\eta_{\hat{y}}] \,(\nu \otimes \QQ \otimes ds) (d\hat{y}) - \Var(F).
\]
Following the steps (8.15)-(8.22) (with p=2) in the proof of \cite[Theorem~3.3]{TT}, we have
\begin{equation}\label{ip:Gbound}
\E |G| \leq \Bigg( \int_{\hat{\YY}} \E \bigg( \int_{\hat{\YY}} \E\big[ |D_{\hat{x}} h(\eta,\hat{y})| \big| \eta_{\hat{x}}\big] (\nu \otimes \QQ_{\MM} \otimes ds)(d\hat{y}) \bigg)^2 (\nu \otimes \QQ_{\MM} \otimes ds)(d\hat{x}) \Bigg)^{1/2},
\end{equation}
where
\[
h(\eta,\hat{y}) := D_{\hat{y}} F \cdot \E\big[D_{\hat{y}} F \big| \eta_{\hat{y}}\big].
\]
Note that in \cite{TT}, it is assumed for simplicity of notation that $\Var(F)=1$. 
However the proof goes through unchanged without this assumption, since $D(cF) = c DF$, for any $c \in \R$.

We now apply Minkowski's integral inequality to the right hand side of \eqref{ip:Gbound}. Note that we group the integrals over marks together with the expectations, splitting them off from the integrals over $\XX \times [0,1]$. We thus get
\begin{multline*}
    \E |G| \leq \Bigg(\int_{\XX \times [0,1]} \bigg( \int_{\XX \times [0,1]} \bigg( \int_\MM \E \bigg( \int_\MM \E\big[ |D_{(x,m_x,s_x)} h(\eta,(y,m_y,s_y)| \big| \eta_{|\XX \times \MM \times [0,s_x)}\big] \\
    \QQ(dm_y) \bigg)^2 \QQ(dm_x) \bigg)^{1/2} (\nu \otimes ds) (dy,ds_y) \bigg)^2 (\nu \otimes ds) (dx,ds_x) \Bigg)^{1/2}.
\end{multline*}
By Jensen's inequality and the tower property, using that $\QQ$ is a probability measure, we have
\begin{multline*}
    \E |G| \leq \Bigg(\int_{\XX \times [0,1]} \bigg( \int_{\XX \times [0,1]} \bigg( \int_\MM\int_\MM\E \big[\big(  D_{(x,m_x,s_x)} h(\eta,(y,m_y,s_y) \big)^2\big]\\
     \QQ(dm_y)\QQ(dm_x) \bigg)^{1/2} (\nu \otimes ds) (dy,ds_y) \bigg)^2 (\nu \otimes ds) (dx,ds_x) \Bigg)^{1/2}.
\end{multline*}
By \cite[equations (8.25),(8.26)]{TT}, we have
\[
D_{\hat{x}} h(\eta,\hat{y}) = \DD_{\hat{x},\hat{y}} F \cdot \E[D_{\hat{y}} F|\eta_{\hat{y}}] + \ind{s_x<s_y} D_{\hat{y}} F \cdot \E[\DD_{\hat{x},\hat{y}} F|\eta_{\hat{y}}] + \ind{s_x<s_y} \DD_{\hat{x},\hat{y}} F \cdot \E[\DD_{\hat{x},\hat{y}} F|\eta_{\hat{y}}].
\]
It follows that
\begin{align*}
    &\bigg( \int_\MM\int_\MM\E\big[\big( D_{(x,m_x,s_x)} h(\eta,(y,m_y,s_y)) \big)^2 \big] \QQ(dm_y)\QQ(dm_x) \bigg)^{1/2}\\
    &\hspace{1cm}\leq \bigg( \int_\MM\int_\MM\E\big[\big(  |\DD_{(x,m_x,s_x),(y,m_y,s_y)} F \cdot \E[D_{(y,m_y,s_y)} F|\eta_{(y,m_y,s_y)}]|\big)^2\big] \QQ(dm_y)\QQ(dm_x) \bigg)^{1/2}\\
    &\hspace{1.5cm}+ \bigg( \int_\MM\int_\MM\E\big[\big( |D_{(y,m_y,s_y)} \hat{F} \\
    &\hspace{5cm}\cdot \E[\DD_{(x,m_x,s_x),(y,m_y,s_y)} F|\eta_{(y,m_y,s_y)}]|\big)^2\big] \QQ(dm_y)\QQ(dm_x) \bigg)^{1/2}\\
    &\hspace{1.5cm}+ \bigg( \int_\MM\int_\MM\E\big[\big(  |\DD_{(x,m_x,s_x),(y,m_y,s_y)} F \\
    &\hspace{5cm}\cdot \E[\DD_{(x,m_x,s_x),(y,m_y,s_y)} F|\eta_{(y,m_y,s_y)}]| \big)^2\big] \QQ(dm_y)\QQ(dm_x) \bigg)^{1/2}.
\end{align*}
We can now apply the Cauchy-Schwarz inequality to the $\int_\MM\int_\MM \E$ integral and use the tower property to derive
\begin{align*}
    &\bigg( \int_\MM\int_\MM\E\big[\big( D_{(x,m_x,s_x)} h(\eta,(y,m_y,s_y)) \big)^2\big] \QQ(dm_y)\QQ(dm_x) \bigg)^{1/2}\\
    &\hspace{1cm}\leq \bigg( \int_\MM\int_\MM\E\big[  (\DD_{(x,m_x,s_x),(y,m_y,s_y)} F)^{4}\big]\QQ(dm_y)\QQ(dm_x)\bigg)^{1/4} \\
    &\hspace{5cm} \cdot
    \bigg( \int_\MM\int_\MM\E\big[  (D_{(y,m_y,s_y)} F)^{4}\big]\QQ(dm_y)\QQ(dm_x)\bigg)^{1/4}\\
    &\hspace{1.5cm}+ \bigg( \int_\MM\int_\MM\E\big[  (D_{(y,m_y,s_y)} F)^{4}\big]\QQ(dm_y)\QQ(dm_x)\bigg)^{1/4}\\
    &\hspace{5cm} \cdot
    \bigg( \int_\MM\int_\MM\E\big[  (\DD_{(x,m_x,s_x),(y,m_y,s_y)} F)^{4}\big]\QQ(dm_y)\QQ(dm_x)\bigg)^{1/4}\\
    &\hspace{1.5cm}+ \bigg( \int_\MM\int_\MM\E\big[  (\DD_{(x,m_x,s_x),(y,m_y,s_y)} F)^{4}\big]\QQ(dm_y)\QQ(dm_x)\bigg)^{1/2}.
\end{align*}
Note now that since $F$ is a function of $\Pi_{\XX \times \MM}(\eta)$, we have 
\[
D_{(x,m_x,s_x)}F(\eta) = D_{(x,m_x)} F(\chi) \qquad \PP-\text{a.s.}
\]
and similarly for the double derivatives. Moreover, we can replace the integrals over $\MM$ by the expectations with respect to the laws of the random variables $M_x,M_y$, simplifying the expression further to
\begin{multline*}
    \bigg( \int_\MM\int_\MM\E\big[\big( D_{(x,m_x,s_x)} h(\eta,(y,m_y,s_y)) \big)^2\big] \QQ(dm_y)\QQ(dm_x) \bigg)^{1/2}\\
    \leq 2\E\big[  (\DD_{(x,M_x),(y,M_y)} F)^{4}\big]^{1/4}
    \cdot
    \E\big[  (D_{(y,M_y)} F)^{4}\big]^{1/4} + \E\big[  (\DD_{(x,M_x),(y,M_y)} F)^{4}\big]^{1/2}.
\end{multline*}

This implies that
\begin{align*}
    \E |G| \leq & \Bigg( \int_\XX \int_{[0,1]} \bigg( \int_\XX \int_{[0,1]} 2\E\big[  (\DD_{(x,M_x),(y,M_y)} F)^{4}\big]^{1/4}
    \cdot
    \E\big[  (D_{(y,M_y)} F)^{4}\big]^{1/4}\\
    &+ \E\big[  (\DD_{(x,M_x),(y,M_y)} F)^{4}\big]^{1/2} ds_y \nu(dy)^2 ds_x \nu(dx) \Bigg)^{1/2} \\
    =& \Bigg( \int_\XX \bigg( \int_\XX 2\E\big[  (\DD_{(x,M_x),(y,M_y)} F)^{4}\big]^{1/4}
    \cdot
    \E\big[  (D_{(y,M_y)} F)^{4}\big]^{1/4}\\
    &+ \E\big[  (\DD_{(x,M_x),(y,M_y)} F)^{4}\big]\bigg)^{1/2} \nu(dy)\bigg)^2 \nu(dx) \Bigg)^{1/2}\\
    \leq& \hat{\gamma}_1 + \hat{\gamma}_2.
\end{align*}
\noindent\textbf{Step 2.} We now show that the second term in \eqref{wb} is bounded by $\hat{\gamma}_3$. Apply Hölder's inequality to deduce that
\begin{multline*}
\E \int_{\hat{Y}} |\E [D_{\hat{y}} F|\eta_{\hat{y}}]| \cdot |D_{\hat{y}}F|^2 \,(\nu \otimes \QQ \otimes ds) (d\hat{y}) \\
\leq \E \int_{\hat{Y}} \E \big[|\E [D_{\hat{y}} F|\eta_{\hat{y}}]|^3 \big]^{1/3} \cdot \E \big[|D_{\hat{y}}F|^3\big]^{2/3} \,(\nu \otimes \QQ \otimes ds) (d\hat{y}).
\end{multline*}
Jensen's inequality now yields the result.

\vskip.3cm
\noindent\textbf{Step 3.} In this step, we show that the second term in \eqref{kb} is bounded by $\hat{\gamma}_4+\hat{\gamma}_5+\hat{\gamma}_6$. This is one of our main innovations, as it shows that the term $\gamma_7$ in \cite{TT} is superfluous in the case $p=2$.

We follow closely the proof of Step 2 of Theorem~3.4 in \cite{TT}(with $p=2$), but we replace the use of Corollary~4.7 in that proof with the bound \eqref{Cor4.7}, which we recall has one fewer term. 
In more detail, we define
\[
g(\eta,\hat{y}):= D_{\hat{y}} F \cdot \big|\E\big[ D_{\hat{y}}F \big| \eta_{\hat{y}} \big]\big|
\]
and
\[
Z_z := Ff_z(F) + \ind{F>z}.
\]
Following the development in \cite[Step 2, proof of Thm. 3.4]{TT}, 
we find that $g \in L^1(\mathbf{N} \times \hat{\YY})$ and $|Z_z| \leq 2$.  Thus we may 
apply Proposition~\ref{prop:Cor4.7} and deduce that the second term in \eqref{kb} is bounded by
\begin{align*}
    &\sup_{z \in \R} \E \int_{\hat{\YY}} g(\eta,\hat{y}) D_{\hat{y}} Z_z (\nu \otimes \QQ \otimes ds)(d\hat{y}) \\
    &\leq 2 \bigg( \E \int_{\hat{\YY}} g(\eta,\hat{y})^2 (\nu \otimes \QQ \otimes ds)(d\hat{y}) \\
    &+ \E \int_{\hat{\YY}} \int_{\hat{\YY}} (D_{\hat{x}} g(\eta,\hat{y}))^2 (\nu \otimes \QQ \otimes ds)(d\hat{y}) (\nu \otimes \QQ \otimes ds)(d\hat{x}) \bigg)^{1/2}\\
    &=: I_1+I_2.
\end{align*}
This gives the terms $I_1$ and $I_2$ in \cite[Step 2, proof of Thm. 3.4]{TT}, but without the term $I_3$, and we proceed to treat the terms $I_1,I_2$ as in \cite{TT}. This results in the final bound having one fewer term, that is to say $\gamma_7$ in 
\cite{TT} is removed.
\end{proof}

\section{Proof of Berry-Esseen bounds  for sums of  BL-localizing scores} \label{Proofs}

\sloppy First  we show that closeness in the $d_{BL}$ distance of four-tuples of random variables implies closeness of their products, whence BL-localization in space-time
implies that products of powers of scores
are well approximated by products of short-range scores.  This establishes  that first and second order difference operators on such products also localize.  The Mecke formula  yields that the first order difference operator $D_{(x,t_x,M_x)}$
applied to the Poisson functional $H$
has a finite fourth moment bounded by a multiple of
$\psi(d(x,W))^{1/50} \phi(t_{x})^{1/50}$ uniformly in $x \in X$  and it also implies that the second order difference operator $\DD_{(x,t_x,M_x),(y,t_y,M_y)}$
applied to $H$ has a fourth moment which decays like 
$$
\psi\left( \frac{d(x,y)}{2} \right)^{1/60} \psi\left(d(x,W)\right)^{1/60} \phi(t_x)^{1/60}\phi(t_y)^{1/60}
$$
uniformly for all $(x,t_x),(y,t_y) \in X \times \R$, up to multiplication with terms independent of $(x,t_x),(y,t_y)$. The aforementioned bounds on the difference operators suffice to deduce Theorem \ref{mainthm1} from 
Theorem~\ref{thm:LPSMarked}, namely they suffice to control the terms $\hat{\gamma}_i, 1 \leq i \leq 6$.

\begin{lemm}\label{AmainLemma}
(closeness in $d_{BL}$ implies closeness of mixed moments, bounded case)
Let $X_1,...X_4,X_1',...,X_4'$ be random variables such that
\[
d_{BL}\big( (X_1,...,X_4),\ (X_1',...,X_4') \big) \leq \alpha,
\]
where $\alpha \geq 0$. 
Assume that there exists $L \in [1,\infty)$ such that $|X_i|,|X_i'|\leq L$ for all $i=1,...4$. Then for all non-negative integers $p_1,p_2,p_3,p_4$ with $q=p_1+p_2+p_3+p_4 > 0$ it is the case that
\[
| \E \Pi_{i=1}^4 X_i^{p_i} - \E \Pi_{i=1}^4 X_i'^{p_i}| \leq 2q L^q \alpha.
\]
\end{lemm}

\begin{proof} If $\alpha=0$, the distribution of $(X_1,\hdots,X_4)$ coincides with that of  $(X_1',\hdots,X_4')$ and the bound holds trivially. We thus assume henceforth $\alpha>0$. The goal is to control 
$| \E \Pi_{i=1}^4 X_i^{p_i} - \E \Pi_{i=1}^4 X_i'^{p_i}|$ via a judicious choice of a 
Lipschitz function on $\R^4$.

To this end, define $f: [-L,L]^4 \to \R$ by $f(s,t,u,v) = s^{p_1}t^{p_2}u^{p_3}v^{p_4}$ and observe that $f$ is continuous and hence Lipschitz on the bounded domain $[-L,L]^4$. 
By the Kirszbraun extension theorem we may extend $f$ to a function $\tilde{f}$ on all of $\R^4$ in such a way that the extension $\tilde{f}$ preserves the Lipschitz norm of $f$, which we assert is bounded by $ 2qL^{q-1}$.

Indeed, by the mean value theorem for multivariate functions, we have
\[
|f(s) - f(t)| \leq \left\| \nabla f \right\|_{\infty} \cdot \| s-t \|,
\]
where $s,t \in \R^4$.
The supremum of each partial derivative of $f$ is at most $qL^{q-1}$, hence $\left\| \nabla f \right\|_{\infty} \leq 2qL^{q-1}$, whence our assertion. Now note that $|f|$ is bounded by $L^q$ and define the truncated function $g_{L^q}$ by
\[
g_{L^q}(x) :=
\begin{cases}
    -L^q    & \text{if } x<-L^q \\
    x       & \text{if } -L^q \leq x \leq L^q \\
    L^q     & \text{if } x > L^q.
\end{cases}
\]
Then $g_{L^q}$ is Lipschitz with Lipschitz constant $1$. Define the function $\tilde{f}_{L^q}:=g_{L^q} \circ \tilde{f}$ and note that $\tilde{f}_{L^q}$ is Lipschitz with Lipschitz constant bounded by $2qL^{q-1}$, it is bounded by $L^q$ and its restriction to $[-L,L]^4$ is identical to $f$. The function $\tilde{f}_{L^q}/\max(2qL^{q-1},L^q)$ is thus 
in $BL(\R^4)$ and  
\begin{align*}
&\big|\E  \Pi_{i = 1}^4 X_i^{p_i} -  \Pi_{i = 1}^4 X_i'^{p_i} \big| \\
&= \big|\E  \tilde{f}_{L^4}\big(X_1,...,X_4\big) - \tilde{f}_{L^4}\big(X_1',...,X_4'\big)\big|\\
&\leq \max(2qL^{q-1},L^q) \cdot d_{BL}\big( (X_1,...,X_4),\ (X_1',...,X_4') \big).
\end{align*}
The proof is complete upon bounding $\max(2qL^{q-1},L^q)$ by $2qL^{q}$.
\end{proof}

We obtain a similar result for unbounded random variables having  a $(q + 1)$-moment, 
though the difference in products of moments is only with precision proportional to 
 $\alpha^{1/(q + 1)}$, which is a loss in accuracy as typically $\alpha \in (0,1)$.

\begin{lemm}\label{A5thmoment}
(closeness in $d_{BL}$ implies closeness of mixed moments, unbounded case)
Let $X_1,...X_4,X_1',...,X_4'$ be random variables such that
\[
d_{BL}\big( (X_1,...,X_4),\ (X_1',...,X_4') \big) \leq \alpha,
\]
with $0\leq\alpha \leq 2$. Assume that there exists $M>0$ and $q \in \N$ such that $\E|X_i|^{q+1},\E|X_i'|^{q+1}\leq M$ for all $i=1,...,4$. Then for all non-negative integers $p_1,p_2,p_3,p_4$ with $p_1+p_2+p_3+p_4=q$, it is the case that
\[
| \E \Pi_{i=1}^4 X_i^{p_i} - \E \Pi_{i=1}^4 X_i'^{p_i}| \leq (36q+16) \max(1,M^{q/(q+1)}) \alpha^{1/(q+1)}.
\]
\end{lemm}

\begin{proof} The case $\alpha=0$ is trivial and we assume henceforth $\alpha>0$. We deduce this from Lemma \ref{AmainLemma}.
For all $L \geq 1$ and $i=1,...,4$, define the truncated variable $X_{i,L} := X_i \cdot \ind{|X_i| \leq L}$, and similarly for $X_i'$. The exact value of $L$ will be chosen later.

We start by showing that
\[
d_{BL}\big( (X_{1,L},...,X_{4,L}),\ (X_{1,L}',...,X_{4,L}') \big) \leq \frac{16M}{L^{q+1}}+\alpha.
\]
Let $f$ be a bounded Lipschitz function with Lipschitz coefficient bounded by $1$. Then we have 
\begin{align*}
& \quad |\E f(X_1,...,X_4) -\E f(X_{1,L},...,X_{4,L})|
\\ 
& 
\leq \E\bigg|\big( f(X_1,...,X_4) - f(X_{1,L},...X_{4,L})\big)
\cdot \bigg( \Pi_{i = 1}^4  {\bf 1}(|X_i| \leq L) +
 {\bf 1}(\bigcup_{i=1}^4 \{|X_i|>L\})\bigg)\bigg|
\\
& 
\leq \E\bigg|\big( f(X_1,...,X_4) - f(X_{1,L},...X_{4,L})\big)
\cdot {\bf 1}(\bigcup_{i=1}^4 \{|X_i|>L\})\bigg|
\\
&
\leq 2 \cdot \big( \sum_{i=1}^4 \PP(|X_i| \geq L) \big) \leq  \frac{2} {L^{q + 1}} \cdot \sum_{i=1}^4 \E |X_i|^{q + 1}
\leq \frac{8M}{L^{q+1}},
\end{align*}
where the third to last inequality uses that $|f|$ is bounded by $1$. The same bound holds for $|\E f(X_1',...,X_4') - \E f(X_{1,L}',...,X_{4,L}')|$ and the bound on the $d_{BL}$ distance of the truncated versions follows.

We now use Lemma~\ref{AmainLemma} on the truncated versions $X_{i,L}$ and $X_{i,L}'$. This yields 
\begin{align*}
&|\E \Pi_{i=1}^4 X_{i,L}^{p_i} - \E  \Pi_{i=1}^4 X_{i,L}'^{p_i} |
\leq 2q L^q \cdot \left(\frac{16M}{L^{q+1}}+\alpha \right)  
\end{align*}

As a next step, we extend this bound to the untruncated variables. In particular, we want to show that
\[
|\E\Pi_{i=1}^4 X_i^{p_i} - \E\Pi_{i=1}^4 X_{i}'^{p_i}| \leq \frac{16M}{L}+ \frac{32qM}{L} + 2qL^q\alpha.
\]
Note that it suffices to show that
\[
\E |\Pi_{i=1}^4 X_i^{p_i} - \Pi_{i=1}^4 X_{i,L}^{p_i}| \leq \frac{8M}{L}
\]
and the same for $X_i',X_{i,L}'$. We only prove the bound for $X_i$, as the one for $X_i'$ follows by identical methods.

We have
\begin{align*}
&\E|X_1^{p_1}X_2^{p_2}X_3^{p_3}X_4^{p_4}  - 
 X_{1,L}^{p_1}X_{2,L}^{p_2}
 X_{3,L}^{p_3}X_{4,L}^{p_4}| \\
 &= \E| X_1^{p_1}X_2^{p_2}X_3^{p_3}(X_4^{p_4} - X_{4,L}^{p_4})
 + X_1^{p_1}X_2^{p_2}(X_3^{p_3} - X_{3,L}^{p_3})X_{4,L}^{p_4}\\
 &\hspace{1cm}+ X_1^{p_1}(X_2^{p_2} - X_{2,L}^{p_2})X_{3,L}^{p_3}X_{4,L}^{p_4}
 + (X_1^{p_1} - X_{1,L}^{p_1})X_{2,L}^{p_2} X_{3,L}^{p_3}X_{4,L}^{p_4}|.
\end{align*}

It thus suffices to show that
\begin{equation}\label{powerbound}
    \E| (X_1^{p_1} - X_{1,L}^{p_1})X_2^{p_2}X_3^{p_3}X_4^{p_4}| \leq \frac{2M}{L}.
\end{equation}
This can be checked as follows. The Hölder and Markov inequalities imply that
\begin{align*}
\E| (X_1^{p_1} - X_{1,L}^{p_1})X_2^{p_2}X_3^{p_3}X_4^{p_4}|  
& \leq 2 \E [|X_1|^{p_1} {\bf 1}(|X_1| \geq L) \cdot |X_2|^{p_2}|X_3|^{p_3}|X_4|^{p_4}]
\\
&
\leq 2\PP( |X_1| \geq L)^{1/(q + 1)} \Pi_{i = 1}^4(\E|X_i|^{q + 1})^{p_i/(q + 1)} \\
& \leq 2(\E|X_1|^{q + 1})^{1/(q+1)} \frac{1}{L} \cdot M^{q/(q+1)},\\
&=\frac{2M}{L}.
\end{align*}
where we use $\PP( |X_1| \geq L) \leq \E|X_1|^{q + 1}/L^{q + 1}$ and
$p_1 + p_2+ p_3+ p_4 = q$. We thus deduced
\[
|\E\Pi_{i=1}^4 X_i^{p_i} - \E \Pi_{i=1}^4 X_{i}'^{p_i}| \leq \frac{16M}{L}+ \frac{32qM}{L} + 2qL^q\alpha.
\]
If $M\geq 1$, we set $L=\left(\frac{2M}{\alpha}\right)^{1/q+1}$, else, we choose $L=\left(\frac{2}{\alpha}\right)^{1/q+1}$. In both cases, $L\geq 1$ as required. Bounding $2^{q/(q+1)}$ by $2$ and $2^{-1/(q+1)}$ by $1$, we get, if $M \geq 1$,
\[
|\E\Pi_{i=1}^4 X_i^{p_i} - \E \Pi_{i=1}^4 X_{i}'^{p_i}| \leq (36q+16) M^{q/(q+1)} \alpha^{1/(q+1)}
\]
and if $M \leq 1$,
\[
|\E\Pi_{i=1}^4 X_i^{p_i} - \E \Pi_{i=1}^4 X_{i}'^{p_i}| \leq (36q+16) \alpha^{1/(q+1)}.
\]
This concludes the proof.
\end{proof}

\begin{lemm}\label{lem:Aprodscores} 
(localization of mixed moments of scores)
Assume the score $\xi: (X \times \R \times \MM)
\times \mathbf{N} \to \R$ satisfies $BL(\theta,\theta')$ space-time localization on $(X,d,\nu) \times (\R, \mu)$ as in Def.~\ref{def:space-time-loc} with short-range scores $(\xi^{[r]})_{r>0},(\xi^{(s)})_{s \in \R}$ and functions $\psi$ and $\phi$.
    
    Then for  any $\mathcal{A}_i \subseteq \{(z_1,t_{z_1}),...,(z_4,t_{z_4}),(x,t_{x}),(y,t_y)\} \in X \times \R$, with $i=1,...,4$, and any non-negative integers $p_1,p_2,p_3,p_4$ such that $p_1+p_2+p_3+p_4=p-1>0$, we have for any $r>0$
    \begin{equation}\label{eq:spaceprods}
    | \E \Pi_{i=1}^4 \xi((z_i,t_{z_i},M_{z_i}),\hat{\mathcal{P}} \cup \hat{\mathcal{A}}_i)^{p_i} - \E \Pi_{i=1}^4 \xi^{[r]}((z_i,t_{z_i},M_{z_i}),\hat{\mathcal{P}} \cup \hat{\mathcal{A}}_i)^{p_i}| \leq C_p \psi(r)^{1/p}
    \end{equation}
    and  $s \in \R$ 
    \begin{equation}\label{eq:timeprods}
    | \E \Pi_{i=1}^4 \xi((z_i,t_{z_i},M_{z_i}),\hat{\mathcal{P}} \cup \hat{\mathcal{A}}_i)^{p_i} - \E \Pi_{i=1}^4 \xi^{(s)}((z_i,t_{z_i},M_{z_i}),\hat{\mathcal{P}} \cup \hat{\mathcal{A}}_i)^{p_i}| \leq C_p \phi(s)^{1/p},
    \end{equation}
    where $C_p=(36p-20) (M_{5,X}^\xi)^{(p-1)/p}$.
\end{lemm}
\begin{proof}
    Inequalities \eqref{eq:spaceprods} and \eqref{eq:timeprods} follow immediately from Lemma~\ref{A5thmoment} with $q=p-1$ using the assumptions on $\xi$ and the short-range scores $\xi^{[r]}$ and $\xi^{(s)}$. When $|\xi|\leq M$, we can apply Lemma~\ref{AmainLemma} directly. Indeed, by the argument in Remark~\eqref{rem:boundedScores} in Section~\ref{remarks}, we can assume without loss of generality that $\xi^{[r]}$ and $\xi^{(s)}$ are also bounded in absolute value by $M$ and so  Lemma~\ref{AmainLemma} is applicable.
\end{proof}
\vskip.3cm
\begin{lemm} \label{AmainlemmaD}
(localization of  mixed moments of first and  second order difference operators)
Assume the score $\xi: (X \times \R \times \MM) \times \mathbf{N} \to \R$ satisfies $BL(\theta,\theta')$ space-time localization on  $(X,d,\nu) \times (\R, \mu)$ as in Def.~\ref{def:space-time-loc} with short-range scores $(\xi^{[r]})_{r>0},(\xi^{(s)})_{s \in \R}$ and functions $\psi$ and $\phi$. Put $C_p=(36p-20) (M_{5,X}^\xi)^{(p-1)/p}$ for $p \in \N$.
    
\noindent (i) 
For all non-negative integers $p_1, p_2, p_3, p_4$ satisfying $p_1 + p_2+ p_3+ p_4 = p-1 >0$, there is a constant $c>0$ depending only on $p$ such that for 
 all $(z_1,t_{z_1}),(z_2,t_{z_2}),(z_3,t_{z_3}),(z_4,t_{z_4}),(x,t_{x}) \in X \times \R$
and   for all $r > 0$
\begin{multline*}
\sup_{\substack{\mathcal{A}_i \subset \{(z_1,t_{z_1}),\hdots,(z_4,t_{z_4})\},\\i=1,\hdots,4}}
\bigg|\E \Pi_{i = 1}^4 \left(D_{(x,t_{x},M_x)} \xi((z_i, t_{z_i},M_{z_i}), \tP \cup \tA_i))\right)^{p_i}\\
 - \E \Pi_{i = 1}^4 \left(D_{(x,t_{x},M_x)} \xi^{[r]}((z_i, t_{z_i},M_{z_i}), \tP \cup \tA_i))\right)^{p_i} \bigg| \leq c C_p \psi(r)^{1/p}
\end{multline*}
and for all $s \in \R$
\begin{multline*}
\sup_{\substack{\mathcal{A}_i \subset \{(z_1,t_{z_1}),\hdots,(z_4,t_{z_4})\},\\i=1,\hdots,4}}
\bigg|\E \Pi_{i = 1}^4 \left(D_{(x,t_{x},M_x)} \xi((z_i, t_{z_i},M_{z_i}), \tP \cup \tA_i))\right)^{p_i}\\
 - \E \Pi_{i = 1}^4 \left(D_{(x,t_{x},M_x)} \xi^{(s)}((z_i, t_{z_i}, M_{z_i}), \tP \cup \tA_i))\right)^{p_i} \bigg| \leq  c C_p\phi(s)^{1/p}.
\end{multline*}

\noindent (ii) 
For all non-negative integers $p_1, p_2, p_3, p_4$ satisfying $p_1 + p_2+ p_3+ p_4 = p-1$,
there is a constant $c>0$ 
depending only on $p$  such that for all $(z_1,t_{z_1}),$ $(z_2,t_{z_2}),$ $(z_3,t_{z_3}),$ $(z_4,t_{z_4}),$ $(x,t_x),$ $(y,t_y) \in X \times \R$
and all $r > 0$
\begin{multline*}
\sup_{\substack{\mathcal{A}_i \subset \{(z_1,t_{z_1}),\hdots,(z_4,t_{z_4})\},\\i=1,\hdots,4}}
\bigg|\E \Pi_{i = 1}^4 \left(\DD_{(x,t_x,M_x),(y,t_y,M_y)} \xi((z_i, t_{z_i}, M_{z_i}), \tP \cup \tA_i))\right)^{p_i}  \\
 - \E \Pi_{i = 1}^4 \left(\DD_{(x,t_x,M_x),(y,t_y,M_y)} \xi^{[r]}((z_i,t_{z_i}, M_{z_i}), \tP \cup \tA_i))\right)^{p_i} \bigg| \leq c C_p \psi(r)^{1/p}
\end{multline*}
and for all $s \in \R$
\begin{multline*}
\sup_{\substack{\mathcal{A}_i \subset \{(z_1,t_{z_1}),\hdots,(z_4,t_{z_4})\},\\i=1,\hdots,4}}
\bigg|\E \Pi_{i = 1}^4 \left(\DD_{(x,t_x,M_x),(y,t_y,M_y)} \xi((z_i, t_{z_i}, M_{z_i}), \tP \cup \tA_i))\right)^{p_i}  \\
 - \E \Pi_{i = 1}^4 \left(\DD_{(x,t_x,M_x),(y,t_y,M_y)} \xi^{(s)}((z_i,t_{z_i}, M_{z_i}), \tP \cup \tA_i))\right)^{p_i} \bigg| \leq c C_p \phi(s)^{1/p}.
\end{multline*}
\end{lemm}

\begin{proof}  This follows from Lemma \ref{lem:Aprodscores} upon writing out each difference operator  and expressing the resulting products as sums of differences of the kind appearing in the statement of Lemma \ref{lem:Aprodscores} and then using the triangle inequality. 
\end{proof}

From here on, we use $c>0$ to denote a scalar which does not depend on $(x,t_x)$, $\xi$, $X$, $W$, or anything else chosen in the setup of Theorem~\ref{mainthm1}. It may depend on the parameter $p$, and nothing else. From this point onward, we choose $p=5$. 

\begin{lemm} \label{AD1}
(fourth moment bounds for first order difference operators applied to $H$)
Let $\xi:  (X \times \R \times  \MM) \times \mathbf{N} \to \R$ be a score function satisfying $BL(\frac{1}{50},\frac{1}{50})$-space-time localization on $(X,d,\nu) \times (\R,\mu)$  as in Def.~\ref{def:space-time-loc} with functions $\psi,\phi$.        
Then there is a constant $c>0$ such that for all $(x,t_x) \in X \times \R$,
\[
\E \big[ (D_{(x,t_{x},M_x)} H)^4\big] \leq c C_5 
{\cal I}_{\psi,X}(\tfrac{1}{50})^4
 I_{\phi}(\tfrac{1}{50})^4 \psi(d(x,W))^{1/50} \phi(t_{x})^{1/50}.
\]
\end{lemm}

\vskip.3cm

\begin{proof}[Proof of Lemma~\ref{AD1}]
To ease the notation, we prove the result for unmarked Poisson points and then indicate the minor changes needed for the marked case. Recall that $\P$ is a Poisson point process on $X$ with intensity $\nu$.

We start by noting that
\[
D_{(x,t_{x})}H = \ind{x \in W} \xi((x,t_{x}),\mathcal{P}) + \sum_{(z,t_z) \in \mathcal{P}\cap (W \times \R)} D_{(x,t_{x})} \xi((z,t_z),\mathcal{P}).
\]

For any $x \in X$, we have
\begin{align}
&\E ( D_{(x,t_{x})} H)^4 \\
& = \E\bigg(\ind{x \in W} \xi((x,t_{x}),\mathcal{P}) + \sum_{(z,t_z) \in \mathcal{P}\cap (W \times \R)} D_{(x,t_{x})} \xi((z,t_z),\mathcal{P})\bigg)^4 \nonumber \\
& \leq 8 \ind{x \in W} \E\xi((x,t_{x}),\P)^4 + 8 \E\bigg(\sum_{(z,t_z) \in \P \cap (W \times \R)}D_{(x,t_{x})} \xi((z,t_z), \P)\bigg)^4. \label{ALemma4.4}
\end{align}
\textbf{Step 1.} We consider  the first term in \eqref{ALemma4.4}. Note that for $s=t_{x}$, we have 
\[
\xi^{(s)}((x,t_{x}), \mathcal{P}) = 0,
\]
by the properties of $\xi^{(s)}$.
In particular, using Lemma~\ref{lem:Aprodscores} with $s=t_{x}$, we deduce that
\begin{align*}
\big|\E\xi((x,t_{x}),\P)^4\big| 
&= \big|\E\xi((x,t_{x}),\P)^4 - \E\xi^{(t_{x})}((x,t_{x}),\P)^4\big| \\
&\leq C_5\phi(t_{x})^{1/5}.
\end{align*}
Moreover, recalling $\psi(0) = 2$, we have 
\[
\ind{x \in W_\la} \leq \psi(d(x,W))/\psi(0).
\]
It follows that the first term in \eqref{ALemma4.4} is bounded by
\be \label{firstterm}
c C_5 \phi(t_{x})^{1/5}\psi(d(x,W)).
\ee
\textbf{Step 2.} 
We expand the second term into the sum of five sums, depending on whether
points $(z_1,t_{z_1}),(z_2,t_{z_2}),(z_3,t_{z_3}),(z_4,t_{z_4}) \in \mathcal{P}$ coincide or not and then apply the Mecke formula to each sum. We obtain
\begin{align} \label{Aexpansion}
& \E\left(\sum_{(z,t_z) \in \P \cap W \times \R}D_{(x,t_{x})} \xi((z,t_z), \P)\right)^4 \\
\nonumber 
&
= I_1(4) + 4 I_2(3,1) + 3 I_2(2,2) + 6 I_3(2,1,1) + I_4(1,1,1,1),
\end{align}
where for any $\ell \in \{1,...,4\}$, we let $p_1,\hdots,p_\ell$ be non-negative integers such that $p_1+\hdots+p_\ell=4$ and define
\begin{multline*}
    I_\ell(p_1,...,p_\ell) := \int_{(W \times \R)^\ell}  \E \big[ \Pi_{i=1}^\ell D_{(x,t_{x})} \xi((z_i,t_{z_i}),\P\cup\{(z_1,t_{z_1}),\hdots,(z_\ell,t_{z_\ell})\})^{p_i} \big] \\
    (\nu \otimes \mu)(dz_1,dt_{z_1})\hdots(\nu \otimes \mu)(dz_\ell,dt_{z_\ell}).
\end{multline*}
We need to show that $I_\ell(p_1,...,p_\ell)$ is bounded by a constant times 
$$
C_5 
( {\cal I}_{\psi}(\frac{1}{50}))^4
( I_{\phi}(\frac{1}{50}))^4
\psi(d(x,W))^{1/50} \phi(t_{x})^{1/50}
$$
for $\ell,p_1,...,p_\ell$ as above. We introduce the simplified notation
\[
\mathbf{z}_\ell := \{(z_1,t_{z_1}),\hdots,(z_\ell,t_{z_\ell})\}
\]
as well as
\[
\xi_i := \xi\big((z_i,t_{z_i}),\P_\la\cup\mathbf{z}_\ell\big),
\]
and similarly for $\xi_i^{[r]}$.

We now specify a value of the short-range parameter $r$ which  ensures that the product $\Pi_{i = 1}^{\ell} D_{(x,t_{x})}\xi_i^{[r]}$ vanishes.  
Defining 
\begin{equation}\label{definr}
r:= \max(d(x, z_1),\hdots,d(x, z_\ell)),
\end{equation}
which is strictly positive almost everywhere, we guarantee that   $\Pi_{i = 1}^{\ell} D_{(x,t_{x})}\xi_i^{[r]}$ vanishes.
Indeed, say that the maximum is realized at the index $i_{\text{max}}$, with $r=d(x,z_{i_{\text{max}}})$. Then 
\begin{align*}
    D_{(x,t_{x})}\xi_{i_{\text{max}}}^{[r]} &= D_{(x,t_{x})} \xi^{[r]}\big((z_{i_{\text{max}}},t_{z_{i_{\text{max}}}}),\big(\mathcal{P} \cup \mathbf{z}_\ell \big) \cap (B_r(z_{i_{\text{max}}}) \times \R)\big)\\
    &= \xi^{[r]}\big((z_{i_{\text{max}}},t_{z_{i_{\text{max}}}}),\big(\mathcal{P} \cup \mathbf{z}_\ell \cup \{(x,t_{x})\} \big) \cap (B_r(z_{i_{\text{max}}}) \times \R)\big) \\
    &\hspace{2cm}- \xi^{[r]}\big((z_{i_{\text{max}}},t_{z_{i_{\text{max}}}}),\big(\mathcal{P} \cup \mathbf{z}_\ell\big) \cap (B_r(z_{i_{\text{max}}}) \times \R)\big)\\
    &=0,
\end{align*}
since the ball $B_r(z_{i_{\text{max}}})$ is open and thus both short-range  scores are evaluated on exactly the same point set. Hence we have
\[
\big| \E \Pi_{i = 1}^{\ell} \big(D_{(x,t_x)}\xi_i\big)^{p_i} \big| = \big|\E \Pi_{i = 1}^{\ell} \big(D_{(x,t_x)}\xi_i\big)^{p_i} - \E \Pi_{i = 1}^{\ell}\big(D_{(x,t_x)}\xi_i^{[r]}\big)^{p_i}\big|.
\]
By Lemma~\ref{AmainlemmaD}~(i), we deduce that for all $\ell \in \{1,...,4\}$
\begin{equation}\label{eq:DxbyPsiBound}
\big| \E \Pi_{i = 1}^{\ell}\big(D_{(x,t_x)}\xi_i\big)^{p_i}  \big| \leq c C_5 \psi(r)^{1/5},
\end{equation}
with $r:= \max(d(x, z_1),\hdots,d(x, z_\ell))$.
Likewise we put 
$$
s = \max(t_{z_1},...,t_{z_\ell},t_{x})
$$
and note that the product $\Pi_{i = 1}^{\ell} D_{(x,t_x)}\xi_i^{(s)}$ vanishes. Indeed, if the maximum is realized as $s=t_{z_i}$, then
\[
\xi^{(s)}\big((z_i,t_{z_i}), \mathcal{P} \cup \mathbf{z}_\ell \cup \mathcal{A} \big) = \ind{t_{z_i}<s} \xi^{(s)}\big( (z_i,t_{z_i}), (\mathcal{P} \cup \mathbf{z}_\ell \cup \mathcal{A}) \cap (\XX \times (-\infty,s) \big) = 0,
\]
where $\mathcal{A} \in \{\{(x,t_x)\},\emptyset\}$. Hence $D_{(x,t_x)} \xi_i^{(s)}=0$ and the product vanishes. If the maximum is realized as $s=t_x$, then for any $i \in \{1,\hdots,\ell\}$
\begin{align*}
&\xi^{(t_x)}\big((z_i,t_{z_i}), \mathcal{P} \cup \mathbf{z}_\ell \cup \{(x,t_x)\} \big) \\
&\quad = \ind{t_{z_i}<t_x} \xi^{(t_x)}\big( (z_i,t_{z_i}), (\mathcal{P} \cup \mathbf{z}_\ell \cup \{(x,t_x)\}) \cap (\XX \times (-\infty,t_x)) \big) \\
&\quad = \ind{t_{z_i}<t_x} \xi^{(t_x)}\big( (z_i,t_{z_i}), (\mathcal{P} \cup \mathbf{z}_\ell) \cap (\XX \times (-\infty,t_x)) \big) \\
&\quad = \xi^{(t_x)}\big((z_i,t_{z_i}), \mathcal{P} \cup \mathbf{z}_\ell \big),
\end{align*}
and hence $D_{(x,t_x)} \xi_i^{(s)}=0$.

By Lemma~\ref{AmainlemmaD}~(i), we deduce that for all $\ell \in \{1,...,4\}$
\begin{equation}\label{eq:DxbyPhiBound}
\big| \E \Pi_{i = 1}^{\ell}\big(D_{(x,t_x)}\xi_i\big)^{p_i}  \big| \leq c C_5\phi(s)^{1/5},
\end{equation}
with $s = \max(t_{z_1},...,t_{z_\ell},t_{x})$.
Hence  \eqref{eq:DxbyPsiBound} and \eqref {eq:DxbyPhiBound} yield
\begin{equation}\label{eq:DxbyPsiPhiBound}
\big| \E \Pi_{i = 1}^{\ell}\big(D_{(x,t_x)}\xi_i\big)^{p_i}  \big| \leq c C_5 \min(\psi(r)^{1/5}, \phi(s)^{1/5})
\leq c C_5 \psi(r)^{1/10} \phi(s)^{1/10}.
\end{equation}

Now note that
\[
d(x,W) \leq r = \max(d(x, z_1),\hdots,d(x, z_\ell)).
\]
Indeed, if $x \in W$, then $d(x,W)=0$ and the inequality holds trivially. If $x \notin W$, then $d(x,W) = \inf_{z \in W} d(x,z) \leq r$. 
The non-increasing property of $\psi$ and $\phi$ and the simple inequality $\psi(\max(a_1,...,a_k)) = \min_{i \leq k} \psi(a_i) \leq \Pi_{i \leq k} (\psi(a_i))^{1/k}$ holding for any $a_1,a_2,...,a_k > 0$  and $k \in \N$ implies that for our specific choice of $r>0$
\begin{align}
    &\psi(r) \leq \psi(d(x, z_1))^{1/(\ell+1)} \cdot \hdots \cdot \psi(d(x, z_\ell))^{1/(\ell+1)} \cdot \psi(d(x,W))^{1/(\ell+1)} \label{psibound}
 \end{align}   
 and similarly for  our choice of $s>0$
    \begin{align}
        &\phi(s) \leq \phi(t_{z_1})^{1/(\ell+1)} \cdot \hdots \cdot \phi(t_{z_\ell})^{1/(\ell+1)} \cdot \phi(t_{x})^{1/(\ell+1)} \label{phibound}. 
\end{align}
Thus
\begin{align}
    &I_\ell(p_1,\hdots,p_\ell) \nonumber \\
    &=\int_{(W \times \R)^\ell} \big| \E \big(D_{(x,t_x)}\xi_1\big)^{p_1} \cdot \hdots \cdot \big(D_{(x,t_x)}\xi_\ell\big)^{p_\ell} \big| (\nu \otimes \mu)(dz_1,dt_{z_1})\hdots(\nu \otimes \mu)(dz_\ell,dt_{z_\ell}) \nonumber \\
    &\leq c C_5 \psi(d(x,W))^{\frac{1}{10(\ell+1)}} \phi(t_x)^{\frac{1}{10(\ell+1)}}
    \int_{W^\ell} \Pi_{i=1}^\ell \psi(d(x, z_i))^{\frac{1}{10(\ell+1)}} \nu(dz_1)\hdots\nu(dz_\ell)
    \nonumber \\
    &
    \hskip7cm\int_{\R^\ell} \Pi_{i=1}^\ell \phi(t_{z_i})^{\frac{1}{10(\ell+1)}} \mu(dt_{z_1}) \hdots \mu(dt_{z_\ell})\nonumber \\
    &= c C_5 \psi(d(x,W))^{\frac{1}{10(\ell+1)}} \phi(t_x)^{\frac{1}{10(\ell+1)}}
    \left(\int_{W} \psi(d(x, z_1))^{\frac{1}{10(\ell+1)}} \nu(dz_1)\right)^\ell
    \left(\int_{\R} \phi(t_{z_1})^{\frac{1}{10(\ell+1)}} \mu(dt_{z_1})\right)^\ell
    \nonumber \\
    &  \leq c C_5 
    I_\psi(\tfrac{1}{50})^{\ell} I_\phi(\tfrac{1}{50})^{\ell}
    \psi(d(x,W))^{\frac{1}{10(\ell+1)}} \phi(t_x)^{\frac{1}{10(\ell+1)}}
    , \label{boundIl}
\end{align}
where we recall the definitions of   and 
$I_\psi$  and $I_\phi$, at \eqref{integralpsi} and \eqref{integralphi}, respectively.
This concludes the proof of the unmarked version.

The proof of the marked version necessitates the following changes.  We attach random marks $M_{z_1}, M_{z_2}, M_{z_3}, M_{z_4}$
to the respective points $z_1,...,z_4$ as well as marks $M_x, M_y$ to points $x$ and $y$. 
The expansion of $\E( D_{(x,t_x,M_x)} H)^4$ via the Mecke formula goes through, provided all expectations here and subsequently are understood to denote  expectations with respect to these marks as well as the Poisson point process.
The choice of the radius $r$ is  unchanged
and the proof goes through mutatis mutandis.
\end{proof}

\begin{lemm} \label{D2}
(fourth moment bounds for second  order difference operators applied to $H$)
Let $\xi:  (X \times \R \times  \MM) \times \mathbf{N} \to \R$ be a score function satisfying $BL\big(\frac{1}{60},\frac{1}{60}\big)$-space-time localization on $(X,d,\nu) \times (\R,\mu)$ as in Def.~\ref{def:space-time-loc} with functions $\psi,\phi$.
Then there is a constant $c>0$ such that for  all $(x,t_x),(y,t_y) \in X \times \R$,
\begin{multline*}
\E \big[( \DD_{(x,t_x,M_x),(y,t_y,M_y)} H)^4\big] \leq 
c C_5 {\cal I}_{\psi,X} \big(\frac{1}{60}\big)^4 \cdot  {\cal I}_\phi\big(\frac{1}{60}\big)^4 \\
\cdot \psi\left( \frac{d(x,y)}{2} \right)^{1/60} \psi\left(d(x,W)\right)^{1/60} \phi(t_x)^{1/60}\phi(t_y)^{1/60}.
\end{multline*}
\end{lemm}

\vskip.3cm 
\begin{proof}[Proof of Lemma~\ref{D2}]
As in the proof of Lemma~\ref{AD1}, we first prove the result for unmarked Poisson points, as only minor changes are needed for the marked case.
As in Lemma 5.2 of \cite{LSY}, a computation shows that
for all $(x,t_x),(y,t_y) \in X \times \R$
\begin{align*}
\DD_{(x,t_x),(y,t_y)} H & = \ind{y \in W} D_{(x,t_x)}\xi((y,t_y), \P) + 
\ind{x \in W} D_{(y,t_y)}\xi((x,t_x), \P) \\
& + 
\sum_{(z,t_z) \in \P \cap W \times \R} \DD_{(x,t_x),(y,t_y)} \xi((z,t_z), \P). 
\end{align*}
It follows that 
\begin{align}
\E ( \DD_{(x,t_x),(y,t_y)} H)^4 & 
 \leq 3^3\bigg[ \ind{y \in W} \E\left( D_{(x,t_x)}\xi((y,t_y), \P)\right)^4 + \ind{x \in W} \E \left(D_{(y,t_y)}\xi((x,t_x),\P)\right)^4 \notag\\ 
 & 
\hskip2cm + \E\Big(\sum_{(z,t_z) \in \P \cap W \times \R}\DD_{(x,t_x),(y,t_y)} \xi((z,t_z), \P)\Big)^4\bigg]. \label{ip:Dxydecomp}
\end{align}
We now show that each term on the right-hand side is bounded by
$$
c C_5 \psi\left( \frac{d(x,y)}{2} \right)^{1/60} \psi\left(d(x,W)\right)^{1/60} \phi(t_x)^{1/60}\phi(t_y)^{1/60}
$$
for some $c$ large enough.

\noindent \textbf{Step 1.}
For the first term in \eqref{ip:Dxydecomp}, note that by \eqref{eq:DxbyPsiBound} with $\ell=1$ and $p_1=4$ in the proof of Lemma~\ref{AD1}, we have
\[
\E \left(D_{(x,t_x)}\xi((y,t_y), \mathcal{P})\right)^4\ind{y\in W} \leq c C_5 \psi(d(x,y))^{1/5} \ind{y\in W}.
\]
Given $y \in W$, it holds that $\max\{d(x, W), \frac{1}{2}d(x,y)\} \leq d(x,y)$,
and hence the decreasing property of $\psi$ implies
\[
\psi(d(x,y))^{1/5} \leq \psi(d(x,W))^{1/10} \cdot \psi\left(\frac{d(x,y)}{2} \right)^{1/10}.
\]

We also have
$$
\E \left(D_{(x,t_x)}\xi((y,t_y), \mathcal{P})\right)^4\ind{y\in W} \leq c C_5 \phi(t_x))^{1/10} \phi(t_y))^{1/10}, 
$$
which holds by \eqref{eq:DxbyPhiBound} and 
\eqref{phibound} with $\ell = 1$. 
Combining the last two displays and again using the decreasing property of $\phi$ and $\psi$ gives
$$
\E \left(D_{(x,t_x)}\xi((y,t_y), \mathcal{P})\right)^4\ind{y\in W} \leq
cC_5 \psi(d(x,W))^{1/20}  \psi\left(\frac{d(x,y)}{2} \right)^{1/20}
-\phi(t_x))^{1/20} \phi(t_y))^{1/20}.
$$

By boundedness of $\psi$ and $\phi$, we can respectively decrease the exponents to $1/60$ at the cost of introducing a multiplicative constant $c$.  The second term in \eqref{ip:Dxydecomp} can be treated in the same way. 

\noindent \textbf{Step 2.}
To treat the last term in \eqref{ip:Dxydecomp}, we expand it and apply the Mecke formula to each sum to obtain the analog of \eqref{Aexpansion}, namely 
\begin{multline} \label{expansionDxy}
\E\Big(\sum_{(z,t_z) \in \P \cap W \times \R}\DD_{(x,t_x),(y,t_y)} \xi((z,t_z), \P)\Big)^4 \\
\leq J_1(4) + 4 J_2(3,1) + 3 J_2(2,2) + 6 J_3(2,1,1) +  J_4(1,1,1,1) 
\end{multline}
where
\begin{multline*}
    J_\ell(p_1,...,p_\ell) := \int_{(W \times \R)^\ell} \E \Pi_{i=1}^\ell \left(\DD_{(x,t_x),(y,t_y)} \xi((z_i,t_{z_i}),\mathcal{P} \cup \{(z_1,t_{z_1}),...,(z_\ell,t_{z_\ell})\})\right)^{p_i} \\
    (\nu \otimes \mu)(dz_1,dt_{z_1})... (\nu \otimes \mu)(dz_\ell,dt_{z_\ell}).
\end{multline*}

The sequence of simple identities
\begin{align}
     \DD_{u,v} \xi(w, \chi) \label{eq:Dxyident} 
    &= \xi(w, \chi \cup \{u,v\}) - \xi(w, \chi \cup \{u\}) - \xi(w, \chi \cup \{v\}) - \xi(w, \chi) \\ \nonumber
    &= D_{u} \xi(w,\chi \cup \{v\}) - D_{u} \xi(z,\chi) \\\nonumber
    & = D_{v} \xi(w, \chi \cup \{u\}) - D_{v} \xi(w,\chi),
\end{align}
hold for any $u,v,w \in \XX \times \R$ and any locally finite set $\chi \subset \XX \times \R$, and hence  
$$
  \DD_{(x,t_x),(y,t_y)}\xi^{[r]}((z,t_z), \P \cup \mathcal{A})= 0
$$
when we choose the radius $r$ as
\[
r := \max(d(x,z),d(y,z)).
\]
To show the bounds on the terms $J_\ell(p_1,...,p_\ell)$, we proceed as in the proof of Lemma~\ref{AD1} and take positive reals  $p_1,\hdots,p_\ell$ such that $p_1+\hdots+p_\ell=4$ and $\ell \leq 4$.
Recall the shorthand $\xi_i:=\xi((z_i,t_{z_i}),\mathcal{P} \cup \{(z_1,t_{z_1}),\hdots,(z_\ell,t_{z_\ell})\})$.
We want to show
\begin{multline} \label{JlBound}
    J_\ell(p_1,...,p_\ell) \leq  
    c C_5 \max_{l \leq 4} ({\cal I}_{\psi,\X}(\tfrac{1}{60})^{l} \cdot  {\cal I}_\phi(\tfrac{1}{60})^{l}) \\
    \cdot \psi\left( \frac{d(x,y)}{2} \right)^{1/60} \psi\left(d(x,W)\right)^{1/60} \phi(t_x)^{1/100}\phi(t_y)^{1/100}.
\end{multline}
Define the radius $r$ to be 
\be \label{defrxy}
r:= \max\{d(x,z_1),d(y,z_1),\hdots, d(x,z_\ell),d(y,z_\ell)\}
\ee
Note that
\[
\DD_{(x,t_x),(y,t_y)}\xi_1^{[r]} \cdot \hdots \cdot \DD_{(x,t_x),(y,t_y)}\xi_\ell^{[r]}=0,
\]
since at least one of the terms in the product will vanish. It follows that
\begin{align*}
\big|\E \Pi_{i = 1}^{\ell} \big(\DD_{(x,t_x),(y,t_y)} \xi_i\big)^{p_i} \big|
&= \big|\E \Pi_{i = 1}^{\ell} \big(\DD_{(x,t_x),(y,t_y)} \xi_i\big)^{p_i} - \E \Pi_{i = 1}^{\ell} \big(\DD_{(x,t_x),(y,t_y)} \xi_i^{[r]}\big)^{p_i} \big| \\
&\leq c C_5 \psi(r)^{1/5},
\end{align*}
by Lemma~\ref{AmainlemmaD} (ii). Note that by the triangle inequality
\[
\frac{1}{2}d(x,y) \leq \frac{1}{2} \big(d(x,z_1) + d(y,z_1) \big) \leq r,
\]
and moreover
\[
d(x,W) \leq d(x,z_1) \leq r,
\]
since $z_1 \in W$. It follows that
\[
\psi(r) \leq 
c C_5
\psi\left( \frac{d(x,y)}{2} \right)^{1/(\ell+2)} \psi(d(x,W)^{1/(\ell+2)} \Pi_{i=1}^\ell \psi(d(x,z_i))^{1/(\ell+2)}
\]
and hence
\begin{equation}\label{eq:DxbyPhiBound.5}
\big|\E \Pi_{i = 1}^{\ell} \big(\DD_{(x,t_x),(y,t_y)} \xi_i\big)^{p_i} \big|
\leq c C_5 \psi\left( \frac{d(x,y)}{2} \right)^{1/30} \psi(d(x,W)^{1/30} \Pi_{i=1}^\ell \psi(d(x,z_i))^{1/30}.
\end{equation}
Next, we put 
\[
s := \max(t_{z_1},...,t_{z_\ell},t_x,t_y).
\]
Note that $\xi^{(t_{z_i})}((z_i,t_{z_i}),\chi)=0$ for any locally finite point set $\chi \subset \X \times \R$ and hence $\DD_{(x,t_x),(y,t_y)}\xi^{(t_{z_i})}_i = 0$. Moreover,
\[
D_{(x,t_x)}\xi^{(t_y)}((z,t_z), \chi \cup \{(y,t_y)\}) = D_{(x,t_x)}\xi^{(t_y)}((z,t_z), \chi)
\]
for any set $\chi \subset \X \times \R$, therefore
\[
\DD_{(x,t_x),(y,t_y)}\xi_i^{(t_y)} = 0,
\]
and by the same reasoning $\DD_{(x,t_x),(y,t_y)}\xi_i^{(t_x)} = 0$. 

As in the proof of Lemma~\ref{AD1}, using Lemma~\ref{AmainlemmaD}, we derive that for all $\ell \leq 4$
\begin{align} 
\big| \E \Pi_{i = 1}^{\ell}\big(\DD_{(x,t_x),(y,t_y)}\xi_i\big)^{p_i} \big|
&\leq c C_5  \phi(s)^{1/5} \notag\\
&\leq  c C_5 \phi(t_{z_1})^{1/30} \cdot \hdots \cdot \phi(t_{z_\ell})^{1/30} \cdot \phi(t_{x})^{1/30}\cdot \phi(t_{y})^{1/30}. \label{eq:DxyPhiBound4}
\end{align}

We conclude that $\big| \E \Pi_{i = 1}^{\ell}\big(\DD_{(x,t_x),(y,t_y)}\xi_i\big)^{p_i}  \big|$ is bounded by the minimum of the right-hand sides of 
\eqref{eq:DxbyPhiBound.5} and 
\eqref{eq:DxyPhiBound4}.  As before, using that the minimum is bounded by the product of the square roots, this gives
\begin{align*}
    &\big| \E \Pi_{i = 1}^{\ell}\big(\DD_{(x,t_x),(y,t_y)}\xi_i\big)^{p_i}  \big|\\
    & \leq c C_5 \psi\left( \frac{d(x,y)}{2} \right)^{1/60} \psi(d(x,W)^{1/60} \phi(t_{x})^{1/60}\phi(t_{y})^{1/60}\Pi_{i=1}^\ell \psi(d(x,z_i))^{1/60}
    \Pi_{i = 1}^{\ell} \phi(t_{z_i})^{1/60}.
\end{align*}
We conclude that for any $\ell \in \{1,...,4\}$
\begin{align*}
     J_\ell(p_1,...,p_\ell) & \leq  c C_5 \psi\left( \frac{d(x,y)}{2} \right)^{1/60} \psi(d(x,W)^{1/60}  \phi(t_{x})^{1/60}\phi(t_{y})^{1/60}
    \\
    & \hskip.5cm
    \int_{(W \times \R)^{\ell}}
    \Pi_{i=1}^\ell \psi(d(x,z_i))^{1/60} 
\Pi_{i = 1}^{\ell} \phi(t_{z_i})^{1/60}(\nu \otimes \mu)(dz_1,dt_{z_1})... (\nu \otimes \mu)(dz_{\ell},dt_{z_{\ell}})\\
    &\leq c C_5
    ( {\cal I}_{\psi,\X} (\tfrac{1}{60}))^4 \cdot ( {\cal I}_\phi(\tfrac{1}{60}))^4
    \\
    & \hskip1.5cm
    \psi\left( \frac{d(x,y)}{2} \right)^{1/60} \psi(d(x,W)^{1/60}  \phi(t_{x})^{1/60}\phi(t_{y})^{1/60},
\end{align*}
where the last inequality follows by Fubini's theorem and 
by the quantities defined in \eqref{integralpsi} and \eqref{integralphi}.  This gives \eqref{JlBound} and
concludes the proof.
\end{proof}

We are now ready to prove our main theorem.

\begin{proof}[Proof of Theorem~\ref{mainthm1}]
We first need to show that $\E H^2<\infty$ and that $D_{(x,t_x,M_x)}H$ is square-integrable. By Lemma~\ref{AD1}, we have
\begin{multline*}
    \int_{X \times \R} \E \left[ (D_{(x,t_x,M_x)} H)^2 \right] (\nu \otimes \mu) (dx,dt_x) \\
    \leq c (M^\xi_{5,X})^{1/2}  \mathcal{I}_{\psi,X}(\tfrac{1}{50})^2 \mathcal{I}_\phi(\tfrac{1}{50})^2 \\
    \int_X \psi(d(x,W))^{1/100}\nu(dx) \int_\R \phi(t_x)^{1/100} \mu(dt_x).
\end{multline*}
This is finite, since $M^\xi_{5,X},\ \mathcal{I}_{\psi,X}(\frac{1}{240}),\ \mathcal{I}_{\phi}(\frac{1}{120})$ and $G_{1/120}$ are all finite.

To show that $\E H^2<\infty$, we use Remark~\ref{Hisfinite}, which establishes that $\E |H|<\infty$, and the Poincaré inequality (see e.g. Theorem~10 in \cite{Last16}) by which we have
\[
\E H^2 \leq (\E |H|)^2 + \int_{X \times \R} \E \left[ (D_{(x,t_x,M_x)} H)^2 \right] (\nu \otimes \mu) (dx,dt_x) < \infty.
\]

We now proceed to bound the terms $\hat{\gamma}_1,\hdots,\hat{\gamma}_6$ in Theorem~\ref{thm:LPSMarked}, with $\XX$ in that theorem set to be $X \times \R$, equipped with the measure $\nu \otimes \mu$, where $\mu$ accounts for the extra space $\R$.   
We use the version of Theorem~\ref{thm:LPSMarked} as detailed in Remark~\ref{rem:LPSasUsual}, which means in particular that the term $\hat{\gamma}_0$ is zero. 
Put 
$$
A := c C_5 {\cal I}_{\psi,X} (\tfrac{1}{60})^4 \cdot {\cal I}_\phi(\tfrac{1}{60})^4.
$$

Plugging the bounds from  Lemmas~\ref{AD1} and \ref{D2} into the term $\hat{\gamma}_1$, we get 
\begin{multline*}
\hat{\gamma}_1 \leq c
A^{1/2}
\bigg( \int_{X \times \R} \bigg( \int_{X \times \R}
\psi \left( \frac{d(x,y)}{2}\right)^{1/240} \underbrace{\psi(d(y,W))^{1/200}}_{\leq 2} \psi(d(x,W))^{1/240}\\
\phi(t_y)^{11/1200} \phi(t_x)^{1/240}(\nu \otimes \mu)(dy,dt_y)) \bigg)^2 (\nu \otimes \mu)(dx,dt_x)\bigg)^{1/2}.
\end{multline*}
We now use the definitions of $\mathcal{I}_{\psi,X}(\theta)$ and $\mathcal{I}_\phi(\theta)$ at \eqref{integralpsi} and \eqref{integralphi} to infer
\[
\hat{\gamma}_1 \leq cC_5^{1/2} {\cal I}_{\psi, X}\big(\tfrac{1}{240}\big)^3 \cdot {\cal I}_{\phi}\big(\tfrac{1}{120}\big)^{7/2} G_{1/120}^{1/2},
\]
where
\[
G_{q} := \int_{X} \psi(d(x,W))^{q} \nu(dx).
\]

The term $\hat{\gamma}_2$ satisfies
\begin{multline*}
\hat{\gamma}_2 \leq A^{1/2} \bigg( \int_{\X \times \R} \bigg( \int_{\X \times \R} \psi\left(\frac{d(x,y)}{2}\right)^{1/120}\psi(d(x,W))^{1/120} \\
\phi(t_y)^{1/120} \phi(t_x)^{1/120} (\nu \otimes \mu)(dy,dt_y) \bigg)^2 (\nu \otimes \mu)(dx,dt_x) \bigg)^{1/2},
\end{multline*}
and hence 
\[
\hat{\gamma}_2 \leq c C_5^{1/2} 
 {\cal I}_{\psi, X} (\tfrac{1}{120})^3 \cdot {\cal I}_{\phi}(\tfrac{1}{120})^{7/2} G_{1/60}^{1/2}.
\]
For the term $\hat{\gamma}_3$, we get
\[
\hat{\gamma}_3 \leq c C_5^{3/4} 
{\cal I}_{\psi,X}(\tfrac{1}{50})^3 \cdot {\cal I}_{\phi}(\tfrac{3}{200})^4 G_{3/200},
\]
whereas 
\[
\hat{\gamma}_4 \leq c C_5^{1/2} 
{\cal I}_{\psi,X}(\tfrac{1}{50})^2 \cdot {\cal I}_{\phi}(\tfrac{1}{50})^{5/2}
G_{1/50}^{1/2},
\]
For $\hat{\gamma}_5$, we find
\begin{multline*}
\hat{\gamma}_5 \leq 
c A^{1/2}
\bigg( \int_{(X \times \R)^2} 
\psi \left( \frac{d(x,y)}{2}\right)^{1/60} \psi(d(x,W))^{1/60} \\
\phi(t_y)^{1/60} \phi(t_x)^{1/60}(\nu \otimes \mu)(dy,dt_y)  (\nu \otimes \mu)(dx,dt_x)\bigg)^{1/2}
\end{multline*}
whence 
\[
\hat{\gamma}_5 \leq c C_5^{1/2} 
{\cal I}_{\psi,X}(\tfrac{1}{60})^{5/2} \cdot {\cal I}_{\phi}(\tfrac{1}{60})^{\ 3}  G_{1/60}^{1/2}
\]
and for $\hat{\gamma}_6$ we have 
\[\hat{\gamma}_6 \leq c C_5^{1/2} 
{\cal I}_{\psi,X}(\tfrac{1}{120})^{5/2} \cdot {\cal I}_{\phi}(\tfrac{1}{120})^{3}
G_{11/600}^{1/2}. 
\]

Collecting terms and taking the variance prefactors  into account, as in Remark~\ref{rem:LPSasUsual}, the  bound for $\mathbf{d}_K \left(  \frac{H - \E H} {\sqrt{\Var H}}, N \right)$ is thus given by
\[
\mathbf{d}_K \left(  \frac{H - \E H} {\sqrt{\Var H}}, N \right)  \leq 
c C_5^{1/2} 
{\cal I}_{\psi,X}(\tfrac{1}{240})^{3} \cdot {\cal I}_{\phi}(\tfrac{1}{120})^{7/2}
G_{1/120}^{1/2} \frac{1}
{\Var H}
\]
whereas for the Wasserstein distance we get
\[
\mathbf{d}_W \left(  \frac{H - \E H} {\sqrt{\Var H}}, N \right)
\leq c C_5^{1/2} 
 {\cal I}_{\psi,X}(\tfrac{1}{240})^{3} \cdot {\cal I}_{\phi}(\tfrac{1}{120})^{7/2}
G_{1/60}^{1/2} \frac{1}
{\Var H}
\]
\[
+  c C_5^{3/4} 
 {\cal I}_{\psi,X}(\tfrac{1}{50})^{3} \cdot{\cal I}_{\phi}(\tfrac{3}{200})^{4}
G_{3/200} \frac{1}
{\Var H^{3/2}}.
\]
This concludes the proof of Theorem~\ref{mainthm1}.
\end{proof}

\section{BL-localizing statistics in space}
\label{applic}
BL-localized statistics appear in a wide range of stochastic spatial models.
Here we deduce from our main results 
quantitative CLT's for local $U$-statistics in metric measure spaces,  as well as quantitative CLT's for Poisson functionals expressible as a sum of stabilizing scores in hyperbolic space. Statistics which do not satisfy stopping set stabilization 
may nonetheless localize in space, putting them within the scope of our general results.   This   includes statistics of  interacting diffusions on random spatial graphs in $\R^d$
\cite{BYY24}. This section discusses these examples in more detail. 

\subsection{Local U-statistics on the Poisson space}

We  establish rates of normal convergence for general local $U$-statistics of Poisson point processes on general metric spaces. 
Local $U$-statistics with deterministic radii include, among others, statistics of the random geometric graph, such as total (weighted) edge-lengths, simplex counts, and number of isolated points. CLTs in previous works
often involved considerable effort or were limited in scope.  Theorem~\ref{thm:USTAT}, a simple consequence of our main results, allows one to show proximity bounds to the normal almost effortlessly in general metric spaces.

\begin{theo}\label{thm:USTAT}
(normal approximation of local U-statistics on metric measure spaces)
Let $(X,d,\nu)$ be a  metric measure space with $\sigma$-finite measure $\nu$ and metric $d$. 
Fix $\delta>0$ and a function $f_\delta:X^k \times \MM \rightarrow \R$ such that $f_\delta((x_1,m_{x_1}),\hdots,(x_k,m_{x_k}))=0$ whenever $\max(d(x_i,x_j): i,j \in \{1,\hdots,k\})\geq\delta$.  Let $W \subset X$ with $\nu(W)<\infty$.
We put
\[
c_\delta := \sup_{x \in X} \nu(B_{2\delta}(x)).
\]
and assume that $c_\delta < \infty$.
Let $(\MM,\QQ)$ be a space of marks and $\hat{\P}$ a Poisson measure of intensity $\nu \otimes \QQ$
on $X \times \MM$.
For $x\in W$ and $m_x \in \MM$, we define
\[
\xi((x,M_x), \tP) := 
\sum_{(x_1,...,x_{k-1}) \in (\P \setminus \{x\})^{k-1}_{\neq}} f_{\delta}((x, M_x),(x_1,M_{x_1}),\hdots,(x_{k-1}, M_{x_{k-1}})),
\]
where $(\P \setminus \{x\})^{k-1}_{\neq}$ denotes the set of $(k-1)$-tuples of distinct points in $\P\setminus \{x\}$.
We define the \textit{local U-statistic}
\[
H := \sum_{x \in \P \cap W} \xi((x,M_x),\hat{\P}),
\]
where $M_x$ is the independent mark associated to $x \in \P$. Then $H$ satisfies the bounds
\[
\mathbf{d}_K \left(\frac{H - \E H}{\smash{\sqrt{\Var H}}}, N \right) \leq c (M^\xi_{5})^{2/5} \max(1,c_\delta)^3 \frac{(\nu(W)+\nu(W_\delta))^{1/2}}{\Var H}
\]
and
\begin{multline*}
\mathbf{d}_W \left(\frac{H - \E H}{\smash{\sqrt{\Var H}}}, N \right) \leq c (M^\xi_{5})^{2/5} \max(1,c_\delta)^3 \frac{(\nu(W)
+\nu(W_\delta))^{1/2}}{\Var H} \\
+ c (M^\xi_{5})^{3/5} \max(1,c_\delta)^3 \frac{\nu(W)+\nu(W_\delta)}{(\Var H)^{3/2}},
\end{multline*}
where $W_\delta = \{x \in X \setminus W: d(x,W)<\delta\}$ and
\[
M_5^\xi := \max(1,\sup_{\substack{z \in X \\ \mathcal{A} \subset X, |\mathcal{A}|\leq 6}} \E |\xi((z,M_z),\hat{\mathcal{P}} \cup \hat{\mathcal{A}})|^5).
\]
\end{theo}

\noindent \textbf{Remark.} Proximity bounds to the normal have already been established in 
\cite{RSchulte}, in a comparable set-up, albeit only with respect to the Wasserstein distance. Here we obtain normal approximation bounds in both the Wasserstein and Kolmogorov distance as a consequence of the space version of Corollary~\ref{mainthmXisW}. Our set-up is optimized to treat the case where $\delta$ does not 
 depend on $W$, and hence $M_5^\xi$ and $c_\delta$ are also independent of 
 $W$, giving Berry-Esseen bounds for $H$ whenever $\Var H$ is of the order  $\nu(W)$.

\begin{proof} We deduce this from Corollary~\ref{cor:spaceonly}. 
We show $BL(\theta)$-space localization for any $\theta > 0$. We take the short-range score  $\xi^{[r]}$ to be
\[
\xi^{[r]}(\hat{x},\hat{\chi}) :=
\begin{cases}
    0 & \text{if } r < \delta\\
    \xi(\hat{x},\chi \cap B_r(x) \times \MM) & \text{else}.
\end{cases}
\]
The scores $\xi((z,M_z),\hat{\chi})$ and $\xi^{[r]}((z,M_z),\hat{\chi})$ coincide if $r \in [\delta, \infty)$ and hence the $5$th moment of $\xi^{[r]}((z,M_z),\hat{\chi})$ is bounded by $M_5$. Recalling that the $d_{BL}$ distance is bounded by $2$, we have
\[
    d_{BL} \big( \big(\xi((z_i,M_{z_i}),\hat{\P}_\la \cup \hat{\mathcal{A}}_i)\big)_{i=1,\hdots,4}, \big(\xi^{[r]}((z_i,M_{z_i}),\hat{\P}_\la \cup \hat{\mathcal{A}}_i)\big)_{i=1,\hdots,4}\big) \leq 2 \ind{r < \delta}.
\]
Put $\psi(r) :=  2\ind{r < \delta}$. Then for any $\theta > 0$
\[
\sup_{x \in X} \int_X \psi\left(\frac{d(x,z)}{2}\right)^{\theta} \nu(dz) = 2^\theta \sup_{x \in X} \nu(B_{2\delta}(x)) = 2^\theta c_\delta < \infty,
\]
which yields $C_W,C_K \leq c\max(1,c_\delta^3)$. As for $G_q$ for some $q>0$, we have
\[
G_q = 2^q \nu(W) + \int_{X \setminus W} 2^q \ind{d(x,W)<\delta} \nu(dx) = 2^q (\nu(W)+\nu(W_\delta)).
\]
This concludes the proof.
\end{proof}

\subsection{Stabilization and localization in hyperbolic space} 

In this section we deduce from the space version of Corollary~\ref{mainthmXisW} good Poincar\'e bounds for localizing functionals in hyperbolic space.   Prior attempts to establish either quantitative or qualitative CLTS for such functionals sometimes involved
mapping the problem to Euclidean space \cite{FY}. This involved considerable technical analysis.  Here we bypass the need for such mappings and directly obtain
quantitative  CLT's via Corollary~\ref{mainthmXisW}. In works like \cite{STT}, the Malliavin-Stein method is used to derive a CLT in hyperbolic space, but it necessitates bounding add-one cost operators directly, which is usually more involved than computing bounds for score functions. In our work, it is enough to show that score functions stabilize or localize.

\begin{theo}\label{thm:hyperbolic} (normal approximation for sums of localizing score functions in hyperbolic space)
    Denote by $\mathcal{H}^d$ the $d$-dimensional hyperbolic space of curvature $-1$ and $\nu$ its standard Riemannian measure. For any $\lambda\geq 1$, let $W_\la \subseteq X_\la \subseteq \mathcal{H}^d$ with either $W_\la=X_\la$ or $W_\la=B(p,\la)$ for a fixed point $p\in \hyp$ and arbitrary $X_\la \subseteq \hyp$ containing $W_\la$. Let $\mathcal{P}_\la$ be a Poisson measure on $X_\la$ of intensity $\nu$ and $\hat{\mathcal{P}}_\la$ its marked version. Let $\xi:(X_\la \times \MM) \times \mathbf{N}_{X_\la \times \MM} \rightarrow \R$ be a score function on the marked space domain $X_\la \times \MM$ and define
    \[
    H_\la=\sum_{z \in \P_\la \cap W_\la} \xi((z,M_z),\hat{\P}_\la).
    \]
    Assume that the score $\xi$ satisfies $BL(\tfrac{1}{240})$-localization in space uniformly over $\la \geq 1$ such that 
    \begin{enumerate}[(i)]
        \item the function $\psi$ can be written as $\psi(r)=\min(2,\exp(-\tau(r)))$, where $\tau:\R \rightarrow \R$ is such that
        \begin{equation}\label{eq:hypPsiCond}
        \liminf_{r \rightarrow \infty} \frac{\tau(r)}{240(d-1)r} > 1.
        \end{equation}
        \item the moments $M_{5,X_\la}^\xi$ satisfy
        \[
        \sup_{\la \geq 1} M_{5,X_\la}^\xi < \infty.
        \]
    \end{enumerate}
    Then there is a constant $C>0$ such that for any $\la \geq 1$
    \begin{equation}\label{dKhyp}
    \mathbf{d}_K\left( \frac{H_\la - \E H_\la}{\sqrt{\Var H_\la}},N\right) \leq C \frac{\nu(W_\la)^{1/2}}{\Var H_\la}
    \end{equation}
    and
    \begin{equation}\label{dWhyp}
    \mathbf{d}_W \left( \frac{H_\la - \E H_\la}{\sqrt{\Var H_\la}},N\right) \leq C \frac{\nu(W_\la)^{1/2}}{\Var H_\la} + C \frac{\nu(W_\la)}{\Var H_\la^{3/2}}.
    \end{equation}
\end{theo}
\textbf{Remark.} If $\xi$ admits a radius of stabilization $R^\xi$ as defined in Remark~\eqref{condstab} of Subsection~\ref{3conditions}, then it suffices that  $R^\xi$  satisfies the tails bounds \eqref{Rstab} with $\psi(r)=\exp(-\tau(r))$ as above and $BL(\tfrac{1}{240})$-localization follows.

We now provide some background on hyperbolic space and then prove  Theorem~\ref{thm:hyperbolic}. 
Recall that  $d$-dimensional hyperbolic space $\hyp$ is the unique simply connected Riemannian manifold of constant sectional curvature $-1$. Let $d_{\hyp}(.,.)$ denote the corresponding Riemannian metric and $\nu$ the Riemannian volume. We refer to \cite{CFK97,Rat19,HHT21} for further details. Throughout this section, we fix $p \in \hyp$ as an arbitrarily chosen fixed point of $\hyp$. For any  $x \in \hyp$ and any $r>0$, the volume of the ball $B_r(x)$ centered at $x$ and of radius $r$ is 
\[
\nu(B_r(x)) = d \kappa_d \int_0^r \sinh^{d-1}(u)du,
\]
where $\kappa_d = \frac{\pi^{d/2}}{\Gamma(1+d/2)}$, see e.g. \cite[Eq. (3.26)]{Rat19}. There are constants $\gamma_d,\Gamma_d>0$ such that 
\begin{equation}\label{eq:hypvolBound}
    \gamma_d e^{r(d-1)} \leq \nu(B_r(x)) \leq \Gamma_d e^{r(d-1)}, \
    r \geq 2
\end{equation}
as shown e.g. in \cite[Lemma~4]{OT23}.

Moreover by \cite[pages 123-125]{Chavel}, the following integration formula in polar coordinates holds:
\begin{equation}\label{eq:hypSpher}
    \int_{\hyp} f(x) \nu(dx) = d\kappa_d \int_{S^{d-1}(p)} \int_0^\infty \sinh^{d-1}(u) f(\exp_p(uv)) du\,\sigma_p(dv),
\end{equation}
where $S^{d-1}(x)$ is the $(d-1)$-dimensional unit sphere in the tangent space $T_{x}\hyp$ of $\hyp$ at $x$ and $\sigma_x$ is the normalized spherical Lebesgue measure on $S^{d-1}(x) \subset T_{x}\hyp$. The map $\exp_x : T_{x}\hyp \to \hyp$ from the tangent space to hyperbolic space denotes the exponential map. By the nature of the construction, the point $\exp_x(uv)$ is the point in $\hyp$ which is reached when traveling a distance $u$ from $x$ along a geodesic ray in the direction of $v\in S^{d-1}(x)$. Hence  $d(\exp_x(uv),x)=u$. See \cite[pages 123--125]{Chavel} for further details.

\begin{proof}[Proof of Theorem~\ref{thm:hyperbolic}]
    We use Corollary~\ref{cor:spaceonly} and need to show that the terms $\mathcal{I}_{\psi,X_\la}(\tfrac{1}{240})$, and thus $C_W$ and $C_K$, are bounded by constants independent of $\la$ and that $G_q$ can be bounded by a constant multiple of $\nu(W_\la)$. We start by giving a bound on $\mathcal{I}_{\psi,X_\la}(\theta)$ with $\theta = \frac{1}{240}$. Fix $x\in \hyp$ and introduce hyperbolic spherical coordinates as in \eqref{eq:hypSpher}:
\begin{align*}
    \int_{\hyp} \psi(d(x,z)/2)^\theta \nu(dz) 
    &= d\kappa_d \int_{S^{d-1}(x)} \int_0^\infty \sinh^{d-1}(u) \psi\big( \frac{d(\exp_x(uv),x)}{2}\big)^\theta du \sigma_x(dv)\\
    &= d\kappa_d \int_0^\infty \sinh^{d-1}(u) \exp(-\theta\,\tau(u/2)) du \\
    &\leq 2^{-d+1}d\kappa_d \int_0^\infty \exp(-\theta\, \tau(u/2) + (d-1)u) du\\
    &= 2^{-d+2}d\kappa_d \int_0^\infty \exp(-\theta\, \tau(s) + 2(d-1)s) ds\\
    &<\infty
\end{align*}
where the last line follows because of \eqref{eq:hypPsiCond} together with $\theta = 1/240$. The last line is independent of the choice of $x \in \hyp$, and hence $\mathcal{I}_{\psi,X_\la}(\theta)$ is uniformly bounded. It remains to show a bound on $G_q$ with $q=\frac{1}{120}$ or $q=\frac{3}{200}$.
Recall that
\[
G_q = 2^q \nu(W_\la) + \int_{X_\la \setminus W_\la} \psi(d_{\hyp}(x,W_\la))^q \nu(dx).
\]
If $W_\la=X_\la$, there is nothing to show. We will now deal with the case $W_\la = B(p,\la)$. Introducing hyperbolic coordinates centered at $p$ (see \eqref{eq:hypSpher}), we get
\begin{align*}
    \int_{\hyp \setminus W_\la} \psi(d_{\hyp}(x,W_\la))^q \nu(dx)
    &= \int_{\hyp \setminus W_\la} \psi(d_{\hyp}(x,p)-\la)^q \nu(dx)\\
    &= d\kappa_d\int_{S^{d-1}(p)} \int_\la^\infty \sinh^{d-1}(u) \psi(d_{\hyp}(\exp_p(uv),p)-\la)^q du \sigma_p(dv) \\
    &=d\kappa_d \int_\la^\infty \sinh^{d-1}(u) \psi(u-\la)^q du,
\end{align*}
where the first equality follows because $W_\la = B(p,\la)$ and the second equality follows from the fact that $\exp_p(uv)$ is exactly the point reached from $p$ when traveling a distance $u$ along a geodesic ray in the direction of $v$. Now note that $\sinh(x)\leq e^x$ for $x>0$ and we get
\begin{align*}
    \int_{\hyp \setminus W_\la} \psi(d_{\hyp}(x,W_\la))^q \nu(dx)
    &\leq d \kappa_d \int_\la^\infty e^{(d-1)u} e^{-q\tau(u-\la)} du\\
    &= d\kappa_d e^{(d-1)\la} \int_0^\infty e^{(d-1)u-q\tau(u))} du.
\end{align*}
By condition \eqref{eq:hypPsiCond}, the integral is finite for both $q=\frac{1}{120}$ and $q=\frac{3}{200}$. Moreover, since by \eqref{eq:hypvolBound} it holds that $\nu(W_\la)\geq \gamma_de^{r(d-1)}$, we have
\[
\int_{\hyp \setminus W_\la} \psi(d_{\hyp}(x,W_\la))^q \nu(dx) \leq C \nu(W_\la)
\]
for some positive constant $C>0$ independent of $\la$. This shows that $G_q \leq C_1 \nu(W_\la)$ for a constant $C_1>0$ independent of $\la$ and concludes the proof.
\end{proof}

\begin{coro} (normal approximation for the sum of power weighted edge lengths of the $k$-nearest neighbor graph in $d$-dimensional hyperbolic space)
We work in the setup of Theorem~\ref{thm:hyperbolic}. Let $X_\la=W_\la=B(p,\la) \subset \hyp$ by a ball of radius $\la\geq2$ and $p\in\hyp$ a fixed point. Let $\mathcal{P}_\la$ be a Poisson measure of intensity $\nu$ on $W_\la$ and let $k\text{NN}(\P_\la)$ be the  undirected $k$-nearest neighbor graph on 
$\P_\la$. For $\alpha \sg 0$ we consider the sum of power weighted edge lengths
\[
H_\la^\alpha := \sum_{\text{edges } e \in k\text{NN}(\P_\la)} \text{length}(e)^\alpha.
\]
Then the functional $H_\la^\alpha$ satisfies \eqref{dKhyp} and \eqref{dWhyp} with $\nu(W_\la) \leq \Gamma_d e^{(d-1)\la}$.
\end{coro}

We make no attempt to find lower bounds for
$\Var H_\la$ and regard this as a separate problem for which the strategies outlined in \cite{SchulteTrapp2024} could be useful.

\begin{proof}
Denote by $V_k(x,\P_\la)$ the set of $k$ points closest to $x$ in $\P_\la \setminus \{x\}$, or all points (except $x$) if $\P_\la$ contains $k$ or fewer points (note that this definition is valid whether or not $x$ is in $\P_\la$). In the undirected $k$-nearest neighbor graph, we connect points $x$ and $y$ by an edge if $x \in V_k(y,\P_\la)$ or $y \in V_k(x,\P_\la)$. Define the score $\xi$ as follows:
\[
\xi(x,\P_\la) := \sum_{y \in V_k(x,\P_\la)} \big( \tfrac{1}{2}\ind{x \in V_k(y,\P_\la)} + \ind{x \notin V_k(y,\P_\la)}\big) d_{\hyp}(x,y)^\alpha,
\]
where $d_{\hyp}$ is the intrinsic distance in $\hyp$ and $\alpha\geq 0$. Then 
\[
H_\la^\alpha := \sum_{x \in \P_\la} \xi(x,\P_\la) = \sum_{\text{edges } e \in k\text{NN}(\P_\la)} \text{length}(e)^\alpha.
\]
We will apply Theorem~\ref{thm:hyperbolic}. We need to check that
\begin{itemize}
    \item $\xi$ has a radius of stabilization as \eqref{eq:ScoreStab} in Remark~\eqref{condstab} in Subsection~\ref{3conditions}, where $R^\xi$ verifies \eqref{Rstab} with a function $\psi$ of the type given in \eqref{eq:hypPsiCond};
    \item The scores $(\xi^{[r]})_{r \in (0,\infty]}$ (with $\xi^{[r]}$ as in \eqref{Restricted2}) satisfy $M^\xi_{5,\hyp} < \infty$. 
\end{itemize}

 Define 
    \[
    R(x,\P_\la) := 2 d_{kNN}(x,\P_\la) := 2 \max\{d_{\hyp}(x,y) : y\in V_k(x,\P_\la)\}
    \]
    and note that for any point sets $\chi,\chi' \in \mathbf{N}_{\mathcal{H}_d}$, we have 
    \[
    \xi(x,\chi) = \xi(x,(\chi \cap B(x,R(x,\chi)) \cup (\chi' \cap B^c(x,R(x,\chi)).
    \]
    Indeed, the set $V_k(x,\chi)$ is clearly included in $B(x,R(x,\chi))$ and does not depend on points outside of $B(x,R(x,\chi))$, nor does it change when points are added, removed or resampled outside of $B(x,R(x,\chi))$.  To compute $\xi(x,\chi)$, it remains to determine if, for $y \in V_k(x,\chi)$, we have $x \in V_k(y,\chi \cap \{x\})$ or not. Note that $x$ is in $V_k(y,\chi \cup \{x\})$ if and only if $(\chi \setminus \{y\}) \cap B(y,d_{\hyp}(x,y))$ contains fewer than $k$ points. However, since $y\in V_k(x,\chi)$, one has $B(y,d_{\hyp}(x,y)) \subset B(x,R(x,\chi))$, hence we conclude that the value of $\xi(x,\chi)$ is entirely determined by points inside $B(x,R(x,\chi))$.

    Next, we give a bound for the tails of $d_{kNN}(x,\P_\la)$ and hence for the tails of $d_{kNN}(x,\P_\la \cup \mathcal{A})$, since $d_{kNN}(x,\chi)$ can only decrease when points are added to $\chi$. For all $r \sg 0$ we have
    \begin{align*}
        \PP(d_{kNN}(x,\P_\la) \geq r) &= \PP(|\P_\la \cap B(x,r)| \leq k-1) \\
        &\leq k\exp(-\nu(W_\la \cap B(x,r))) \max\{1,\nu(W_\la \cap B(x,r))^{k-1}\} \\
        &\leq c_k \exp(-\nu(W_\la \cap B(x,r))/2).
    \end{align*}
   Note that if $r \leq 4\la$ we have
    \[
    \nu(W_\la \cap B(x,r)) = \nu(B(p,\la) \cap B(x,r)) \geq \nu(B(w,r/4)),
    \]
    for some point $w \in \hyp$. Indeed, choose a point $w$ lying on the geodesic between $x$ and $p$ such that $d_{\hyp}(x,w)=r/2$. Then $B(w,r/4) \subset B(x,r)$ and for any $z \in B(w,r/4)$, one has $d_{\hyp}(z,p) \leq d_{\hyp}(z,w) + d_{\hyp}(w,p)\leq r/4 + d_{\hyp}(x,p) - d_{\hyp}(x,w) \leq \la-r/4$, hence $z \in B(p,\la)$ and $B(w,r/4) \subset B(p,\la)$. Note that if $r>4\la$, then $\PP(d_{kNN}(x,\P_\la) \geq r)=0$ and there is nothing to prove. As we are only interested in the tails of the above bound, we can assume from here on that $r\geq 8$. Then by \eqref{eq:hypvolBound}, we have
\[
\nu(B(w,r/4)) \geq \gamma_d \exp(r(d-1)/4)
\]
and hence
\begin{equation}\label{eq:dkNNTails}
\PP(d_{kNN}(x,\P_\la) \geq r) \leq c_k \exp(-\gamma_d e^{r(d-1)/4}).
\end{equation}
The function $\tau(r):= \gamma_d e^{r(d-1)/4}-\log c_k$ clearly satisfies \eqref{eq:hypPsiCond}, hence by Remark~\eqref{condstab} in Section~\ref{3conditions}, the score $\xi$ satisfies $BL(\tfrac{1}{240})$-localization. To show the moment bounds on $\xi^{[r]}$, note that for $r \in (0,\infty]$ and $\mathcal{A}\subset \mathcal{H}^d$, we have the sequence of upper bounds
\begin{align*}
\xi(x,(\P_\la \cup \mathcal{A}) \cap B_r(x))
&\leq \sum_{y \in V_k((\P_\la \cup \mathcal{A})\cap B_r(x))} d_{\mathcal{H}^d}(x,y)^\alpha \\
&\leq \sum_{y \in V_k(\P_\la \cup \mathcal{A})} d_{\mathcal{H}^d}(x,y)^\alpha \\
&\leq k d_{kNN}(x,\P_\la)^{\alpha},
\end{align*}
which follows from the following considerations:
\begin{itemize}
    \item $\tfrac{1}{2}\ind{x \in V_k(y,\P_\la)} + \ind{x \notin V_k(y,\P_\la)}$ is upper bounded by $1$ in the first inequality;
    \item one always has $V_k((\P_\la \cup \mathcal{A})\cap B_r(x)) \subset V_k(\P_\la \cup \mathcal{A})$, since either the set $\P_\la \cup \mathcal{A})\cap B_r(x)$ already contains $k$ points, in which case $V_k((\P_\la \cup \mathcal{A})\cap B_r(x)) = V_k(\P_\la \cup \mathcal{A})$, or the set $\P_\la \cup \mathcal{A})\cap B_r(x)$ contains less than $k$ points and the inclusion of sets of neighbours is strict;
    \item the last inequality follows because $d_{kNN}(x,\P_\la \cap \mathcal{A}) \leq d_{kNN}(x,\P_\la)$, which is clear since the distance to the $k$th nearest neighbour can only decrease when more points are added.
\end{itemize}
It follows that
\[
\E [ \xi^{[r]}(x,\P_\la \cup \mathcal{A})^5] \leq k^5\E [d_{kNN}(x,\P_\la)^{5\alpha}],
\]
which is clearly finite and bounded independently of $\la$ by \eqref{eq:dkNNTails}. This finishes the proof.
\end{proof}

It is a simple matter to apply Theorem \ref{thm:hyperbolic} to certain statistics of the random geometric graph.  We illustrate with the following example.

\begin{coro}
(normal approximation of the number of isolated points in the random geometric graph in $d$-dimensional hyperbolic space  $\hyp$) As above, let $X_\la=W_\la=B(p,\la) \subset \hyp$ be a ball of radius $\la\geq2$ and $p\in\hyp$ a fixed point. Let $\mathcal{P}_\la$ be a Poisson measure of intensity $\nu$ on $W_\la$.  Fix $\rho > 0$. Consider the random geometric graph on $\mathcal{P}_\la$ 
with parameter $\rho$, that is to say two points in $\mathcal{P}_\la$ are connected with an edge if they are at a distance at most $\rho$. A point is isolated if it has degree zero. The number of isolated points, denoted 
$H_\la$, satisfies \eqref{dKhyp} and \eqref{dWhyp} with $\nu(W_\la) \leq \Gamma_d e^{(d-1)\la}$.
\end{coro}

 While rates of normal approximation in the $\mathbf{d}_W$ distance could be deduced from \cite{RSchulte}, the rate in the $\mathbf{d}_K$ distance is new.

\begin{proof}  We may write $H_\la$ as
\[
H_\la=\sum_{z \in \P_\la} \xi(z,\P_\la),
\]
where $\xi(z,\P_\la) = 1$ if $z$ is isolated and zero otherwise. 
Put $\xi^{[r]}(z,\P_\la)  = \xi(z,\P_\la \cap B_r(z)), r > 0.$
The radius of stabilization $R^\xi$ for $\xi$ is equal to $\rho$, i.e., the function $\tau(r)$ in
\eqref{eq:hypPsiCond} vanishes for $r \geq \rho$,
that is to say $\xi$ satisfies $BL(\tfrac{1}{240})$-localization in space and \eqref{eq:hypPsiCond} holds.  The fifth moment for $\xi^{[r]}$ is bounded by $1$ for all $r \in (0, \infty]$ and thus  $M^\xi_{5,\hyp} < \infty$.  It follows that  $H_\la$ satisfies \eqref{dKhyp} and \eqref{dWhyp} with $\nu(W_\la) \leq \Gamma_d e^{(d-1)\la}$. 
\end{proof}

\subsection{Upgrading qualitative CLTs to quantitative CLTs: Interacting diffusions}

Functionals on the Poisson space often satisfy asymptotic normality, {\em but lack good rates}.
We recall one such functional and show 
that it falls within the scope of our general results. We obtain via  Theorem~\ref{mainthm1} presumably optimal rates of 
normal convergence for Poisson functionals of interacting diffusions on random graphs, thus refining an existing qualitative central limit theorem   \cite{BYY} in the case of Poisson input.   The lack of good rates in the literature is a consequence of the fact that the relevant score function 
is not stabilizing and hence the quantitative CLTs in \cite{LSY} \cite{BM} do not apply. 

\vskip.3cm

\noindent{\bf{Interacting diffusions on  graphs on Poisson input}}.
By {\em admissible graph} on a locally finite  $V \subset \R^d$, we mean either the
geometric graph, the  $k$-nearest neighbors graph, the Voronoi tessellation,
the Delaunay tessellation, or the sphere of influence graph, as defined in Appendix A of \cite{BYY24}.  
Given a locally finite $V \subset \R^d$,  let $G$ be an admissible graph on $V$. 
Let $N_v$ be the neighborhood of $v \in V$, i.e.  set of the points in $V$ which are connected to  $v$ in $G$.
Fix a time-horizon $t_0 \in (1,\infty)$. 

As in \cite{BYY24}, consider the system of interacting $\R^{d'}$-valued diffusions
$M(v,t):=M^G(v,t), v \in V, \,  t \in [0, t_0]$ defined by
\begin{equation}
 \label{e:sys_ID_V}
 dM(v,t) = b(t,M(v,t),M(N_v,t))dt + \sigma(t,M(v,t),M(N_v,t))dZ_v(t), 
\end{equation}
where $Z_v(\cdot), v \in V,$ are i.i.d. standard Brownian motions in $\R^{d'}, d' \in \N$,
and where $M(N_v,t) = \{M(v',t)\}_{v' \sim v}$,
where $v' \sim v$ means that $v$  is a neighbor of $v'$.  Let $M(v):= M(v,0) \in \R^{d'}, v \in V,$ be the initial states.
For paths up to time $t_0$, these processes take values in the space ${\cal C}(t_0):= {\cal C}([0,t_0],\R^{d'})$,
the path space of continuous  
functions on $\R^{d'}$, equipped with the topology of uniform convergence on compact sets. 
The  functions $b$
 and $\sigma$ are the {\em drift} and {\em diffusion} coefficients, respectively. Both $b$
 and $\sigma$ have domain $\R^+ \times {\cal C}(t_0) \times \hat{\cal N}_{{\cal C}(t_0)}$, where $\hat{\cal N}_{{\cal C}(t_0)}$ denotes the finite subsets of ${\cal C}(t_0)$ \cite{BYY24}.
 The range of $b$ is $\R^{d'}$ whereas the range of the matrix-valued $\sigma$ is $\R^{d'} \times \R^{d'}$. 
  The particles interact directly only with their (finite) neighbors in the graph $G$. 

We further  assume  Lipschitz conditions on $b,\sigma$ as spelled out in Definition 9.2 of \cite{BYY24} and in 
Lacker et al. \cite{LRW}. When $G$ is admissible, the
Lipschitz assumption implies 
 the existence of a strong solution for the  system of interacting diffusions, as 
 established in \cite[Theorem 3.1]{LRW}. The works \cite{LRW}, \cite{BYY} 
 provide a rigorous discussion of such diffusions, they include precise assumptions, and they also address measurability issues. When the processes are defined on the entirety of the graph 
$G$, we denote them simply as $M := M^G$. The processes up to time $t$ are denoted by $M[v,t]$.

Let $\P$ be a Poisson point process on $\R^d$ with intensity  $\nu$ and consider the marked point process $\hat{\P} := \{(x,M(x),Z_x)\}_{x \in \P}$ on $\R^d \times \R^{d'} \times {\cal C}(t_0)$ with  $M(x) \in \R^{d'}$ the independent initial states, assumed to be a.s. uniformly bounded. 
To cast  statistics of interacting diffusions in the framework of
our general results, we consider for fixed $t_0> 0$ real-valued score functions of the trajectories (paths)
\begin{equation}\label{xisysID}
\xi ((x,M(x), Z_x), \hat{\P}) :=\xi^{(h,\G(\P))} ((x,M(x),Z_x), \hat{\P}):=  h(M^{\G(\P)}[x,t_0])
\end{equation}
where  $h : {\cal C}(t_0) \to \R$ is a Lipschitz($1$) function with respect to the sup-norm on ${\cal C}(t_0)$. The measurability of $\xi$ with respect to
$\hat{\P}$ is established in the proof of Theorem 7.8 of \cite{LRW} and also in Section 9 of \cite{BYY},  
where dependence on $t_0$ is suppressed.  Fix a window $W \subset \R^d$. Then such scores give  rise
to the diffusion statistics 
\begin{equation}
\label{sysidfunctional}
H:= \sum_{x \in \P \cap W} \xi ((x,M(x),Z_x),\hat{\P}).
\end{equation}
Note that $\xi$ does not satisfy stopping set stabilization. 
For all $r > 0$, define the short-range score $\xi^{[r]}$ by
$$
\xi^{[r]}((x,M(x),Z_x),\hat{\P}) = \xi^{(h,\G(\P \cap B_r(x))} ((x,M(x),Z_x), \hat{\P} \cap B_r(x)).
$$

\begin{coro} \label{t:sysID} (normal approximation for statistics of interacting diffusions) Let $\hat{\P} = \{(x,M(x),Z_x)\}$ be as
above, with $M(x)$ independent initial states which are a.s. uniformly bounded and $Z_x$ independent standard Brownian motions, and let
 $\G=\G(\cdot, \sim)$ be an admissible graph on $\P$.  Let $M = M^{\G(\P)}$ be the system of interacting diffusions as in \eqref{e:sys_ID_V} on $\G$ with  diffusion coefficients $b,\sigma$ satisfying the Lipschitz assumption of
 \cite{LRW}. If $\xi$  satisfies $M^\xi_{5,\R^d} < \infty$, 
 then $H$ satisfies the bounds at  \eqref{eq:dkMain} and \eqref{eq:dwMain}. 
\end{coro}

\begin{proof}  We will deduce this from Corollary~\ref{cor:spaceonly}. Put $X = \R^d$ and let the space of marks be $\M = \R^{d'} \times {\cal C}(t_0)$.
Section 9 of \cite{BYY} establishes  
that  $\xi$ is BL($\theta$)-localizing in space for all $\theta > 0$. 
This is a consequence of display (9.11) in \cite{BYY}, where, letting the point set in that display be the Poisson measure $\P$, it is shown that the scores on $\P$, when augmented by any finite number of points, satisfy $L^2$ stabilization, as defined in part~\eqref{condLq} of Subsection~\ref{3conditions}.  $L^2$ stabilization implies localization, as noted in Subsection~\ref{3conditions}.
In fact \cite{BYY} shows that $\psi$ satisfies the `fast-decreasing' property, meaning that it decays faster than any power. 
All  assumptions of 
Corollary~\ref{cor:spaceonly}
are satisfied and thus Corollary~\ref{t:sysID} follows. 
\end{proof}

\section{BL-localizing statistics in space-time} \label{Sectionspacetime}

Functionals $H$ of dynamic geometric models are expressed as sums of scores of space and time variables in a natural way.   When the time variable ranges over a bounded interval $[a,b]$ of $\R$ then one may apply the set-up of BL-localization in space 
(with a marked  Poisson point process $\P$ on
$\X$ and with marks in $[a,b])$ and appeal to Corollary \ref{cor:spaceonly} to establish normal approximation bounds for  $(H - \E H)/\sqrt{\Var H}$.

However when the time variable ranges over all of $\R$
this set-up does not apply, as $H$ a priori is not necessarily finite  as seen in Remark \ref{Hisfinite}.  Still, many functionals of interest consist of sums of scores which localize in space and in time, as in Definition \ref{def:space-time-loc}.  Two prime examples concern (i) the number of accepted particles 
in spatial birth-growth models with random unbounded particle speeds in infinite time, and
 (ii) extreme points in  Laguerre tessellations, or equivalently, shocks in the zero-viscosity solution to the Burgers PDE with random inital data.  In the following we take the underlying space to be $\R^d$, but the methods could be extended to study growth processes on more general metric spaces.

\subsection{Spatial birth-growth models}

For $\la \geq 1$, let $W_\la := [-\frac{1}{2} \la^{1/d},\frac{1}{2} \la^{1/d}]^d$ and denote by $\P_\la$ a Poisson process on $W_\la \times \R^+$ whose  intensity $\nu$ equals Lebesgue measure.
Let the generic points of $\P_\la$ be denoted by $(z,t_{z})$.

Fix a time horizon $t_0 \in (0, \infty]$ and note that we allow $t_0=\infty$.
Let $\tP_\la := \{(z,t_{z}, R_z)\}_{(z, t_z) \in \P_\la}$ where $t_{z}$  represent times at which particles arrive 
on the substrate $W_\la$ at location $z$, and where the marks $R_z, z \in \P_\la,$ are  i.i.d. random variables with values in $\MM := (0, \infty)$.
 Order the points in $\tP_\la$ according to the increasing order of their time coordinates, with the first point being that with the smallest time coordinate. 
This gives $\tP_\la
:= \{(z_i,t_{z_i}, R_{z_i})\}_{(z_i,t_{z_i})  \in \P_\la}$.
\vskip.1cm

Let $R_{i}:=R_{z_i}^{\emph{\tiny{birth-growth}}}$ denote the random
speed at which the particle (seed) at $z_i$ grows radially in all directions,
provided the seed is accepted.  The acceptance rule is as follows. 
The seed having smallest time coordinate is accepted and subsequently arriving seeds are accepted if they do
not belong to the union of the (growing) balls of
previously accepted seeds. A statistic of interest is
$H_\la:= H^{\text{birth-growth}}_\la$, the total number of seeds  in $\tP_\la$ which are accepted up to time $t_0$.  
When all seeds grow with the same constant  speed, then
one obtains the model of crystal growth
introduced by Kolmogorov and Johnson and Mehl in the 1930's.  We  do not restrict to  constant speeds, and to rule out trivialities arising from seeds with degenerate growth,  we assume that particle speeds
are a.s. bounded below by some small but positive constant $\rho_{\emph{\tiny{birth-growth}}}$.

\begin{coro} \label{IPS} (normal approximation  for the total number of particles accepted in  birth-growth models with unbounded random particle speeds)
We assume exponential decay of 
$R^{\emph{\tiny{birth-growth}}}$,
namely there is a constant $C>0$ such that
\begin{equation}\label{velocitydecay} 
\sup_{z \in \R^d} \PP(R_z \geq r) \leq \exp(-Cr), \ r > 0. 
\end{equation}
Fix $t_0 \in (0,\infty]$. Then  $H_\la$ satisfies the bounds
\[
\mathbf{d}_K \left(  \frac{H_\la - \E H_\la} {\sqrt{\Var H_\la}}, N \right) \leq C_K \cdot  \frac{\la^{1/2}}{\Var H_\la}
\]
and 
\[
\mathbf{d}_W \left(  \frac{H_\la - \E H_\la} {\sqrt{\Var H_\la}}, N \right) \leq C_K  \frac{\la^{1/2}}{\Var H_\la} + C_W  \frac{\la}{(\Var H_\la)^{3/2}},
\]
where $C_K,C_W$  are  constants depending only on $d$ and $C>0$. 
\end{coro}

For spatial birth-growth models where
particle growth speeds are governed by a deterministic possibly non-linear  function, the paper \cite{SY08} used  dependency graph methods to
establish a quantitative CLT for the number of accepted points; the rates of normal convergence in \cite{SY08}  have extraneous logarithmic factors. The case of  unbounded and random particle growth speeds has received scant attention even for finite time horizons.  This is due to the difficulty of controlling long-range interactions, which, to quote  the authors of 
\cite{BMT}: 
\begin{quote}
    `create[s] a causal chain of influences that
seems quite difficult to study with the currently known methods of stabilization for Gaussian
approximation'.
\end{quote}
\noindent The authors of \cite{BMT} circumvent this obstacle by restricting to finite-time horizon models and by assuming  
that a rejected seed could still block seeds arriving in the future, which is arguably less realistic. 

Here {\em we remove this assumption}, putting us in the framework of the classic birth-growth models.  We allow for {\em random unbounded  speeds of particle growth, infinite time horizons, and the rates of normal approximation
contain no extraneous logarithmic factors.}

Birth-growth models with unbounded, random growth speeds fit neither into the  framework of stabilization established in \cite{LSY} nor into the set-up of \cite{SY08}. This is because  the stabilization criteria in these papers require that adding or changing points outside the  stabilization ball has no effect on the score, whereas with unbounded growth speeds, it is always possible 
that a far away point may arrive arbitrarily close to time $t = 0$ while simultaneously having a very large growth speed, so that it covers the point at which the score is calculated, thus potentially affecting the score.

\vskip.3cm

\noindent{\bf Proof of Corollary \ref{IPS}}.
We will deduce Corollary~\ref{IPS} from Corollary~\ref{mainthmXisW} with $X = W = W_\la.$
As in \cite{BYY24}, we use a graphical construction which keeps track of the spatial dependencies
and which shows that the event that a particle falls within the interaction range of previously arrived particles  depends 
on the local data with high enough probability to yield the desired BL-space-time localization of scores.

Define the  score $\xi^{\text{birth-growth}}$ as follows:
\[
\xi^{\text{birth-growth}}((z,t_z,R_z), \tP_\la)
= \begin{cases}
1 &  \text{ if the particle at } z \in W_\la  \text{ at time}\  t_z \text{ is  accepted} \\
0 & \text{otherwise}.
\end{cases}
\]
We often write $\xi$
instead of $ \xi^{\text{birth-growth}}$. 
We use Corollary~\ref{mainthmXisW} 
to establish  a quantitative central limit theorem for   
$$
H_\la:= H_\la(\hat{\mathcal{P}}_\la) = \sum_{(z,t_z) \in \mathcal{P}\cap (W_\la \times \R^+)}  \xi((z,t_z,R_z),\hat{\P}_\la).
$$
We choose $\xi^{[r]}$ and $\xi^{(s)}$ to be the restricted scores as at \eqref{Restricted2}.  All scores under consideration are bounded by $1$ and thus $M_{5,\R^d}^\xi = 1$.  
It remains to show $BL(\tfrac{1}{240},\tfrac{1}{120})$ space-time localization of scores.  We will in fact show that $\xi$ satisfies $BL(\theta,\theta')$ space-time localization for all $\theta, \theta' > 0$.

\vskip.3cm
\noindent{\bf $BL$-localization of $\xi^{\text{birth-growth}}$ in time}.
Given any $s > 0$ and any marked set $\hat{{\cal A}}$
of cardinality at most six, 
in view of Remark~\eqref{condLq} in Subsection~\ref{3conditions},
it suffices that  the $L^1$ norm of the difference between $\xi((z,t_z,R_z), \tP_\la \cup \hat{{\cal A}})$
and $\xi^{(s)}((z,t_z,R_z), \tP_\la \cup \hat{{\cal A}})$ decreases exponentially fast in $s$, uniformly in $z$ and $\la$. First, note that  the score $\xi$ at $(z,t_z,R_z)$ is only influenced by points having smaller time coordinate. Thus for any point set $\chi \subset \R^d \times \R^+ \times \R^+$ it follows that
\[
\one(t_z<s)\xi((z,t_z,R_z),\chi) = \one(t_z<s) \xi((z,t_z,R_z),\chi \cap (\R^d \times [0,s] \times \R^+)) = \xi^{(s)}((z,t_z,R_z),\chi),
\]
where the last  equality follows by definition of $\xi^{(s)}$.
We have thus
\begin{align*}
& \quad \E|\xi((z,t_z,R_z), \tP_\la \cup \hat{{\cal A}}) - 
\xi^{(s)}((z,t_z,R_z), \tP_\la \cup \hat{{\cal A}})|
\\ 
&
=  \E|{\bf{1}}(t_z \geq s)\xi((z,t_z,R_z), \tP_\la \cup \hat{{\cal A}})| 
\\ 
&
\leq  \E|{\bf{1}}(t_z \geq s)\xi((z,s,R_z), \tP_\la \cup \hat{{\cal A}})| 
\\ 
&
\leq  \PP(\xi((z,s,R_z), \tP_\la \cup \hat{{\cal A}}) \neq 0), 
\end{align*}
where the penultimate inequality follows from stochastic domination, that is to say the probability of acceptance is non-increasing  with increasing time.  
Given $s > 0$, we show
$$
\sup_{\la \geq 1} \sup_{z \in W_\la} \sup_{\A \subset W_\la\times \R^+} \PP ( \xi((z,s, R_z), \tP_\la \cup \hat{\A}) \neq 0)$$
decays exponentially fast in $s$.  

If  $\xi^{\text{birth-growth}}( (z,s,R_z), \tP_\la \cup \hat{\A}) \neq 0$ then the seed at time $s$ is accepted, but this implies 
$\P$ puts no points in the right circular cone centered on the time-line with apex $(z,s)$ and base $(W_\la \cap B_{s\rho}(z)) \times \{0\}$. If $s\rho<\la$, there is constant $c=c(d)$ such that a fraction $c$ of the ball $B_{s\rho}(z)$ is always inside $W_\la$ and the volume of $W_\la \cap B_{s\rho}(z)$ is lower bounded by $cs^d\rho^d$. If $s\rho>\la$, then a fraction  $c'=c'(d)$ of $W_\la$ is always inside $W_\la \cap B_{s\rho}(z)$ and the volume of the intersection is lower bounded by $c'\la \geq c'$. There is a constant $c''>0$ such that the volume of the cone is thus lower bounded by the minimum of $c''s^{d+1}\rho^d$ and $c''s$. The probability of the event that $\mathcal{P}$ does not put any points into the cone is hence upper bounded by $C\exp(-c''s)$, for some constants $c'',C>0$.

\vskip.3cm
\noindent {\bf BL-localization of $\xi^{\text{birth-growth}}$ on space domains $W_\la$.} Showing that  $\xi^{\text{birth-growth}}$ satisfies BL-space localization over the  domains $(W_\la)_{\la \geq 1}$ is more delicate and is facilitated via a graphical construction as in \cite{BYY24}.
\vskip.3cm

\noindent{\bf Graphical construction.} For any $\rho > 0, z \in \R^d$, we let ${\Cyl}(z,\rho) \subset \R^d \times \R^+$ be the right circular cylinder $B_{\rho}(z) \times \R^+$.

\sloppy  Define the graph $G^{\text{birth-growth}}(\tP_\la)$ in $\R^d \times  \R^+ \times \R^+$ by joining  points  $(z_i, t_{z_i}, R_{z_i})$
and   $(z_j,t_{z_j}, R_{z_j})$,  $t_{z_i} < t_{z_j}$, 
with a directed edge from  
$(z_i, t_{z_i}, R_{z_i})$ to $(z_j,t_{z_j}, R_{z_j})$  if  $(z_j, t_{z_j}) \in \Cyl(z_i, (t_{z_j}-t_{z_i}) R_{z_i})$. This means that the point with smaller time coordinate potentially influences the point with larger time coordinate.
We  extend the above definition in the natural way to $G^{\text{birth-growth}}(\tP_\la \cup \tA)$, where $\tA$ is a marked point set in $\R^d \times \R^+ \times \R^+$ with cardinality at most six.

\vskip.3cm 
\noindent{\bf Backward clusters.}
Given  $G^{\text{birth-growth}}$
and $\tA$ a marked point set in $\R^d \times \R^+ \times \R^+$,
define 
$$
C(z,t_z):=  C(z,t_z,R_z):= C((z,t_z,R_z); \tP_\la \cup \tA)$$
 to be the {\em backward (in time) cluster} of $(z,t_z)$,  
i.e. $(z', t_{z'}, R_{z'}) \in C(z,t_z,R_z)$ if there exists a path in $G(\tP_\la \cup \tA)$  from $(z',t_{z'},R_{z'})$ to $(z,t_z,R_z)$.
The cluster contains all marked points which may potentially influence the status of a seed arriving at $z$. 
In general, the cluster has an unbounded spatial diameter. 

We first show  for all $(z_1,t_{z_1}),...,(z_4,t_{z_4}), (x,t_{x}), (y,t_y) \in W_\la \times [0, t_0)$ that $ \xi^{\text{birth-growth}}$ evaluated at $(z_1,t_{z_1}, R_{z_1})$
satisfies $BL(\theta)$-localization for all $\theta > 0$ provided
$\xi^{[r]}$ is chosen as  the restricted score as at \eqref{Restricted2}.  In view of \eqref{simple1} it suffices to show 
  that there is a $\psi: [0, \infty) \to [0,2]$ satisfying \eqref{integralpsi} for all $\theta >0$ and such that for 
  any $z_1,...,z_4,x,y \in W_\la$ and $t_{z_1},...,t_{z_4},t_{x},t_y \in \R^+$   we have uniformly in $\lambda\geq 1$
  and uniformly in $\mathcal{A} \subset \{(z_1,t_{z_1}),...,(z_4,t_{z_4}), (x,t_{x}), (y,t_y)\}$
  \begin{equation*}
  \PP(\xi((z_1,t_{z_1}, R_{z_1}), \tP_\la \cup \tA) \neq \xi^{[r]}((z_1,t_{z_1}, R_{z_1}), \tP_\la \cup \tA)) \leq \psi(r), \quad r  > 0.
\end{equation*}
 Let $m_r:=r^{\gamma}$ with $\gamma \in (0,1)$ a small constant to be chosen. 
We may assume that $t_{z_1} \leq m_r$ since if $t_{z_1} \geq m_r$ then by time-localization, both $\xi$ and $\xi^{[r]}$ vanish on an event with probability exceeding $1 - e^{-cm_r}$. In other words 
\begin{multline} \label{Lipstabmarksindicator}
    \sup_{\lambda\geq 1} \sup_{\substack{\mathcal{A} \subset \{(z_1,t_{z_1}),...,(z_4,t_{z_4}), (x,t_{x}), (y,t_y)\}}}
  \PP(\xi((z_1,t_{z_1}, R_{z_1}), \tP_\la \cup \tA) \\  \hskip3cm \neq \xi^{[r]}((z_1,t_{z_1}, R_{z_1}), \tP_\la \cup \tA)); \  t_{z_1} \geq m_r) \leq e^{-cm_r}.
\end{multline}

Therefore it is enough to show there is a
$\psi$ satisfying \eqref{integralpsi} for all $\theta >0$ and 
such that for any $z_1,...,z_4,x,y \in W_\la$ and $t_{z_1},...,t_{z_4},t_{x},t_y \in \R^+$ we have
\begin{align} \label{Lipstabmarksindicatorcompl}
     & \sup_{\lambda\geq 1} \sup_{\substack{\mathcal{A} \subset \{(z_1,t_{z_1}),...,(z_4,t_{z_4}), (x,t_{x}), (y,t_y)\}}}
  \PP(\xi((z_1,t_{z_1}, R_{z_1}), \tP_\la \cup \tA) \nonumber \\ & \hskip3cm \neq \xi^{[r]}((z_1,t_{z_1}, R_{z_1}), \tP_\la \cup \tA); \  t_{z_1} \leq m_r ) \leq \psi(r).
\end{align}

\vskip.3cm 

\noindent{\textbf{Diameter of the backward cluster.}} To show BL-localization in space, we control the diameter of the backward (in time) clusters, extending and generalizing the  approach taken in \cite{BYY24}. 
The diameter of the backward clusters at $\hat{z}_1=(z_1,t_{z_1},R_{z_1})$ is given by  
$$
D(\hat{z}_1, \tP_\la \cup \tA):= \max_{(y,t_y) \in C(z_1, t_{z_1})} |y - z_1|.
$$
The score $\xi((z_1,t_{z_1}, R_{z_1}), \tP_\la \cup \tA)$ is determined by 
the restriction of $\tP_\la \cup \tA$
to the ball centered at $z_1$ and with radius equal to $D(\hat{z}_1, \tP_\la \cup \tA)$, as the points in $\tP_\la \cup \tA$ outside this ball do not contribute to the score.
Thus, {\em given  $t_{z_1} \in (0,m_r]$,} it  suffices to show 
for any $\mathcal{A} \subseteq \{(z_1,t_{z_1}),...,(z_4,t_{z_4}),(x,t_{x}),(y,t_y)\}$ that 
\begin{equation}
\label{expbound}
\PP(D(\hat{z}_1, \tP_\la \cup \hat{\A}) \geq r) \leq C \psi(r), \quad r > 0,
\end{equation}
where $\psi$ satisfies \eqref{integralpsi} for all $\theta >0$.

The cylinder in $\R^d \times [0, m_r]$ centered at $z_1$ with radius
$D(\hat{z}_1, \tP_\la \cup \tA)$ contains the backward (in time) cluster
at $\hat{z}_1$ and thus contains all  marked points which may potentially determine the score at $\hat{z}_1$. 
We want to show that there is a high probability event such that conditional on this event, the configurations of subsets of 
$\tP_\la \cup \tA$ which may potentially determine the score  at $\hat{z}_1$ are
contained within $B_r(z_1) \times [0, m_r] \times \R^+$.  An overview of the argument goes as follows. 

Recall $m_r:=r^{\gamma}$ with $\gamma \in (0,1)$ a small constant to be chosen. Recall also that we are working in the regime $t_{z_1} \in (0, m_r]$. 
{\em{By (spatial) interaction range}} of a particle $(z,t_z,R_z)$ we mean the extent of
spatial growth of the particle in the time interval $(t_z, m_r]$, which is generously upper bounded by $m_r R_z$. 

The exponential decay of $R$ at \eqref{velocitydecay} implies that there 
is a high probability event $E_{\la,r}(z_1)$ guaranteeing  that  particles
in $\tP_\la \cup \hat{\A}$ and belonging to $B_{r}(z_1) \times [0,m_r] \times \R^+$ have 
 interaction range at most $m_r^2$ whereas particles in  
 $B_{r}^c(z_1) \times [0,m_r] \times \R^+$ 
 have interaction ranges which are  small enough not to have a large influence on the inside of $B_r(z_1)$. Conditional on this event, we show that it is unlikely to have a large backward cluster exiting $B_r(z_1)$, because this would imply the existence of a chain of interactions containing at least $r/(14m_r^2)$ points with decreasing time coordinates as counted from the center outwards.
The details go as follows. 

\vskip.3cm
\noindent{\bf Definition of the high probability event}.   
 Let  $E_r(z_1) :=E_{\la,r}(z_1), z_1 \in \R^d,$ be the intersection of the following three events, which effectively upper bounds  interaction ranges and allows us to control the size of the backward cluster on this event:
 \begin{enumerate}[(i)]
     \item  $
 (z,t_z, R_z)  \in \{ \tP_\la  
 \cap (B_r(z_1) \times [0,m_r] \times \R^+)\} \Rightarrow R_z \leq m_r.
 $
    \item $(z,t_z, R_z)  \in \{ \tP_\la  
 \cap (B_{r}^c(z_1) \times [0,m_r] \times \R^+)\}$
 $\Rightarrow m_rR_z \leq \max \left(m_r^2, d(z,\partial B_{r}({z_1}))\right).$ 
    \item $(z,t_z, R_z) \in \tA  \cup \{(z_1,t_{z_1},R_{z_1})\} \Rightarrow R_z \leq m_r.$
 \end{enumerate}
 On the event $E_r(z_1)$, we are assured that any edges incident to points outside $B_r(z_1)$ in the backward cluster having an influence on points inside $B_r(z_1) \times [0,m_r]$ do not reach $B_{r- m_r^2}(z_1) \times [0,m_r]$.
 To influence the state of the particle at $z_1$, there thus needs to be a long chain of edges shorter than $m_r$ linking points outside $B_r(z_1)$ to $z_1$.

First we show that  $\PP(E_r(z_1)^c)$ decays exponentially fast in $r$. Let the spatial coordinates of the points in $\tA$ be denoted by $w_1,...,w_6$.
Applying the union bound and the Mecke formula  we have  
\begin{align*}
\PP(E_r(z_1)^c) & \leq \int_0^{m_r}\int_{B_{r}(z_1)} \PP(R_{z} \geq m_r) dz dt\\
& +  \int_0^{m_r} \int_{B_{r}^c( z_1)} \PP(R_{z} \geq  \max(m_r, d(z,\partial B_{r}({z_1}))/m_r)) dz dt \\
&+ \sum_{i = 1}^6 \PP(R_{w_i} \geq m_r) +\PP(R_{z_1}\geq m_r).
\end{align*}
The exponential decay of $R_{.}$
at \eqref{velocitydecay} implies that all four terms decay exponentially fast in $r$. 
The first integral is  $O(m_rr^{d}\exp(-C m_r))$, which decays exponentially fast. The sum of the last two terms decays as $O(\exp(-C m_r))$.
To show the decay of the second integral,  we may assume without loss of generality that $z_1 = \0$ and abbreviating $B_{r}( z_1)$ by $B_{r}$ we have 
\begin{align*}
&\int_{B_{r}^c} \PP\left(R_{z} \geq  \max(m_r, d(z,\partial B_r)/m_r)\right) dz
\\
&
= \int_{B_r^c} {\bf 1}(d(z,\partial B_r) \leq m_r^2) \exp(-C m_r)dz
\\
&
+ \int_{B_r^c} {\bf 1}(d(z,\partial B_r) >  m_r^2) \exp(-C d(z,\partial B_r)/m_r)dz.
\end{align*}
The first integral is  $O(m_r^2 r^{d-1}\exp(-C m_r))$, which decays exponentially fast, since
 $m_r=r^{\gamma}$.  
 Let $\kappa_d$ be the volume of the unit ball in $\R^d$.
Evaluating the second integral in polar coordinates gives 
\begin{align*}
\quad & \int_{B_r^c} {\bf 1}(d(z,\partial B_r) >  m_r^2) \exp(-C d(z,\partial B_{r})/m_r)dz
\\
&
= \int_{r}^{\infty} \int_{\S^{d-1}}{\bf 1}(d(sw, \partial B_{r}) \geq m_r^2) 
\exp(-C(s - r)/m_r) s^{d-1} dw ds  
\\
&
=  \int_{r+ m_r^2}^{\infty} \int_{\S^{d-1}}
\exp(-C(s - r)/m_r) s^{d-1} dw ds  \\ &
= d \kappa_d \int_{r+ m_r^2}^{\infty} 
\exp(-C(s - r) /m_r) s^{d-1} ds  \\ &
= d\kappa_dm_r \int_{m_r}^\infty  \exp(-Cu) (m_ru+r)^{d-1} du  
\\ &
= O(m_r^dr^{d-1}\exp(- C m_r)).
\end{align*}
Thus $\PP(E_r(z_1)^c) = O(\max(m_r^d,m_r^2) r^{d} \exp(- C m_r)).$ 
This estimate holds uniformly over any $\tA$.

\vskip.3cm
\noindent{\bf Controlling the size of the backwards cluster on the high probability event  $E_r( z_1)$}.
Conditioning on $E_r( z_1)$ and its complement $E_r^c( z_1)$ yields
\begin{align}
\label{bound1}
\PP(D(\tz_1, \tP_\la \cup \hat{\A}) \geq r) 
 \leq  \PP(\{D(\tz_1,\tP_\la \cup \hat{\A}) \geq r \} \cap E_r( z_1)) + \PP(E_r^c( z_1)).
\end{align}

We show that the first  term on the right-hand side of \eqref{bound1} is also decreasing exponentially fast in $r$.
Let $\A' \subset \A$ be the (possibly empty) subset of points in $\A$ belonging to $B_{r}( z_1)$. 
When all edges of  $G(\tP_\la \cup  \hat{\A'})$ incident to points in $B_{r}( z_1)$ have spatial length (interaction range) less than  $m_r^2$,  then one needs at least $\ell'_r:=\lfloor (r-m_r^2)/(2m_r^2) \rfloor $ distinct sites 
$\{y'_i:i=1,\ldots,\ell'_r\}$ in $\tP_\la \cup \hat{\A'}$ to cover the distance from $ z_1$ to $B^c_{r-m_r^2}( z_1)$ on $G(\tP_\la \cup  \hat{\A'})$. 
Among these sites $\{y'_i\}$ there might be the fixed atoms $w_1,\ldots,w_6$, with each atom contributing at most $2m_r^2$ to the length of the chain; after  removing them, and now putting 
$$
\ell_r :=\lfloor (r-m_r^2)/(14 m_r^2) \rfloor,  \quad r \in\R^+,$$ 
one  finds that there there exist {\em distinct} sites $y_1,\ldots,y_{\ell_r}\in\P_\la \cap B_{r}( z_1)$,
{\em different from} $w_1,\ldots,w_6$, and such that $|y_{i}-y_{i-1}|< 14 m_r^2$, 
$i=1,\ldots,\ell_r$, where we set $y_0:= z_1$.

Looking backwards in time
from time $t = t_0$, and enumerating the arrival times of the particles at $y_i$ in 
reverse chronological order we note
$t_0 \ge  T_1> T_2>\ldots> T_{\ell_r}>T_{\ell_r+1}=0$.

Considering the path $\{y_i:i=1,\ldots,\ell_r\}$, using the Markov inequality
and the Mecke formula one  bounds the considered probability by
\begin{align*}
& \quad \PP (  \{ D(\tz_1, \tP_\la \cup \tA) \geq r \} \cap E_r( z_1)  )  \\
& \leq \E \bigg[ \sum_{(y_1,\ldots,y_{\ell_r})  \in \P_\la^{\ell_r} }
\prod_{i=1}^{\ell_r}  {\bf{1}}{( |y_{i} - y_{i-1}|< 14 m_r^2)} {\bf{1}}(T_{i}>T_{i+1}) \bigg]  \\
& \leq  \int_{  (\R^d)^{\ell_r} }  \prod_{i=1}^{\ell_r} {\bf{1}}(|y_{i}-y_{i-1}|< 14 m_r^2) d y_1... dy_{\ell_r}
      \times \E \Bigl[ 
\prod_{i=1}^{\ell_r} {\bf{1}}(T_{i} > T_{i+1}) ] \Bigr] 
\end{align*}
where we use the independence of the spatial locations of seeds and their arrival times.   
Note 
\begin{align*}
\E \Bigl[
\prod_{i=1}^{l_r}{\bf{1}}(T_i> T_{i+1}) \Bigr] 
\le \int_0^{m_r} d s_{l_r}\int_{s_{l_r}}^{m_r} d s_{l_r-1}\dots
\int_{s_3}^{m_r} ds_{2}\int_{s_2}^{m_r} ds_{1}
=\frac{m_r^{l_r}}{l_r!}.
\end{align*}
 
It follows that uniformly in $\tA$ one obtains
\begin{align}
 \PP \bigl( \{D(\tz_1,\tP_\la \cup \tA) \ge r \} \cap E_r( z_1) \bigr) 
&
\le \bigl(\kappa_d(14m_r^2)^d \bigr)^{\ell_r} \frac{m_r^{\ell_r}}{\ell_r!} \leq C e^{-cr^{1 - 2\gamma}(1-2d\gamma - 3\gamma)}
\label{bound3}
\end{align}
Put $\gamma \in (0, 1/(3 + 2d))$. The choice of $\gamma$ and hence $m_r$ implies that the right-hand side of \eqref{bound3} is at most $c_1 \exp(-c_2 r^b)$  for some $c_1, c_2, b > 0$.   Thus  the backwards cluster at $z_1$ extends outside the ball $B_r(z_1)$ with an exponentially small
(in $r$) probability.  This establishes 
\eqref{expbound}, completing the proof of Corollary \ref{IPS}.   \qed 

\vskip.3cm 
\noindent{\bf Remark.}  The above proof readily shows that one may refine the normal approximation results for the number of accepted seeds in the   spatial birth-growth models in \cite{SY08}.  The (non-random) speed of particle growth in these models is regulated by a deterministic function of power type and the number of particles accepted satisfies a quantitative CLT 
with a rate of convergence in the $d_K$ distance containing logarithmic terms as in display (3.25) of \cite{SY08}. These logarithmic terms can be easily removed 
by combining the results of \cite{SY08} with  Corollary~\ref{mainthmXisW}. This goes as follows.  Let $\xi$ be the score  which  keeps track of accepted seeds.  Then 
$BL$-localization of $\xi$ follows from Lemma 3.1 in \cite{SY08}, which shows that there exists a family of short-range scores $(\xi^{[r]})_{r > 0}$ such that $\xi$ satisfies  $BL$-localization in space for any $\theta > 0$, 
and in fact $\psi$ can be set to the exponential function.  It is also shown that  $\xi$ satisfies  $BL$ localization in time, as Lemma 3.2 in \cite{SY08} shows that
$$
\sup_{\la \geq 1} \sup_{z \in W_\la} \sup_{\A \subset W_\la\times \R^+} \PP ( \xi((z,s, R_z), \tP_\la \cup \hat{\A}) \neq 0)$$
decays exponentially fast in $s$.
Corollary~\ref{mainthmXisW} immediately applies and establishes Berry-Esseen bounds for the total number of accepted particles.

\subsection{Laguerre tessellations and  the stochastic Burgers PDE}

In this section, we study a counting statistic $N_\la$  relevant in the context of Laguerre tessellations and the stochastic Burgers PDE. It is defined as follows. Let $\P$ be a Poisson process on $\R^d \times \R$ of intensity measure $dx \otimes \mu(dh)$, where $\mu$ is a measure on $\R$. Fix $t>0$. Given a point $(x,h)\in\R^d \times \R$, we define the upward and downward paraboloids at $(x,h)$:
\begin{align*}
    &\Pi^{\uparrow}[(x,h)] := \left\{ (z,h_z) \in \R^d \times \R: h_z \geq h + \frac{|x-z|^2}{2t} \right\} \\
    &\Pi^{\downarrow}[(x,h)] := \left\{ (z,h_z) \in \R^d \times \R: h_z \leq h - \frac{|x-z|^2}{2t} \right\}.
\end{align*}
We construct a thinning $\P^{\text{thin}}$ of $\P$ by including a point $(x,h)\in \P$ in $\P^{\text{thin}}$ if and only if the boundary of the upwards paraboloid $\Pi^{\uparrow}[(x,h)]$ is not included in the union of all  upwards paraboloids centered at points of $\P\setminus\{(x,h)\}$, i.e. for any $(x,h) \in \P$,
\[
(x,h) \in \P^{\text{thin}} \Leftrightarrow \partial\Pi^{\uparrow}[(x,h)] \not\subseteq \bigcup_{(z,h_z) \in \P\setminus\{(x,h)\}} \Pi^{\uparrow}[(z,h_z)].
\]
This condition is equivalent to requiring that the entire paraboloid $\Pi^{\uparrow}[(x,h)]$ is not covered by the union of all other paraboloids. 
For any $\la \geq 1$, we define observation windows
 as $W_\la := \left[-\frac{1}{2}\la^{1/d},\frac{1}{2}\la^{1/d}\right]^d$. The statistic $N_\la$ is now given by $N_\la := \P^{\text{thin}} \cap (W_\la \times \R)$, the number of points of the thinned process $\P^{\text{thin}}$ inside the observation window $W_\la$. $N_\la$
depends on $t$ but we will usually suppress this dependence for notational convenience. Abusing notation we represent points in $\P_\la$ by $(z,h_z)$ instead of $(z,t_z)$ so as to avoid a notation clash with the time parameter $t$.

\vskip.3cm
\noindent\textbf{Laguerre tessellations.} 
The union of the paraboloids which are not entirely covered has a lower boundary, termed the paraboloid process,   given by
the variational formula
$$
F(t,w) = \inf_{(z,h_{z}) \in \P }\left(h_{z} + \frac{|z - w|^2}{2t} \right)
$$
where the $\inf$ runs over all points of $\P$, the  
 Poisson point process on $\R^{d} \times \R$. 
 The union of all the points $w \in \R^d$ for which the infimum is realized in $(z,h_{z})$ gives rise to the {\em Laguerre cell} generated by $(z,h_{z})$, namely 
 $$
C((z,h_z), \P):= \left\{ w \in \R^{d}: \frac{|z-w|^2}{2t}+h_z \leq \frac{|z'-w|^2}{2t}+h_{z'} \ {\rm{for \ all} } \
(z',h_{z'}) \in \P\right\}.
$$
Note that this cell is not necessarily non-empty.  It is also a curious geometric observation that a cell center need not belong to the cell which it generates. The collection of non-empty cells thus generated 
is a stationary random tessellation known as  the  {\em Laguerre tessellation} of $\R^d$ induced by $\P$. 
Note that $(x,h) \in \P^{\text{thin}}$ if and only if $x$ is a 
cell center of a non-trivial cell. 
Hence the statistic $N_\la$ counts the number of points $(z,h_z) \in \P \cap (W_\la \times \R)$ which generate non-trivial cells; note that while the cell center must lie in $W_\la$ the same need not be true for the cell itself.

When the measure $\P$ has density growing like  $h^{\beta}$ in the height parameter $h$ with $\beta \in (-1, \infty)$,
(respectively   $(-h)^{-\beta}$ with $\beta \in ( (d+1)/2, \infty)$; respectively $e^{h/2}$) on $\R^{d} \times [0, \infty)$ (respectively on $\R^{d} \times (-\infty, 0]$; respectively on $\R^{d} \times \R$), then the dual to the Laguerre tessellation  is known as the
$\beta$-Delaunay tessellation (respectively $\beta'$-Delaunay tessellation; respectively Gaussian tessellation).
The dual tessellations were  introduced and studied in \cite{GKT}.

When the weights  $h_z, z \in \P$, are a.s. bounded, then the geometry of the Laguerre tessellations is locally determined \cite{FPY}.
However if the weights are not bounded, as in the three afore-mentioned tessellations, then  
the geometry is not in general locally determined. 
On the other hand,  as shown in Lemma 4 of \cite{GKT},
the $\beta$, $\beta'$ and Gaussian-Delaunay tessellations have a geometry which localizes
in space in the sense that with high probability its structure on a cylinder
${\rm{Cyl}}(0,r) := B_d(0,r) \times \R \subset \R^{d} \times \R$ of fixed radius  is not affected by  points in $\P$ lying 
far enough away from the cylinder. 
This high probability localization is enough to 
show that geometric characteristics of the
Laguerre tessellation, such as the number count $N_\la$ of cell centers in $W_\la$, satisfies 
mixing conditions (called absolute regularity in \cite{GKT}) ensuring a qualitative central limit theorem.
 
\vskip.3cm

\noindent{\bf Stochastic Burgers PDE.} 
The statistics $N_\la, \la \geq 1,$ also arise
when considering statistical properties of the zero-viscosity solution
to the Burgers PDE, which in
general dimension $d$ is given by
\be \label{BurgersPDE}
v_t + v \nabla v = \epsilon_{\la} \nabla^2 v,
\ee
with viscosity  $\epsilon_{\la} \to 0$, and  initial data 
 $v(0,w)= -2 \epsilon_{\la} \nabla h$,
 with $h(w)$ standing for the random initial potential at $w$.  
 
 Equation \eqref{BurgersPDE} models turbulence and a standard simplifying assumption (see \cite{AMS}) is to let the initial potential be zero-range Poissonian shot noise, i.e.   $e^{v(0,w)} = \sum_{(z,h_z) \in \P} e^{h_z} \delta(w - z)$,
where $\P$ is a Poisson point process on $\R^d \times \R$.  This corresponds to the  case where there is no interaction between  potentials at distinct points of the  ensemble $\P$, see Albeverio et al. \cite{AMS}. 

 The Burgers festoon $\Phi(\P)$ is the hypersurface in  $\R^d \times \R$ which is the 
boundary of the union $U \subset \R^d \times \R$   of translates of the down paraboloid $\{(z,h_z) \in  \R^d \times \R: h_z \sl -|z|^2/2t\}$, and having the property that $U$ contains no points of $\P$ and $U$ is maximal in the sense that it is not contained
in any other such union.
The graph of the hypersurface $\Phi(\P)$ is given by the variational formula 
$$
H(t,w) = \sup_{(z,h_{z})\in\R^d \times \R}\left(h_{z} - \frac{|z - w|^2}{2t} \right)
$$
where the sup runs over all $(z,h_{z})$ such that $\Pi^{\downarrow}[(z,h_{z})] \cap \P = \emptyset$.   
$H(t,w)$ is known as the `hull process'  \cite{CSY}, it is a concatenation of faces of downwards parabolas,  and it plays an important role in the construction of the geometric solution to \eqref{BurgersPDE} as well in establishing variance asymptotics for the number of vertices of convex hulls of i.i.d. samples \cite{CSY}. 
The solution $v(t,w)$ to \eqref{BurgersPDE} is the space derivative $\frac{\partial H}{\partial w}(t,w)$ when the ensemble $\P$ can be realized 
as a marked point set $\{(z_i,M_i)\}$ with the random variables $M_i$ satisfying distribution and integrability conditions \cite{AMS}. When $d = 1$, the velocity is the saw-tooth curve 
$$
v(t,w) \sim \frac{w - z_i^*}{t},
$$ 
where given $(t,z)$,  $z_i^*:=z_i^*(t,z) \in \{z_i\}$ is the point where the collection
$\{h_z -\frac{|z-w|^2}{2t} ;\ (z,h_z) \in \Phi(\P)\}
$
attains its maximum.

More generally, if the random initial potential is rapidly oscillating then only its local maxima determine $v(t,w)$, as seen in Burgers \cite{Burgers}; when the location of these maxima are well approximated by the maxima of Poissonian shot noise, the zero viscosity solution converges to the space derivative $H_w(t,w)$; see 
 \cite{Ba, MSW95}. 

 Of special importance is the  duality between the hull process and the paraboloid process $F(t,w)$.  The  {\em vertices of the hull process} are the points of the Poisson point process belonging to $H(t,w)$.
 As seen in the discussion in Section 3 of \cite{CSY} and especially the identity  (3.17) ibid, 
 the duality relation between the hull and paraboloid processes yields that 
 the vertices of the hull process coincide with the thinned points 
 $\P^{\text{thin}}$. Thus $N_\la$ also counts the number of vertices of the hull process in $W_\la$.

When $d = 2$,  $N_\la$ is the number of shocks to the zero viscosity solution $\partial H/\partial w$ which belong to $W_\la$, whereas in 
$d = 3$, one may interpret $N_\la$ as  modeling the number of cosmic voids in the window $W_\la$ in the turbulence model of 
the evolution of the universe \cite{Ba}, whenever the  location of the maxima of the initial potential are well approximated by the maxima of shot noise.  

\begin{coro} \label{Burgers} (normal approximation of the count statistics $(N_\la)_{\la \geq 1}$)  
Let the intensity density of $\P$ be $
dx d\mu(h) = dx h^{\beta}\ind{h>0}dh$ for some $  \beta > -1$. Then for each $t>0$, the count statistic $(N_\la)_{\la \geq 1}$ satisfies the normal approximation bounds
\[
\mathbf{d}_K \left(  \frac{N_\la - \E N_\la} {\sqrt{\Var N_\la}}, N \right) \leq C_K \cdot  \frac{\la^{1/2}}{\Var N_\la}
\]
and 
\[
\mathbf{d}_W \left(  \frac{N_\la - \E N_\la} {\sqrt{\Var N_\la}}, N \right) \leq C_K  \frac{\la^{1/2}}{\Var N_\la} + C_W  \frac{\la}{(\Var N_\la)^{3/2}},
\]
where $C_K,C_W$  are  constants depending only on $d,\beta$ and $t$.
\end{coro}

This corollary establishes 
proximity bounds to the normal, adding to the  the qualitative central limit theorem of  Theorem 7 of \cite{GKT},  which does not attempt to find rates, and it also improves upon the rates in Theorem 2.1 of \cite{SY08} (the case $\alpha = 2$), which contains extraneous logarithmic factors and which assumes $\beta \geq 0$.  
During the course of writing this paper we learned of  the preprint of Bhattacharjee and Gusakova \cite{BG} which
establishes Corollary \ref{Burgers}; it  
employs the theory of region-stabilizing scores and hence the approach does not overlap ours.

\vskip.3cm
\noindent{\bf Proof of Corollary \ref{Burgers}}.
  We express $N_\la$ as a sum of score functions defined on the point process $\P$ and apply Corollary~\ref{mainthmXisW} with $X = W = W_\la$ in conjunction with 
  Remark~\eqref{XisRd} in Subsection~\ref{usersguide}. In what follows, we set $t=1/2$ for ease of presentation, though when $t\neq 1/2$, the only change amounts to all constants henceforth depending on $t$.

We put $\xi((z,h_z), \P) = 1$ if $(z,h_z)$ belongs to the festoon $\Phi(\P)$, i.e. if it is  retained in the thinning $\P^{\text{thin}}$, otherwise the score is zero. More precisely, we define $\xi$ as
\[
\xi((z,h_z),\P) :=
\begin{cases}
    1 & \exists \ (w,h_w) \in \partial \Pi^{\uparrow}[(z,h_z)], \text{ s.t. } \Pi^{\downarrow}[(w,h_w)] \cap (\P \setminus \{(z,h_z)\}) = \emptyset \\
    0 & \text{otherwise}.
\end{cases}
\]
Note the duality relation $(w,h_w) \in  \Pi^{\uparrow}[(z,h_z)] \Leftrightarrow (z,h_z) \in \Pi^{\downarrow}[(w,h_w)]$. It follows that a point $(z,h_z) \in \P$ is such that $\partial\Pi^{\uparrow}[(z,h_z)]$ is not covered by $\bigcup_{(x,h_x) \in \P \setminus \{(z,h_z)\}} \Pi^{\uparrow}[(x,h_x)]$ if and only if there is a point $(w,h_w) \in \partial\Pi^{\uparrow}[(z,h_z)]$ such that $\Pi^{\downarrow}[(w,h_w)] \cap \P\setminus\{(z,h_z)\}$ is empty.
Now it holds that
\[
N_\la = \sum_{(z,h_z) \in \P_\la} \xi((z,h_z),\P).
\]

In what follows we do not put $\xi^{[r]}$ to be the space-restricted score as at \eqref{Restricted2} but instead use a slight modification as follows. With $B(y,r)$ denoting the $d$-dimensional ball centered at $y \in \R^{d}$ with radius $r \in (0, \infty)$, we denote by ${\rm{Cyl}}(y,r)$ the cylinder $B(y,r) \times \R$. Define space-localized scores short-range scores
\[
\xi^{[r]}((z,h_z),\P) :=
\begin{cases}
    1 & \exists \ (w,h_w) \in \partial\Pi^{\uparrow}[(z,h_z)] \cap {\rm{Cyl}}(z,r), \\
    &\qquad \text{s.t. } \Pi^{\downarrow}[(w,h_w)] \cap {\rm{Cyl}}(z,r) \cap (\P \setminus \{(z,h_z)\}) = \emptyset \\
    0 & \text{otherwise}.
\end{cases}
\]
The difference between this definition of $\xi^{[r]}$ and the one given in \eqref{Restricted2} is that the former only considers whether the intersection of paraboloids and a radius $r$ cylinder is covered, whereas the latter requires that entire paraboloids are covered. For the time-restricted score, we use $\xi^{(s)}$ as in \eqref{Restricted2}.

The scores and their short-range counterparts are all bounded by $1$ and thus $M_{5,\R^d}^\xi = 1$.  
It remains to show $BL(\tfrac{1}{240},\tfrac{1}{120})$ space-time localization of scores.
We will in fact show that $\xi$ satisfies $BL(\theta,\theta')$ space-time localization for all $\theta, \theta' > 0$.
  
\vskip.3cm
\noindent{\bf $BL$-localization of $\xi$ in time.}
  We show  that the scores satisfy $BL(\theta')$ time-localization for all $\theta' \sg 0$, and in fact 
\begin{equation}\label{dBLtimeBurger}
    \PP \bigg( \xi((z,h_z),{\mathcal{P}} \cup {\mathcal{A}}) \neq
    \xi^{(s)}((z,h_z), {\mathcal{P}} \cup {\mathcal{A}}) \bigg) \leq \phi(s), \quad s \geq 0,
    \end{equation}
where $\phi$ decays exponentially fast uniformly over point sets ${\mathcal{A}}$ of cardinality six.
Without loss of generality we may assume $s \geq 4$, as we may take $\phi(s) = 2$  for $0 \leq s \leq 4$. 
By stationarity we may take 
$(z,h_z) = (\0, h_{\0})$. 

Suppose $h_\0 \geq s$. Then $\xi^{(s)}(({\bf 0},h_\0),\P \cup \A) = 0$  and we prove that $\PP(\xi(({\bf 0},h_\0),\P \cup \A) \neq 0)$ decreases exponentially fast in $h_\0$ and hence also in $s$. Assume $\xi(({\bf 0},h_\0),\P \cup \A) \neq 0$. Then there is a point $(w,h_w) \in \partial\Pi^{\uparrow}[({\bf 0},h_\0)]$ such that $\Pi^{\downarrow}[(w,h_w)] \cap (\P\cup\A)\setminus \{(\0,h_\0)\} = \emptyset$. Note that for any point $(w,h_w) \in \Pi^{\uparrow}[({\bf 0},h_\0)]$, at least half of $\Pi^{\downarrow}[({\bf 0},h_\0)]$ is included inside $\Pi^{\downarrow}[(w,h_w)]$, hence half of $\Pi^{\downarrow}[({\bf 0},h_\0)]$ does not contain any points of $\P$. The inclusion can be seen e.g. assuming without loss of generality that $w=(a,0,...,0)$ with $0 \leq a \leq \sqrt{2(h_w-h_\0)}$ and checking analytically that any point $(x,h_x) \in \Pi^{\downarrow}[({\bf 0},h_\0)]$ with $x_1\geq 0$ also belongs to $\Pi^{\downarrow}[(w,h_w)]$.
The volume of $\Pi^{\downarrow}[({\bf 0},h_\0)]$ is given by $c_{\Pi^{\downarrow}}h_\0^{\beta+1+\frac{d}{2}} \geq c_{\Pi^{\downarrow}}h_\0^{\frac{d}{2}}$, where $c_{\Pi^{\downarrow}} = \frac{d\kappa_d\sqrt{2}^{d-2}}{\beta+1}B(\frac{d}{2},\beta+2)$ and where $B$ is the Beta function. It follows that $\PP(\xi(({\bf 0},h_\0),\P \cup \A) \neq 0) \leq \exp(-\frac{1}{2}c_{\Pi^{\downarrow}} h_\0^{d/2}) \leq \exp(-\frac{1}{2}c_{\Pi^{\downarrow}}s^{d/2})$.    

Next, suppose  $0 \leq h_\0 \leq s$. If
$\xi((\0,h_\0), ({\mathcal{P}} \cup {\mathcal{A}}) \cap (\R^d \times [0,s)))$ vanishes then so must $\xi((\0,h_\0),{\mathcal{P}} \cup {\mathcal{A}})$.
On the other hand, suppose  $\xi((\0,h_\0), ({\mathcal{P}} \cup {\mathcal{A}}) \cap (\R^d \times [0,s))) \neq 0$,  
but $\xi((\0,h_\0),{\mathcal{P}} \cup {\mathcal{A}}) =0$. In this case, there is a point $(w,h_w) \in \partial \Pi^{\uparrow}[(\0,h_\0)]$ such that $\Pi^{\downarrow}[(w,h_w)]$ does not contain any point of $(\P \cup \A) \cap (\R^d \times [0,s))$ other than $(\0,h_\0)$, but it must contain a point of $(\P \cup \A) \cap (\R^d \times [s,\infty))$. This implies in particular that $h_w \geq s$. Partition $\R^d \times [s,\infty)$ into unit volume cubes and let $q_1,q_2,...$ be an enumeration of those cubes having non-empty intersection with $\partial \Pi^{\uparrow}[(\0,h_\0)] \times [s,\infty)$. If there is a point $(w,h_w) \in q_i$ such that $\Pi^{\downarrow}[(w,h_w)]$ does not contain any point of $\P \setminus \{(\0,h_\0)\} \cap (\R^d\times [0,s))$, then a constant proportion of $\Pi^{\downarrow}[(w_{q_i},h_{q_i})]$ also does not contain any point of $\P \setminus \{(\0,h_\0)\} \cap (\R^d \times [0,s))$,
where $(w_{q_i},h_{q_i})$ is the centre of the cube $q_i$. The number of cubes $q_i$ at height $h$ is at most of order $h^{d/2}$.
Define
\[
p_i:=\PP(\exists (w,h_w) \in q_i: \Pi^{\downarrow}[(w,h_w)] \cap \P\setminus \{(\0,h_\0)\} \cap (\R^d \times [0,s)) = \emptyset).
\]
Then we have
\begin{align*}
\PP(\xi((\0,h_\0),\P \cup \A)\neq 0) &\leq \sum_i p_i \\
&\leq c \int_s^\infty h^{d/2} \exp\bigg(-\int_{\R^d}\int_0^s \one((x,h_x) \in \Pi^{\downarrow}[(\0,h)]) h_x^{\beta} dh_xdx \bigg) dh.
\end{align*}
When $s\leq h \leq 2s$, then $\Pi^{\downarrow}[(\0,h)]$ contains $\Pi^{\downarrow}[(\0,s)]$ and the integral $\int_{\R^d}\int_0^s \one((x,h_x) \in \Pi^{\downarrow}[(\0,h)]) h_x^{\beta} dh_xdx$ has a lower bound of order $s^{d/2}$.

When $h \geq 2s$, then $h-s\geq \frac{1}{2}h$ and
\begin{align*}
    \int_{\R^d}\int_0^s \one((x,h_x) \in \Pi^{\downarrow}[(\0,h)]) h_x^{\beta} dh_xdx &= \kappa_d 2^{d/2} \int_0^s (h-h_x)^{d/2} h_x^\beta dh_x\\
    &\geq \kappa_d h^{d/2} \int_0^s h_x^\beta dh_x\\
    &= \frac{\kappa_d}{\beta+1} h^{d/2}s^{\beta+1} \geq \frac{\kappa_d}{\beta+1} h^{d/2}.
\end{align*}
It follows that
\begin{align*}
    \PP(\xi((\0,h_\0),\P \cup \A)\neq 0) &\leq c\int_s^{2s} h^{d/2} \exp(-cs^{d/2}) dh + c \int_{2s}^\infty h^{d/2} \exp(-ch^{d/2})dh \\
    &\leq c\exp(-cs^{d/2}).
\end{align*}
Hence time localization \eqref{dBLtime} follows, as asserted.

\vskip.3cm 
 
\noindent{\bf $BL$-localization of $\xi$ in space.} To show BL localization in space we follow the proof ideas from \cite{SY08}.
It suffices to show 
\begin{equation}\label{dBLspacetimeBurger}
    \PP \bigg( \xi((z,h_z),{\mathcal{P}} \cup {\mathcal{A}}) \neq
    \xi^{[r]}((z,h_z), {\mathcal{P}} \cup {\mathcal{A}}) \bigg) \leq \psi(r), \quad r \geq 0,
    \end{equation}
 with $\psi$ the exponential function.    
Since the proof relies on probability
  bounds for certain regions being devoid of points of ${\cal P},$   we  assume without loss of generality that ${\cal A} = \emptyset$.

Fix $r_0 \geq 4$  and $h_0 \geq 0$.  For any $r \geq r_0$ consider the event 
  $$
E_r:= \left\{    \xi((\0,h_{\0}),{\cal P})
   \neq \xi^{[r]}( (\0,h_{\0}),{\cal P} ) \right\}.
$$
By stationarity it suffices to show $\PP(E_r)$ decays exponentially fast in $r$. We rewrite $E_r = E_{r,1} \cup E_{r,2}$ where 
 
\begin{align*}
E_{r,1} &=\left\{\xi^{[r]}((\0,h_{\0}),\P)= 1 \text{ and }   \xi((\0,h_{\0}),\P)= 0 \right\} \\
&= \Big\{ \exists (z,h_z) \in \partial\Pi^{\uparrow}[(\0,h_\0)] \cap {\rm Cyl}(\0,r): \Pi^{\downarrow}[(z,h_z)] \cap (\P \setminus \{(\0,h_\0)\}) \cap {\rm Cyl}(\0,r) = \emptyset\\
&\qquad \qquad \text{and } \Pi^{\downarrow}[(z,h_z)] \cap \P \cap {\rm Cyl}^c(\0,r) \neq \emptyset\Big\}
 \end{align*}
  and
\begin{align*}
E_{r,2} &=\Big\{\xi^{[r]}((\0,h_{\0}),\P)= 0 \text{ and }   \xi((\0,h_{\0}),\P)= 1 \Big\} \\
& = \Big\{ \forall (w,h_w) \in \partial \Pi^{\uparrow}[(\0,h_\0)] \cap {\rm Cyl}(\0,r),\ \Pi^{\downarrow}[(w,h_w)] \cap \P \setminus \{(\0,h_\0)\} \cap {\rm Cyl}(\0,r) \neq \emptyset,\\
& \qquad \qquad \text{and } \exists (z,h_z) \in \partial\Pi^{\uparrow}[(\0,h_\0)] \cap {\rm Cyl}^c(\0,r): \Pi^{\downarrow}[(z,h_z)] \cap \P \setminus \{(\0,h_\0)\} = \emptyset.
 \end{align*}
On event $E_{r,1}$ we have $|z|\leq r$ and
   \begin{equation}\label{HUA}
    h_{z} \geq \frac{r^2}{8}  \geq 2.
   \end{equation}
This follows by considering the  two cases  $|z| \geq r \slash 2$ or
$d(z,\partial B_d(\0,r)) \geq r \slash 2$, as in both cases $(z,h_z)$ belongs to an up paraboloid $\Pi^{\uparrow}[(w,h_w)]$ where the point $(w,h_w)$ is such that $|w-z|\geq r/2$. As before, partition $B(\0,r) \cap [\frac{r^2}{8},\infty)$ into cubes $q_i$ of unit volume and it follows that if $(z,h_z) \in q_i$, then a proportion of $\Pi^{\downarrow}[(w_{q_i},h_{q_i})] \cap {\rm Cyl}(\0,r)$ must be devoid of points of $\P\setminus \{(\0,h_\0)\}$. Note that the volume of $\Pi^{\downarrow}[(w_{q_i},h_{q_i})] \cap {\rm Cyl}(\0,r)$ is lower bounded by $ch_{q_i}^{\beta+1}r^d$ for some constant $c>0$. Indeed, the intersection of $\Pi^{\downarrow}[(w_{q_i},h_{q_i})] \cap {\rm Cyl}(\0,r)$ with $\R^d \times \{0\}$ is the set $B(\0,r) \cap B(w_{q_i},\sqrt{2h_{q_i}})$. Since $h_{q_i}\geq r^2/8$ and $w_{q_i}\in B(\0,r)$, this intersection contains a ball of radius $r/4$. It follows that $\Pi^{\downarrow}[(w_{q_i},h_{q_i})] \cap {\rm Cyl}(\0,r)$ contains a cone with base a ball of radius $r/4$ and height $h_{q_i}$, the volume of which is of the order $r^dh_{q_i}^{\beta+1}$.

It follows that
 \begin{align*}
     \PP(E_{r_1}) &\leq \sum_i \PP\left(\exists (z,h_z) \in q_i: \Pi^{\downarrow}[(z,h_z)] \cap (\P \setminus \{(\0,h_\0)\}) \cap {\rm Cyl}(\0,r) = \emptyset\right)\\
     &\leq c \int_{\frac{r^2}{8}}^{\infty} h^{d/2} \exp(-c r^dh^{\beta+1}) dh \\
     &\leq c \exp(-cr^d)
 \end{align*}

To estimate $\PP(E_{r,2})$, note that for $(z,h_{z}) \in
\partial(\Pi^{\uparrow}[(\0,h_{\0})]) \cap {\rm{Cyl}}^c(\0,r)$ we must have
\begin{equation}\label{HU2A}
h_{z} \geq h_{\0} + \frac{r^2}{2} \geq \frac{r^2}{2}.
\end{equation}
Partition $B(\0,r)^c \times [\frac{r^2}{2},\infty)$ into cubes $q_i$ of unit volume as before, and note that if $(z,h_z)\in q_i$, then a proportion of $\Pi^{\downarrow}[(w_{q_i},h_{q_i})]$ is also devoid of points of $\P\setminus\{(\0,h_\0)\}$. This event has probability bounded by $c\exp(-c h_{q_i}^{d/2+\beta+1})$, as explained in the proof of time-localization. The number of cubes at height $h$ is again at most of order $h^{d/2}$ and it follows that
\[
\PP(E_{r,2}) \leq c\int_{\frac{r^2}{2}}^\infty h^{d/2} \exp(-ch^{d/2})dh \leq c \exp(-cr^d).
\]
Since $\PP(E_r) \leq \PP(E_{r,1}) +\PP(E_{r,1})$ where  both $\PP(E_{r,1})$ and $\PP(E_{r,2})$
decay exponentially fast in $r$, this  shows BL-localization in space. 

Corollary~\ref{mainthmXisW} thus establishes the asserted proximity bounds to the normal. \qed

\vskip.3cm
{\bf Acknowledgements.} T.T. was supported by the UK Engineering and Physical Sciences Research Council (EPSRC) grant (EP/T018445/1) and the Research Training Group
“Rigorous Analysis of Complex Random Systems" (RTG 3027) funded by the Deutsche Forschungsgemeinschaft (DFG, German Research Foundation Project Number 524444762). Both of these grants, as well as Lehigh University, provided travel support which facilitated writing this paper and which is gratefully acknowledged.
We also thank C. Bhattacharjee and A. Gusakova
for bringing their preprint \cite{BG} to our attention. 

\printbibliography
\end{document}